\documentclass{mcom-l}
\usepackage{amssymb, amsmath, mathtools, fancyhdr,amsfonts}
\usepackage{graphicx}
\usepackage[table]{xcolor}
\usepackage{booktabs, multirow}
\usepackage{stmaryrd}
\usepackage{amsthm, enumitem}
\usepackage{float}
\usepackage{caption}
\usepackage{subcaption}
\usepackage[colorlinks=true,breaklinks=true,linkcolor=lightblue,citecolor=lightgreen,urlcolor=lightblue]{hyperref}
\allowdisplaybreaks

\newcommand\cA{\mathcal{A}}
\newcommand\cT{\mathcal{T}}

\newcommand\cV{\mathcal{V}}
\newcommand{\cF}{\mathcal{F}}
\newcommand\bH{\mathbf{H}}

\newcommand{\vhc}{V_h^{k,\mathrm{c}}}
\newcommand\vhnc{V_h^{k,\mathrm{nc}}}
\newcommand\tvhnc{\widetilde{V}_h^{k,\mathrm{nc}}}

\newcommand{\qhc}{Q_h^{\ell,\mathrm{c}}}
\newcommand\qhnc{Q_h^{\ell,\mathrm{nc}}}
\newcommand\nn{\boldsymbol{n}}

\newcommand\bu{\boldsymbol{u}}

\newcommand\ba{\boldsymbol{a}}
\newcommand\bt{\boldsymbol{t}}

\newcommand\bv{\boldsymbol{v}}
\newcommand\bx{\boldsymbol{x}}

\newcommand\be{\boldsymbol{e}}

\newcommand\btau{\boldsymbol{\tau}}

\newcommand\pd{\Pi^{\nabla^{2\!}}}
\newcommand\pg{\Pi^\nabla}
\newcommand\p{\mathbb{P}}
\newcommand\M{\mathbb{M}}
\newcommand\e{\mathcal{E}}
\newcommand\V{\mathcal{V}}
\newcommand\pw{\mathrm{pw}}
\newcommand\cero{\boldsymbol{0}}

\newcommand{\dx}{\,\mbox{d}\bx}

\newcommand{\ds}{\,\mbox{d}s}

\newcommand\bdiv{\mathop{\mathbf{div}}\nolimits}

\newcommand{\mean}[1]{\left\{\kern-1.ex\left\{ #1 \right\}\kern-1.ex\right\}}		

\definecolor{lightblue}{rgb}{0.22,0.45,0.70}
\definecolor{lightgreen}{rgb}{0.22,0.50,0.25}
\def\Xint#1{\mathchoice
	{\XXint\displaystyle\textstyle{#1}}%
	{\XXint\textstyle\scriptstyle{#1}}%
	{\XXint\scriptstyle\scriptscriptstyle{#1}}%
	{\XXint\scriptscriptstyle\scriptscriptstyle{#1}}%
	\!\int}
\def\XXint#1#2#3{{\setbox0=\hbox{$#1{#2#3}{\int}$ }
		\vcenter{\hbox{$#2#3$ }}\kern-.6\wd0}}

\def\dashint{\Xint-}

\newtheorem{theorem}{Theorem}[section]
\newtheorem{lemma}[theorem]{Lemma}
\newtheorem{proposition}{Proposition}[section]
\theoremstyle{definition}

\newtheorem{corollary}{Corollary}[section]
\theoremstyle{remark}
\newtheorem{remark}[theorem]{Remark}

\numberwithin{equation}{section}

\begin{document}

\title[Virtual element methods for poroelastic plate models]{Virtual element methods for Biot--Kirchhoff poroelasticity}


\author[Khot]{Rekha Khot}
\address{School of Mathematics, Monash University, 9 Rainforest Walk, Melbourne, Victoria 3800, Australia.}
\curraddr{Inria 75012, Paris, France.}
\email{rekha.khot@inria.fr}
\thanks{This research was supported by the Monash Mathematics Research Fund S05802-3951284; by the National Agency for Research and Development, ANID-Chile through FONDECYT project 1220881 and by project Centro de Modelamiento Matem\'atico (CMM), FB210005, BASAL funds for centers of excellence;   by the Australian Research Council through the Discovery Project grant DP220103160 and the Future Fellowship grant FT220100496; and by the Ministry of Science and Higher Education of the Russian Federation within the framework of state support for the creation and development of World-Class Research Centres ``Digital biodesign and personalised healthcare'' No. 075-15-2022-304.}

\author[Mora]{David Mora}
\address{GIMNAP, Departamento de Matem\'atica, Universidad del B\'io-B\'io, Concepci\'on, Chile;  and CI$^2$MA, Universidad de Concepci\'on, Concepci\'on, Chile.}
\email{dmora@ubiobio.cl}
\author[Ruiz-Baier]{Ricardo Ruiz-Baier}
\address{School of Mathematics, Monash University, 9 Rainforest Walk, Melbourne, Victoria 3800, Australia; and Sechenov First Moscow State Medical University, Moscow, Russia; and Universidad Adventista de Chile, Casilla 7-D, Chill\'an, Chile.}
\email{ricardo.ruizbaier@monash.edu}
  

\date{\today}


\begin{abstract}
This paper analyses conforming and nonconforming virtual element formulations of arbitrary polynomial degrees on general polygonal meshes for the coupling of solid and fluid phases in deformable porous plates. The governing equations consist of one fourth-order equation for the transverse displacement of the middle surface coupled with a second-order equation for the pressure head relative to the solid with mixed boundary conditions. We propose novel  enrichment operators that connect nonconforming virtual element spaces of general degree to continuous Sobolev spaces. These operators satisfy additional orthogonal and best-approximation properties (referred to as conforming companion operators in the context of finite element methods), which play an important role in the nonconforming methods.  This paper proves a priori error estimates in the best-approximation form, and derives residual--based reliable and efficient  a posteriori error estimates in appropriate norms, and shows that these  error bounds are robust with respect to the main model parameters. The computational examples illustrate the numerical behaviour of the suggested virtual element discretisations and confirm the theoretical findings on different polygonal meshes with mixed boundary conditions.
\end{abstract}

\maketitle

\section{Introduction}  
\noindent\textbf{Scope.} Fluid-saturated porous media that deform are an essential ingredient in many engineering, biophysical and environmental applications.
From these materials, a family featuring interesting properties is compressible thin plates. Porosity and permeability characteristics through the thickness can be averaged, leading to a different scaling
of poromechanical properties from the typical structure exhibited in Biot's consolidation systems (see, for example \cite[Chapter 8]{cederbaum00}). 

A number of works have addressed the rigorous derivation of poroelastic plate effective equations \cite{marciniak15,mikelic19,serpilli19,taber92,taber96}. The well-posedness analysis has been conducted, for a slightly different model, in the recent paper \cite{gurvich22}. Regarding numerical methods, a discontinuous Galerkin formulation has been proposed in \cite{iliev16} (following \cite{suli07}) and splitting algorithms have been analysed. High-order finite element methods have been used for layer-wise poroelastic shells in \cite{gfrerer20}.

The virtual element method (VEM) has gained popularity in recent years due to its ability to handle complex geometries and provide high accuracy numerical solutions for partial differential equations (PDEs). Another important feature of VEM is the possibility of easily implementing highly regular discrete spaces. This idea is initiated in \cite{brezzi13}, where spaces of high global regularity (such as $C^1$, $C^2$ or more) are easily built in a very efficient way. This has been applied and tested in some biharmonic models of thin plates. The literature contains error analysis of VEM for biharmonic problems for thin plates models (Kirchhoff plates), with a particular emphasis on conforming and nonconforming approximations, including eigenvalue problems \cite{MR4512466,antonietti18,brezzi13,carstensen22b,chinosi16,MR3875292,mora20}. Other VE  discretisations for biharmonic problems in plate models provide a detailed error analysis of the particular type of method and demonstrate its effectiveness through numerical experiments, and the analysis also includes a posteriori error estimates, the time dependent case, and the extension to 3D, among others. See for example \cite{AMNJSC2022,beirao20,MR3854054,chen2022,zhao18,zhao23}.

Since nonconforming spaces might not be a subset of the continuous space, the enrichment (averaging) operator introduced 
in \cite{brenner93} is adapted to prove a priori error estimates utilising a posteriori error bounds in \cite{gudi10}
(known as medius analysis).  These averaging operators essentially assist in connecting the nonconforming discrete space to the continuous space with the required approximation properties. Such operators are then modified in \cite{carstensen2013computational} (referred to as conforming companion operators) satisfying, in addition, orthogonality properties and best approximation estimates. A similar strategy can be carried out in VEM. Indeed, enrichment operators for VEM were recently proposed in \cite{huang21}, and companion operators for VEM were analysed in \cite{carstensen22,carstensen22b}.  Since VE functions are not known explicitly, the computable  or non-computable enrichment/companion operators acting on VE functions can be constructed depending on the specific purpose.  In particular, computable conforming companion operators  can be exploited to define the discrete problem and allowing rough sources \cite{carstensen22b}. In turn, the non-computable ones can be used for purposes of the analysis  (not necessarily knowing the structure explicitly).  So far, the design of such computable companion operators requires a shape-regular sub-triangulation of the polygonal decomposition together with  lowest-order finite element spaces such as Crouzeix--Raviart for second-order and Morley for fourth--order problems. One of the primal advantages of \textit{computable} companion maps is that they permit to approximate rough sources in $H^{-2}(\Omega)$. In this paper we consider sources in $L^2(\Omega)$, and hence we only need to present the construction of non-computable companion operators. This is enough for the required analysis and most importantly, we can avoid to go down on the sub-triangular mesh as done in the construction of computable companion operators.  We stress that the proposed enrichment operator  maps nonconforming VE spaces to conforming VE spaces of one degree higher (which is different from the construction in \cite{huang21}), and in addition, it satisfies $H^2$-orthogonality and best approximation estimates.  We then modify this enrichment operator through a variety of bubble-functions to design  new companion operators having a lower-order orthogonality property (in $H^1$ and $L^2$).  The treatment of general boundary conditions is carefully addressed in this paper, for which  the definition and thorough analysis of the new companion operators is key to establish well-posedness and to obtain error estimates.

This paper presents an extension of nonconforming VE formulations for the coupling of biharmonic problems and second-order elliptic equations (see the similar methods introduced for each of the component sub-problems of second-order linear elliptic and fourth-order elliptic in the recent contributions \cite{carstensen22,carstensen22b}, respectively). The model encodes the interaction with a fluid phase, and the study of this type of problems has gained significant attention due to its relevance in various physical applications. More generally, the proposed framework offers a unified approach to solve coupled problems with mixed boundary conditions on polygonal domains, even when they are non-convex. For conforming cases, we combine $C^1-C^0$ types of VEMs with various polynomial degrees. The error estimates, measured in the weighted ${H}^2\times{H}^1$ energy norm (for deflection and moment pressure), demonstrate robustness with respect to material parameters. Additionally, we introduce a reliable and efficient a posteriori error estimator of residual type. Leveraging the flexibility of VEMs in utilising polygonal meshes, we employ the error estimator to drive an adaptive scheme. Notably, the proposed a posteriori analysis is novel for high-order nonconforming VEMs and can be applied to tackle more complex coupled problems: we emphasise that the  models presented in this work can serve as a fundamental building block for establishing a comprehensive framework for complex mixed-dimensional poroelastic models. These models can also be extended to incorporate interaction with multi-layered structures, such as thermostats and micro-actuators, offering broad applicability and versatility.

The main contributions of this work can be summarised as follows:
\begin{itemize}
\item Application of the proposed conforming and nonconforming VEMs to the plate Biot equations.
\item Design of new companion operators with the orthogonal properties and the best-approximation estimates.
\item A priori error estimates  in the energy norm for both conforming and nonconforming formulations in the best-approximation form that remain robust with respect to material parameters.
\item The detailed proofs of the inverse estimate and the norm equivalence for the nonconforming VE functions.
\item Introduction and analysis of a residual-based a posteriori error estimator.
\item Presentation of numerical results validating the theoretical estimates and demonstrating the competitive performance of the proposed schemes.
\end{itemize}
\medskip \noindent\textbf{Content and structure.} The remainder of the paper has been organised in the following manner. In the rest of this section we provide preliminary notational conventions and definitions to be used throughout the paper. Section~\ref{sec:model} contains the model description and defines the weak formulation of the governing equations. The  local and global VE spaces, the degrees of freedom and the computable polynomial projection operators are addressed in Subsection~\ref{subsec:3.1}, and the derivations for both conforming and nonconforming approximations and  the analysis of existence and uniqueness of discrete solution are conducted in Subsection~\ref{subsec:3.2}. The a priori error analysis for the conforming VE methods in the best-approximation form is carried out in Section~\ref{sec:error-c}. For the nonconforming case,  the companion operators are defined in Subsection~\ref{subsec:companion} along with the proofs of the properties and the best approximation estimates followed by the a priori error estimates in Subsection~\ref{subsec:5.2}. Subsections~\ref{subsec:pre} recalls the preliminary estimates and \ref{subsec:std_est} contains the detailed proofs of  standard estimates such as the inverse estimate and the norm  equivalence for the nonconforming VE functions, and a Poincar\'e-type inequality for $H^2$ functions.  The reliability and efficiency of an a posteriori error estimator are included in Subsection~\ref{subsec:rel}-\ref{subsec:eff}. Finally, a collection of illustrative numerical tests is presented in Section~\ref{sec:results}. 
	
\medskip \noindent\textbf{Recurrent notation and domain configuration.} 
Consider an open, bounded, connected Lipschitz domain of $\mathbb{R}^3$, denoted $\widehat{\Omega}=\Omega \times (-\zeta,\zeta) \subset \mathbb{R}^3$ and occupied by an undeformed thin poroelastic plate (a deformable solid matrix or an array of solid particles) of characteristic thickness $2\zeta$, and where $\Omega\subset \mathbb{R}^2$ represents the mid-surface of the undeformed poroelastic plate. The plate is assumed to be isotropic in the plate plane and to follow the Kirchhoff law. In particular, it is assumed that the plate fibers remain orthogonal to the deflected mid-surface \cite{ciarlet22}. An appropriate modification of Biot constitutive poroelasticity equations is adopted in combination with Darcy flow in deforming pores (see \cite{li97}). Following the model presented in \cite{iliev16}, we assume that the equations governing the balance of momentum and mass of the solid and fluid phases can be written in terms of the averaged-through-thickness deflection $u$ (vertical displacement of the solid phase) and the first moment of the pressure of the fluid phase $p$. We will denote by $\nn$ the unit normal vector on the undeformed boundary $\partial\Omega$. The boundary $\partial\Omega$  is disjointly split between a closed set $\Gamma^{c}$ and an open set $\Gamma^{s}$ where we impose, respectively, homogeneous deflections and homogeneous normal derivatives of deflections and of pressure moment (clamped sub-boundary with zero-flux) and homogeneous pressures with normal deflections and bending moments (simply supported sub-boundary). 

For a subdomain $S\subseteq \Omega$ we will adopt the notation $(\cdot,\cdot)_{m,S}$ for the inner product, and  $\|\cdot\|_{m,S}$ (resp. $|\cdot|_{m,S}$) for the norm (resp. seminorm) in the Sobolev space $H^m(S)$ (or in its vector counterpart $\bH^m(S)$) with $m\ge 0$. We sometimes drop $0$ from the subscript in $L^2$ inner product and norms  for convenience. In view of the boundary conditions mentioned above, we also define the Sobolev spaces $H^2_{\Gamma^c}(\Omega):=\{v\in H^2(\Omega): v=\partial_{\nn}v=0\;\text{on}\;\Gamma^c\}$, where $\partial_{\nn}v$ denotes the normal derivative of $v$, and  $H^1_{\Gamma^s}(\Omega):=\{q\in H^1(\Omega): q=0\;\text{on}\;\Gamma^s\}$.
Also, given an integer $k\ge1$ and $S\subset \mathbb{R}^d$, $d=1,2$, by $\mathbb{P}_k(S)$ we will denote the
space of polynomial functions defined locally in $S$ and being of total degree up to $k$. Given a barycentre $x_{S}$ and diameter $h_{S}$ of a  domain $S$, we define the set of scaled monomials $\M_k(S)$ of total degree up to $k$ and $\M_k^*(S)$ of degree equal to $k$ by
$$\M_k(S)= \Big\lbrace \Big(\frac{x-x_{S}}{h_{S}} \Big)^\ell: |\ell| \le k \Big\rbrace, \text{and} \M^*_k(S)= \Big\lbrace \Big(\frac{x-x_{S}}{h_{S}} \Big)^\ell: |\ell| = k \Big\rbrace.$$

Throughout the paper we use $C$ to denote a generic positive constant independent of the mesh size  $h$ and of the main model parameters, that might take different values at its different occurrences. Moreover, given any positive expressions $X$ and $Y$, the notation $X \,\lesssim\, Y$  means that $X \,\le\, C\, Y$ (similarly for $X\,\gtrsim\,Y$).
\section{Plate Biot equations and solvability analysis}\label{sec:model}
The two-dimensional poroelasticity problem arising when a fluid flows through a deformable porous plate of thickness $d$ can be written in the following form 
\begin{subequations}\label{org-problem}
\begin{align}
\hat{\rho} \frac{\partial^2u}{\partial t^2} + \text{div}\bdiv \mathbb{A}\nabla^2 u + \alpha b\Delta p & = f & \text{in $\Omega\times(0,T]$},\\
\bigg(c_0+\frac{\alpha^2}{\lambda+2\mu}\bigg) \frac{\partial p}{\partial t} - \alpha b \frac{d^3}{12} \frac{\partial(\Delta u)}{\partial t} - \frac{\kappa}{\eta} \Delta p & = g & \text{in $\Omega\times(0,T]$},\\
u = \partial_{\nn}u = \partial_{\nn}p & = 0 & \text{on $\Gamma^c\times(0,T]$},\\
\label{eq:bc2}
u = \partial_{\nn\nn}u = p &= 0 & \text{on $\Gamma^s\times(0,T]$},
\end{align} 
\end{subequations}
with appropriate initial conditions and where $\hat{\rho}$ is the reduced  density (taking into account a volume-to-surface scaling), $\mathbb{A}:= D((1-\nu)\mathbb{I}+\nu \mathbb{I}\otimes\mathbb{I})$
for the flexural rigidity $D:=\frac{Ed^3}{12(1-\nu^2)}=\frac{\mu(\lambda+\mu)d^3}{3(\lambda+2\mu)}$ of the plate for the Young modulus $E$, the Poisson ratio $\nu\in(0,0.5)$, the Lam\'e parameters $\lambda,\mu$ and $\partial_{\nn\nn}u:=(\nabla^2 u)\nn\cdot\nn$. Also  $\alpha$ is the Biot--Willis poroelastic coefficient,  $b=2\mu(\lambda+2\mu)^{-1}$, $c_0$ is the total storage capacity,  $\kappa$ is the absolute permeability, and $\eta$ is the viscosity of the fluid. 
We concentrate our attention in the following modification of this Biot--Kirchhoff system in the case of a specific adimensionalisation proposed in \cite[eq. (7)-(9)]{iliev16}:  
\begin{subequations}\label{problem}
\begin{align}\label{eq:momentum}
\frac{\partial^2u}{\partial t^2} + \Delta^2 u + \alpha\Delta p & = f & \text{in $\Omega\times(0,T]$},\\ 
\label{eq:mass}
\beta \frac{\partial p}{\partial t} - \alpha \frac{\partial(\Delta u)}{\partial t} - \gamma \Delta p & = g & \text{in $\Omega\times(0,T]$},
\end{align}\end{subequations}
 where with an abuse of notation we have kept the same notation for the dimensionless problem. Here $f \in {L}^2(0,T;\Omega)$ is the normal vertical loading and $g \in {L}^2(0,T;\Omega)$ is a prescribed mass source/sink. Henceforth, we treat $\alpha, \beta,\gamma$ as  the main model parameters, where $\alpha\leq 1 \leq \gamma$ and \[ \beta = \bigl(c_0[\lambda + 2\mu] + \alpha^2\bigr)\gamma, \quad \gamma = \frac{\lambda+\mu}{\mu}.\] Note that the reduced density $\hat{\rho}$ and the rigidity $D$ of the plate are absorbed in the load $f$ and the model parameter $\alpha$. In plate problems, the Poisson ratio $\nu$ tending to $0.5$ is one of the interesting  limiting phenomena and that, in turn, leads to the case of the first Lame parameter $\lambda\to \infty$. In our new normalised problem, it essentially implies that the model parameters $\beta,\gamma$ will go to infinity. The aim in this paper is to analyse the normalised problem \eqref{problem} and show that its robust with respect to the model parameters $\alpha,\beta,\gamma$.

 System \eqref{problem} is similar to the non-inertial problem in \cite{iliev16} which accommodates fluid-saturated plates where diffusion is possible in the in-plane direction (see also the set of problems recently analysed in \cite{gurvich22}), here extended to the case of mixed boundary conditions. 
In order to fix ideas, we will focus first on a simplified system, resulting from applying a centred and backward Euler semi-discretisation in time to \eqref{eq:momentum}-\eqref{eq:mass}, with a conveniently rescaled final time $T$ and rescaled time step to $\Delta t=1$. Define  the displacement and pressure space
	\[V:= H^2_{\Gamma^c}(\Omega)\cap H^1_{0}(\Omega),\quad Q:= H^1_{\Gamma^s}(\Omega).\]
Owing to the   specification of boundary conditions (taken homogeneous for sake of simplicity of the presentation),  a weak formulation is obtained, which reads:
Find $(u, p)\in V\times Q$ such that \begin{subequations} \label{eq:weak}
	\begin{align}
		(u,v)_{\Omega} +  (\nabla^2 u,\nabla^2 v)_{\Omega} - \alpha(\nabla p, \nabla v)_{\Omega}  
		&= (\tilde f,v)_{\Omega}
		\qquad \forall\; v  \in V, \label{eq:weak:a}\\
		\beta(p,q)_{\Omega} + \alpha (\nabla q, \nabla u)_{\Omega} + \gamma (\nabla p, \nabla q)_{\Omega} 
		&= (\tilde g,q)_{\Omega} \qquad \forall\; q  \in Q,\label{eq:weak:b}
	\end{align}
\end{subequations}
with  $\nabla^2v : = \begin{pmatrix}
v_{xx} & v_{xy}\\v_{yx}&v_{yy}
\end{pmatrix}$ being the Hessian matrix (of second-order derivatives) for a given $v\in H^2(\Omega)$. The right-hand side terms also include the value of deflection and pressure moments in the previous backward Euler time steps, denoted as $\widehat{u}^n,\widehat{u}^{n-1}$ and $\widehat{p}^n$, respectively:
\[ \tilde f = f  + 2 \widehat{u}^n - \widehat{u}^{n-1}, \qquad \tilde g = g + \widehat{p}^n,\]
where the index $n\geq 0$ indicates the time step.

The product space $\bH_\epsilon$ 
contains all $\vec \bu\in [H^2_{\Gamma^c}(\Omega)\cap H^1_{0}(\Omega)] \times H^1_{\Gamma^s}(\Omega)$
which are bounded in the norm   
\begin{equation}\label{eq:product-norm}
	\| \vec{\bu} \|_{\bH_\epsilon}^2 : = 
	\| u \|_{\Omega}^2+| u |_{2,\Omega}^2
	+ \beta \|{p} \|_{\Omega}^2
	+ \gamma| {p} |_{1,\Omega}^2.
\end{equation}
The subscript $\epsilon$ denotes the weighting parameters  (in our case, $\beta,\gamma$). Let us now group the trial and test fields as $\vec{\bu}= (u,p)$ and $\vec{\bv} = (v,q)$, respectively; and introduce the operator  $\cA: \bH_\epsilon \to \bH_\epsilon$ defined as 
\begin{align*}
\langle \cA(\vec{\bu}),\vec{\bv}  \rangle&:=
(u,v)_{\Omega} +  (\nabla^2u ,\nabla^2 v)_{\Omega} - \alpha(\nabla p, \nabla v)_{\Omega}  +  \beta(p,q)_{\Omega} + \alpha (\nabla q, \nabla u)_{\Omega} \\&\quad+ \gamma (\nabla p, \nabla q)_{\Omega},
\end{align*}
where $\langle \cdot,\cdot\rangle$ denotes the duality pairing between $\bH_\epsilon$ and $\bH_\epsilon'$.  Note that $|u|_{2,\Omega}$ defines a norm on $V$, which is equivalent to $H^2$-norm \cite[pp.~34]{ciarlet02}. In particular, this implies for any $v\in V$ that
\begin{align}
\|v\|_{2,\Omega}\lesssim |v|_{2,\Omega}.\label{norm-equiv}
\end{align}
We also define the linear and bounded operator $\cF:\bH_\epsilon \to \mathbb{R}$ as
\[
\vec{\bv} \mapsto \cF(\vec{\bv})  :=  (\tilde f,v)_{\Omega} +  (\tilde g,q)_{\Omega},
\]
and therefore Problem \eqref{eq:weak} is recast as: Find $\vec{\bu} \in \bH_\epsilon$ such that 
\begin{equation}\label{eq:operator}
\langle \cA(\vec{\bu}),\vec{\bv} \rangle = \cF(\vec{\bv}) \qquad \forall\; \vec{\bv} \in \bH_\epsilon.
\end{equation}

We are now in a position to state the solvability of the continuous problem~\eqref{eq:operator}.

\begin{theorem}\label{theo:iso}
Problem \eqref{eq:operator} is well-posed in the space $\bH_\epsilon$ equipped with the norm \eqref{eq:product-norm}. 
\end{theorem} 
\begin{proof}
It follows from the Lax--Milgram lemma (see, e.g., \cite[Lemma 25.2]{ern04}), requiring the boundedness of $\cA$ over the space $\bH_\epsilon$
\[ \langle \cA(\vec{\bu}),\vec{\bv}\rangle \lesssim  \|\vec{\bu} \|_{\bH_\epsilon} \|\vec{\bv} \|_{\bH_\epsilon}\qquad \forall\;
\vec{\bu},\vec{\bv} \in \bH_\epsilon,\]
and the boundedness of $\cF$, 
as well as the coercivity condition
\[\langle \cA(\vec{\bu}),\vec{\bu} \rangle 
= \|\vec{\bu} \|^2_{\bH_\epsilon} \qquad  
\forall\; \vec{\bu} \in \bH_\epsilon.\]
For the continuity it suffices to apply the Cauchy--Schwarz  inequality while the coercivity is a direct consequence of the definition of the solution operator (whose off-diagonal terms cancel out). 
\end{proof}
Now, we state an additional regularity result for the solution of problem~\eqref{eq:operator}.

\noindent	\textbf{Regularity estimates \cite{grisvard1980boundary}}. Given  $\tilde{f}\in H^{s-4}(\Omega)$ and $\tilde{g}\in H^{r-2}(\Omega)$ with $s\geq 2$ and $r\geq 1$, there exists a unique solution $\vec{\bu}=(u,p)\in (H^{s}(\Omega)\cap V)\times (H^{r}(\Omega)\cap Q)$ to \eqref{eq:operator}  such that
\begin{align}
\|u\|_{s,\Omega}+\|p\|_{r,\Omega}\lesssim \|\tilde{f}\|_{s-4,\Omega}+ \|\tilde{g}\|_{r-2,\Omega}.
\end{align}
\begin{remark}[Simply supported boundary condition]
The boundary condition $u=\partial_{\nn\nn}u=0$ on the simply supported part and an integration by parts 
\[(\Delta^2 u,v)_\Omega = (\nabla^2 u,\nabla^2 v)_\Omega+(\partial_{\nn}(\Delta u),v)_{\partial\Omega}-(\partial_{\nn\nn}u,\partial_{\nn}v)_{\partial\Omega}-(\partial_{\nn\bt}u,\partial_{\bt}v)_{\partial\Omega}\]
for $v\in V$ allow the hessian term $ (\nabla^2u,\nabla^2v)_\Omega=(\Delta^2 u,v)_\Omega$  in the weak form. However to be consistent with plate mechanics,  a zero bending moment $M_{\nn\nn}(u):=\nu\Delta u+(1-\nu)\partial_{\nn\nn}u = 0$ should be prescribed in place of $\partial_{\nn\nn}u=0$ and instead one can consider the integration by parts 
\begin{align*}
\widetilde{a}(u,v)&:=\int_\Omega \nu \Delta u\Delta v +(1-\nu)u_{ij}v_{ij} \\&= \int_\Omega \Delta^2u v+\int_{\partial\Omega}M_{\nn\nn}(u)\partial_{\nn}v-\int_{\partial\Omega}(\partial_{\nn}(\Delta u)v+(1-\nu)\partial_{\nn\bt}u\partial_{\bt}v).
\end{align*}
This replaces the hessian term by the above bilinear form $\widetilde{a}(u,v)$ and the  weak formulation becomes
\begin{align}
(u,v)_{\Omega} +  \widetilde{a}(u,v) - \alpha(\nabla p, \nabla v)_{\Omega}  
&= (\tilde f,v)_{\Omega}\qquad \forall\; v  \in V, \\\beta(p,q)_{\Omega} + \alpha (\nabla q, \nabla u)_{\Omega} + \gamma (\nabla p, \nabla q)_{\Omega} 
&= (\tilde g,q)_{\Omega} \qquad \forall\; q  \in Q.
\end{align}
The boundedness of modified $\cA$ is straightforward and the coercivity follows from the observation $\widetilde{a}(v,v)= (\nu\|\Delta v\|^2_\Omega+(1-\nu)|v|^2_{2,\Omega})\gtrsim |v|^2_{2,\Omega}$ with the equivalence \eqref{norm-equiv} in the last inequality.
\end{remark}
\section{Virtual element formulation and unique solvability of the discrete problem}\label{sec:VEM} 
Let us denote by $\{\cT_{h}\}_{h>0}$ a shape-regular family of partitions of 
$\bar\Omega$, conformed by polygons $K$ of  diameter $h_K$, and we denote the mesh size by 
$h:=\max\{h_K: K\in\cT_{h}\}$. Let $\V=\V^i\cup\V^c\cup\V^s$ and $\e =\e^i\cup\e^c\cup\e^s$ be the set of  interior vertices $\V^i$ and boundary vertices $\V^c\cup\V^s$, and the set of  interior edges $\e^i$ and boundary edges $\e^c\cup\e^s$.  By $N_K$ we will denote the number of vertices/edges in the generic polygon $K$. For all edges $e\in \partial K$, we denote by $\nn_K^e$
the unit normal pointing outwards $K$, $\bt^e_K$ the unit tangent vector along $e$
on $K$, and $V_i$ represents the $i^{th}$ vertex of the polygon $K$. We suppose that there exists a universal positive constant $\rho$ such that
\begin{enumerate}[label=(\bfseries{M\arabic{*}})] \item\label{M1} every polygon $K\in \cT_h$ of diameter $h_K$ is star-shaped with respect to every point of a ball of radius greater than or equal to $\rho h_K$,
\item\label{M2} every edge $e$ of $K$ has a length $h_e$ greater than or equal to $\rho h_K$. 
\end{enumerate}
Throughout this section we will construct and analyse a conforming and a nonconforming family of VE methods.
\subsection{Virtual element spaces}\label{subsec:3.1}
\textbf{VE spaces for displacement approximation}. {First we} define the bilinear form $a^K$  as the restriction to $K$ of  
\[a(v,w):= \int_\Omega \nabla^2v:\nabla^2w\dx.\] 
For $K\in\cT_h$ and $k\geq 2$, define the projection operator $\pd_k:H^2(K)\to \p_k(K)$, for $v\in H^2(K)$, by 
\begin{align}
a^K(\pd_k v,\chi_k) = a^K(v,\chi_k)&\qquad\forall \chi_k\in \p_k(K),\label{pd1}
\end{align}
with the additional conditions
\begin{subequations}
	\begin{align}
		\overline{\pd_kv} = \overline{v} \quad&\text{and}\quad \overline{\nabla\pd_kv} = \overline{\nabla v} &\text{for conforming VEM}, \label{pd2a}\\\overline{\pd_kv} = \overline{v} \quad&\text{and}\quad \int_{\partial K}\nabla\pd_kv \ds = \int_{\partial K}\nabla v\ds & \text{for nonconforming VEM}, \label{pd2b}&
	\end{align}
\end{subequations}
where    $\overline{v}$ is the average $\frac{1}{N_K}\sum_{i=1}^{N_K}v(V_i)$ of the values of $v$ at the vertices $V_i$ of $K$. Since the linear polynomials $\chi_k\in\p_1(K)\subset\p_k(K)$ lead to the identity $0=0$ in   \eqref{pd1}, it follows that the two conditions in   \eqref{pd2a} for conforming and \eqref{pd2b} for nonconforming fix the affine contribution and  define $\pd_k v$ uniquely for a given $v$. Furthermore, the {Poincar\'e--Friedrichs} inequality implies 
\begin{align}\|v-\pd_kv\|_K\lesssim h_K|v-\pd_kv|_{1,K}\lesssim h_K^2|v-\pd_kv|_{2,K}.\label{PF}\end{align}
\par The local conforming VE space $\vhc(K)$ \cite{brezzi13} is a set of solutions to a biharmonic problem over $K$ with clamped boundary conditions on $\partial K$, and  it is defined, for $k\geq 2$ and $r=\max\{k,3\}$, as 
\[\vhc(K):= \begin{dcases}
	\begin{rcases}
		&v_h \in H^2(K)\cap C^1(\partial K) : \Delta^2 v_h \in\p_k(K),\; v_h|_e\in \p_r(e)\;\text{and}\\&\;\nabla v_h|_e\cdot\nn^e_K\in\p_{k-1}(e)\quad\forall\;e\in\partial K,\;\text{and}\\&\; (v_h-\pd_kv_h,\chi)_K=0\quad\forall\;\chi\in\p_k(K)\setminus\p_{k-4}(K)
	\end{rcases}.
\end{dcases}\]
\par On the other hand, the local nonconforming VE space  is a set of solutions to a biharmonic problem with simply supported boundary conditions and was first introduced in \cite{zhao18}. However, it is pointed out in \cite{carstensen22b} that the definition in \cite{zhao18} works for a polygon $K$ without hanging nodes, and that new work provides an alternative definition for the lowest-order case ($k=2$) with possibly hanging nodes in $K$. In this paper, we extend such a definition of the nonconforming VE space for a general degree $k$. First we need  some preliminary geometrical notations. Let $K\in\cT_h$ be a polygonal element, and $\e_K:=\{e_1,\dots,e_{N_K}\}$ and $V_1,\dots,V_{N_K}$ be the edges and vertices of $K$. Suppose that $z_1,\dots,z_{\tilde{N}_K}$ denote the corner points of $K$ for some $\tilde{N}_K\leq N_K$, where the angle at each $z_j$ is different from $0,\pi,2\pi$.    The boundary $\partial K =e_1\cup\dots\cup e_{N_K}$ can also be viewed as a union of the sides $s_1,\dots,s_{\tilde{N}_K}$, where $s_j:=\text{conv}\{z_{j},z_{j+1}\}$ (convex hull of $z_j$ and $z_{j+1}$) for $z_j = V_{m_j}$ and  $z_{j+1}=V_{m_j+n_j}$ with $z_{\tilde{N}_K+1}=z_1$. Note that $m_j$ is the index in the vertex numbering corresponding to the $j^\text{th}$ index of the corners of $K$ and $n_j$ is the total number of straight edges on the sides $s_j$.   See a sketch in Figure~\ref{fig:P} (e.g., $m_1=1$ and  $n_1=2$, $m_2=3$ and $n_2=1$, and so on.)
\begin{figure}[!t]
\centering
\includegraphics[width=0.4\linewidth]{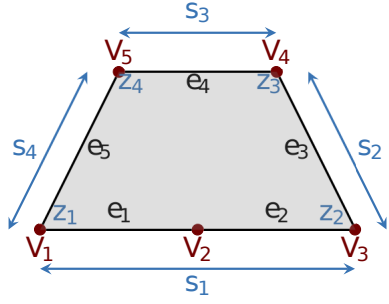}              
\caption{{Sample of pentagonal element} with vertices $V_1,\dots,V_5$, edges $e_1,\dots,e_5$, corners $z_1,\dots,z_4$, and sides $s_1,\dots,s_4$. }
\label{fig:P}
\end{figure}
With these  notations, we are in a position to define the local nonconforming VE space $\vhnc(K)$ for $k\geq 2$ by
\[
\vhnc(K):=\begin{dcases}
\begin{rcases}
	&v_h\in H^2(K): \Delta^2v_h\in \p_k(K),\;v_h|_e\in\p_k(e)\;\text{and}\;\Delta v_h|_e\in\p_{k-2}(e)\\&\quad\forall\;e\in\e_K,\;v_h|_{s_j}\in C^1(s_j),\;\int_{e_{m_j}}v_h\chi\ds=\int_{e_{m_j}}\pd_k v_h\chi\ds\quad\\&\quad\forall\;\chi\in\p_{k-2}(e_{m_j}),\;\text{and}\;\int_{e_{m_j+i}}v_h\chi\ds=\int_{e_{m_j+i}}\pd_k v_h\chi\ds\\&\quad\forall\;\chi\in\p_{k-3}(e_{m_j+i})\;\text{for}\;i=1,\dots,n_j,\;\text{and}\;j=1,\dots,\tilde{N}_K,\\&\quad(v_h-\pd_k v_h,\chi)_K=0\quad\forall\;\chi\in\p_k(K)\setminus\p_{k-4}(K)
\end{rcases}.
\end{dcases}
\]
The local degrees of freedom (DoFs) for both conforming and nonconforming VE spaces are summarised in Table~\ref{table1}. 

\begin{table}[t!]
\centering
\begin{tabular}{|l|l|l|}
	\hline
	\hline
	degree & DoFs of $v_h\in\vhc(K)$   & DoFs of $v_h\in\vhnc(K)$ \vphantom{$\displaystyle{\int_K^K}$} \\
	\hline
	\hline
	$k\geq 2$ & 
	($\mathbb{D}$1)  $v_h(V_i)\quad\forall\;i=1,\dots,N_K$ & ($\mathbb{D}$1$^{\star}$)  $v_h(V_i)\quad\forall\;i=1,\dots,N_K$\vphantom{$\displaystyle{\int_K^K}$}\\
	&($\mathbb{D}$2) $h_{V_i}\nabla v_h (V_i)\quad\forall\;i=1,\dots,N_K$ & ($\mathbb{D}$2$^{\star}$) $\int_e\partial_{\nn} v_h\chi\ds\quad\forall\;\chi\in\M_{k-2}(e)$\vphantom{$\displaystyle{\int_K^K}$}\\[1ex]
	\hline
	$k\geq 3$ & 
	($\mathbb{D}$3) $\int_e\partial_{\nn} v_h\chi\ds\quad\forall\;\chi\in\M_{k-3}(e)$ &  ($\mathbb{D}$3$^{\star}$) $\dashint_e v_h\chi\ds\quad\forall\;\chi\in\M_{k-3}(e)$\vphantom{$\displaystyle{\int_K^K}$} \\[1ex]
	\hline
	$k\geq 4$ & 
	($\mathbb{D}$4) $\dashint_e v_h\chi\ds\quad\forall\;\chi\in\M_{k-4}(e)$ & ($\mathbb{D}$4$^{\star}$) $\dashint_K v_h\chi\dx\quad\forall\;\chi\in\M_{k-4}(K)$\vphantom{$\displaystyle{\int_K^K}$} \\[1ex]
	&($\mathbb{D}$5) $\dashint_K v_h\chi\dx\quad\forall\;\chi\in\M_{k-4}(K)$ &\\[1ex]
	\hline
\end{tabular}
\caption{The left panel describes the DoFs of $\vhc(K)$ with the characteristic length (see \cite{brezzi13}, for example) $h_{V_i}$ associated with each vertex $V_i$  for all $i=1,\dots,N_K$ and the right column lists the DoFs of $\vhnc(K)$. The edge $e$ belongs to the set $\e_K$ of all edges  of $K$.}
\label{table1}
\end{table}
It can be shown that the triplets $(K,\vhc(K),\{(\mathbb{D}1)-(\mathbb{D}5)\})$ and $(K,\vhnc(K),\{(\mathbb{D}1^{\star})-(\mathbb{D}4^{\star})\})$ form a finite element in the sense of Ciarlet \cite{ciarlet02}, and the projection operator $\pd_k v_h$ for $v_h\in \vhc(K)$ (resp. $v_h\in\vhc(K)$) is computable in terms of the DoFs ($\mathbb{D}$1)-($\mathbb{D}$5) (resp. ($\mathbb{D}$1$^{\star}$)-($\mathbb{D}$4$^\star$)). We refer to \cite{brezzi13} (resp. \cite{carstensen22b}) for a proof.
\par Let $\Pi_k$ denote the $L^2$-projection onto the polynomial space $\p_k(K)$. That is, 
\[(\Pi_k v, \chi)_K = (v,\chi)_K\qquad\forall \;\chi\in\p_k(K).
\]
The orthogonality condition in the definition of the local VE spaces $\vhc(K)$ and $\vhnc(K)$ implies that $\Pi_k$ is also computable in terms of the DoFs.
\par For $v\in H^1(K)$ and $\vec{\chi}\in(\p_{k-1}(K))^2$, an integration by parts leads to the expressions 
\begin{align}(\Pi_{k-1}\nabla v, \vec{\chi})_K=-( v, \bdiv\vec{\chi})_K+(v,\chi\cdot\nn^e_K)_{\partial K}=-(\Pi_k v, \bdiv\vec{\chi})_K+(v,\chi\cdot\nn^e_K)_{\partial K},\label{L2g}
\end{align}
owing to the definition of $\Pi_k$ in the last step. Observe that the DoFs ($\mathbb{D}$1)-($\mathbb{D}$2) and ($\mathbb{D}$4) determine  $v_h\in\p_r(e)$ explicitly for all $e\in\partial K$. This and the computability of $\Pi_k$ imply that $\Pi_{k-1}\nabla v_h$ for $v_h\in \vhc(K)$ is computable in terms of the DoFs. Since $\pd_k v_h$ is computable, the values  $\int_{e_{m_j}}v_h\chi\ds$ for $\chi\in\M_{k-2}(e_{m_j})$ are computable from the definition of $\vhnc(K)$. If $n_j=0$, these $(k-1)$ estimates, and  the values at the vertices $V_{m_j}$ and $ V_{m_j+1}$    uniquely determine  $v_h\in\p_k(e_{m_j})$. If $n_j>0$, the point values $v_h(V_{m_j+i}), v_h(V_{m_j+i+1}), \partial_{\mathbf{\tau}} v_h(V_{m_j+i})$  and $\int_{e_{m_j+i}}v_h\chi\ds$ for $\chi\in\M_{k-3}(e_{m_j+i})$  evaluate $v_h$ on each edge $e_{m_j+i}$ for $i=1,\dots,n_j, j=1,\dots,\tilde{N}_K$, and consequently $v_h$ is known on the boundary $\partial K$. Similarly as above, this step and the computability of $\Pi_k$ imply that $\Pi_{k-1}\nabla v_h$   is computable in terms of the DoFs for $v_h\in \vhnc(K)$.
\begin{proposition}[Polynomial approximation \cite{brenner2008}]\label{prop:poly}
Under the assumption \ref{M1}, for every $v\in H^{s}(K)$, there exists $\chi_k\in\p_k(K)$ with $k\in\mathbb{N}_0$ such that
\begin{align*}
|v-\chi_k|_{m,K}\lesssim h_K^{s-m}|v|_{s,K}\quad\text{for}\;0\leq m \leq s\leq k+1.
\end{align*}	
\end{proposition}
\noindent  The global VE spaces $\vhc$  and $\vhnc$ are defined, respectively, as 
\begin{align*}
\vhc:= \{v_h\in V : v_h|_K\in \vhc(K)\quad\forall\;K\in\cT_h\},
\end{align*}
and 
\begin{align*}
\vhnc:=\begin{dcases}
	\begin{rcases}
		&v_h\in L^2(\Omega) :v_h|_K\in \vhnc(K)\quad\forall\;K\in\cT_h,\; v_h\; \text{is continuous at}\\& \text{interior vertices and zero at boundary vertices},\;\int_e[\partial_{\nn}v_h]\chi\ds=0\\&\forall\;\chi\in\p_{k-2}(e),\;e\in\e^i\cup\e^c\;\text{and}\;\int_e [v_h]\chi\ds =0\quad\forall\;\chi\in\p_{k-3}(e),\; e\in\e
	\end{rcases}.
\end{dcases}
\end{align*}
\textbf{VE spaces for pressure approximation}. We define the projection operator $\pg_\ell:H^1(K)\to\p_\ell(K)$ for $\ell\geq 1$ and $q\in H^1(K)$ through the following equation 
\begin{align}
(\nabla \pg_\ell q, \nabla \chi_\ell)_K &= (\nabla q, \nabla \chi_\ell)_K\qquad\forall\;\chi_\ell\in\p_\ell(K),\label{pg1}
\end{align}
with the additional condition needed to fix the constant
\begin{subequations}
\begin{align}
	\overline{\pg_\ell q} &= \overline{q} &\text{for conforming VEM},\label{pg2a}\\
	\int_{\partial K} \pg_\ell q \ds &= \int_{\partial K} q\ds &\text{for nonconforming VEM}. 
\end{align}
\end{subequations}
This defines $\pg_\ell q$ uniquely for a given $q$.  To approximate the pressure space $Q$, we introduce the local conforming  VE space $\qhc(K)$ for $\ell\geq 1$ and $K\in\cT_h$ as the set of solutions to a Poisson problem with Dirichlet boundary conditions \cite{beirao13}. In particular,
\[
\qhc(K):= \begin{dcases}
\begin{rcases}
	&q_h\in H^1(K)\cap C^0(\partial K): \Delta q_h \in \p_\ell(K),\; q_h|_e\in\p_\ell(e)\quad\forall\;e\in\partial K,\\&\qquad\text{and}\; (q_h-\pg_\ell q_h,\chi)_K=0\quad\forall\;\chi\in\p_\ell(K)\setminus\p_{\ell-2}(K)
\end{rcases}.
\end{dcases}
\]
In turn, the local nonconforming VE space $\qhnc(K)$  is the set of solutions to a Poisson problem with Neumann boundary condition \cite{Ayuso16} and is defined for $\ell\geq 1$ as
\[
\qhnc(K):= \begin{dcases}
\begin{rcases}
	&q_h\in H^1(K)\cap C^0(\partial K): \Delta q_h \in \p_\ell(K),\; \partial_{\nn} q_h|_e\in\p_{\ell-1}(e)\\&\quad\forall\;e\in\partial K,\;\text{and}\; (q_h-\pg_\ell q_h,\chi)_K=0\quad\forall\;\chi\in\p_\ell(K)\setminus\p_{\ell-2}(K)
\end{rcases}.
\end{dcases}
\]
The DoFs for $\qhc(K)$ and  $\qhnc(K)$ are provided in Table~\ref{table2}. 

\begin{table}[t!]
\centering
\begin{tabular}{|l|l|l|}
	\hline
	\hline
	degree & DoFs of $q_h\in\qhc(K)$   & DoFs of $q_h\in\qhnc(K)$ \vphantom{$\displaystyle{\int_K^K}$}\\
	\hline
	\hline
	$\ell\geq 1$ & 
	($\mathbb{F}$1)  $q_h(V_i)\quad\forall\;i=1,\dots,N_K$ & ($\mathbb{F}$1$^\star$)  $\dashint_e q_h\chi\ds\quad\forall\chi\in\M_{\ell-1}(e)$\vphantom{$\displaystyle{\int_K^K}$}\\
	\hline
	$\ell\geq 2$ & 
	($\mathbb{F}$2) $\dashint_e q_h\chi\ds\quad\forall\chi\in\M_{\ell-2}(e)$ & ($\mathbb{F}$2$^\star$) $\dashint_K q_h\chi\dx\quad\forall\chi\in\M_{\ell-2}(K)$\vphantom{$\displaystyle{\int_K^K}$}\\
	&($\mathbb{F}$3) $\dashint_K q_h\chi\dx\quad\forall\chi\in\M_{\ell-2}(K)$&\vphantom{$\displaystyle{\int_K^K}$}\\
	\hline
\end{tabular}
\caption{The left (resp. right) panel describes the DoFs of $\qhc(K)$  (resp. $\qhnc(K)$). The edge $e$ belongs to the set $\e_K$ of all edges  of $K$.}
\label{table2}
\end{table}
The triplets $(K,\qhc(K),\{(\mathbb{F}1)-(\mathbb{F}3)\})$ and $(K,\qhnc(K),\{(\mathbb{F}1^\star)-(\mathbb{F}2^\star)\})$ form a finite element in the sense of Ciarlet \cite{ciarlet02} (see, e.g. \cite{beirao13}). Note that  $\pg_\ell q_h$ can be computed from   DoFs of ($\mathbb{F}$1)-($\mathbb{F}$3) (resp. ($\mathbb{F}$1$^\star$)-($\mathbb{F}$2$^\star$)) for $q_h \in\qhc(K)$ (resp. $q_h\in \qhnc(K)$). Refer to \cite{beirao13} (resp. \cite{carstensen22}) for a proof. Consequently,   the $L^2$-projection $\Pi_\ell$ is also  computable  from  the orthogonality condition in the definition of the spaces $\qhc(K)$ and $\qhnc(K)$. This and the explicit expression of $q_h$ on the boundary $\partial K$ in \eqref{L2g} show that $\Pi_{\ell-1}\nabla q_h$ is computable for $q_h\in\qhc(K)$. The computability of $\Pi_\ell$ and ($\mathbb{F}$1$^\star$) in \eqref{L2g} imply that of $\Pi_{\ell-1}\nabla q_h$  for $q_h\in\qhnc(K)$.
\par Next we define the global VE spaces for conforming and nonconforming pressure approximation, for $\ell\geq 1$, as 
\[\qhc:=\{q_h\in Q: q_h|_K\in Q_h^\ell(K)\quad\forall\;K\in\cT_h\},\]
and 
\begin{align*}
\qhnc:=\begin{dcases}
	\begin{rcases}
		q_h\in L^2(\Omega): &q_h|_K\in \qhnc(K)\quad\forall\;K\in\cT_h\;\text{and}\;\\
		&\int_e[q_h]\chi\ds=0\quad\forall\;\chi\in\p_{\ell-1}(e),\;\forall\;e\in\e^i\cup\e^s
	\end{rcases},
\end{dcases}
\end{align*}
respectively. 
\subsection{Discrete problem and well-posedness}\label{subsec:3.2}
Let us first set the continuous bilinear forms $a_1:V\times V, a_2:Q\times V$ and $a_3:Q\times Q$  as
\begin{align*}
a_1(u,v)&:=(u,v)_\Omega+a(u,v) &\forall\;u,v\in V,\\
a_2(p,v)&:= \alpha (\nabla p,\nabla v)_\Omega &\forall\;p\in Q\;\text{and}\;\forall\; v\in V,\\
a_3(p,q)&:= \beta(p,q)_\Omega+\gamma(\nabla p,\nabla q)_\Omega&\forall\;p,q\in Q
\end{align*}
with the local counterparts $a_1^K, a_2^K$ and $a_3^K$ for $K\in\cT_h$ and the piecewise versions $a_1^\pw:=\sum_K a_1^K, a_2^\pw:=\sum_K a_2^K$ and $a_3^\pw:=\sum_K a_3^K$ respectively. For all $u_h,v_h\in \vhc(K)$ or $\vhnc(K)$ and $p_h,q_h\in \qhc(K)$ or $\qhnc(K)$ with $k\geq 2$ and $\ell\geq 1$, define the discrete counterparts by
\begin{subequations}\begin{align}
	a_1^h(u_h,v_h)|_K&:= (\Pi_k u_h,\Pi_kv_h)_K+S_{1,0}^{K}((1-\Pi_k)u_h,(1-\Pi_k)v_h)\nonumber\\&\quad+(\Pi_{k-2}(\nabla^2 u_h),\Pi_{k-2}(\nabla^2 v_h) )_K+S_{\nabla^2}^{K}((1-\pd_k)u_h,(1-\pd_k)v_h),\label{a1h}\\
	a_2^h(p_h,v_h)|_K&:=\alpha(\Pi_{\ell-1}\nabla p_h,\Pi_{k-1}\nabla v_h)_K,\label{a2h}\\
	a_3^h(p_h,q_h)|_K&:=\beta(\Pi_\ell p_h,\Pi_\ell q_h)_K+S_{2,0}^{K}((1-\Pi_\ell)p_h,(1-\Pi_\ell)q_h)\nonumber\\&\quad+\gamma(\Pi_{\ell-1}(\nabla p_h),\Pi_{\ell-1}(\nabla q_h))_K+S_{\nabla}^{K}((1-\pg_\ell)p_h,(1-\pg_\ell)q_h).\label{a3h}
\end{align}\end{subequations}
The stabilisation terms $S_{\nabla^2}^{K}$ and $S_{1,0}^K$ on $\vhc(K)$ or $\vhnc$, and $S_{\nabla}^K$ and $S_{2,0}^K$ on $\qhc(K)$ or $\qhnc(K)$ are positive definite bilinear forms and there exist positive constants $C_{\nabla^2},C_{1,0},C_{\nabla},C_{2,0}$ such that
\begin{subequations}\begin{align}
	C_{\nabla^2}^{-1} |v_h|_{2,K}^2&\leq S_{\nabla^2}^K(v_h,v_h)\leq C_{\nabla^2} | v_h|_{2,K}^2&\forall\;v_h\in \text{Ker$(\pd_k)$},\label{s1}\\C_{1,0}^{-1} \| v_h\|_K^2&\leq S_{1,0}^K(v_h,v_h)\leq C_{1,0} \|v_h\|_K^2 &\forall\;v_h\in \text{Ker$(\Pi_k)$},\label{s2}\\
	C_{\nabla}^{-1}\gamma |q_h|_{1,K}^2&\leq S_{\nabla}^K(q_h,q_h)\leq C_{\nabla}\gamma |q_h|_{1,K}^2 &\forall\;q_h\in \text{Ker$(\pg_\ell)$},\label{s3}\\C_{2,\nabla}^{-1} \beta\|q_h\|_K^2&\leq S_{2,0}^K(q_h,q_h)\leq C_{2,0} \beta\| q_h\|_K^2 &\forall\;q_h\in \text{Ker$(\Pi_\ell)$}.\label{s4}
\end{align} \end{subequations}
Let $\text{dof}_i$ denote the $i^\text{th}$ degree of freedom. Standard examples for  stabilisation terms  satisfying \eqref{s1}-\eqref{s4} respectively are
\begin{align*}
&S_{\nabla^2}^K(v_h,w_h)=h_K^{-2}\sum_{i}\text{dof}_i(v_h)\text{dof}_i(w_h),\quad S_{1,0}^K(v_h,w_h) = h_K^{4}S_{\nabla^2}^K(v_h,w_h),\\
&S_{\nabla}^K(p_h,q_h)=\gamma\sum_{j}\text{dof}_j(p_h)\text{dof}_j(q_h),\quad S_{2,0}^K(p_h,q_h)=\beta h_K^2\sum_{j}\text{dof}_j(p_h)\text{dof}_j(q_h),
\end{align*}
for all $v_h,w_h\in \vhc(K)\;\text{or}\; \vhnc(K)$ and $p_h,q_h\in \qhc(K)\;\text{or}\; \qhnc(K)$. The global discrete bilinear forms $a_1^h: \vhc \times \vhc$ (resp. $\vhnc\times\vhnc$), $ a_2^h:\qhc\times \vhc$ (resp. $\qhnc\times \vhnc$)  and $a_3^h:\qhc\times\qhc$ (resp. $\qhnc\times\qhnc$) are defined by $a_1^h(\cdot,\cdot) := \sum_{K\in\cT_h}a_1^h(\cdot,\cdot)|_K, a_2^h(\cdot,\cdot):= \sum_{K\in\cT_h}a_2^h(\cdot,\cdot)|_K$ and $a_3^h(\cdot,\cdot):=\sum_{K\in\cT_h}a_3^h(\cdot,\cdot)|_K$ for conforming (resp. nonconforming) VEM. We assume that $\ell\leq k$, and then
the discrete problem is to find $(u_h,p_h)\in \vhc\times\qhc$ (resp. $\vhnc\times\qhnc$) such that
\begin{subequations}\label{eq:weak_discrete}
\begin{align}
	a_1^h(u_h,v_h)-a_2^h(p_h,v_h) &= (\tilde{f}_h,v_h)_\Omega &\forall\;v_h\in \vhc\;(\text{resp.}\;\vhnc),\\
	a_2^h(q_h,u_h)+a_3^h(p_h,q_h)&=(\tilde{g}_h,q_h)_\Omega &\forall\; q_h\in\qhc\;(\text{resp.}\;\qhnc),
\end{align}
\end{subequations}
with the discrete right-hand sides $(\tilde{f}_h,v_h)_\Omega:= (\tilde{f},\Pi_kv_h)_\Omega$ and $(\tilde{g}_h,q_h)_\Omega:= (\tilde{g},\Pi_\ell q_h)_\Omega$. To rewrite the above discrete  problem, define the discrete product space $\bH_\epsilon^{h,\mathrm{c}}:=\vhc\times \qhc$ and the discrete  operator  $\cA_h^c: \bH_\epsilon^{h,\mathrm{c}} \to \bH_\epsilon^{h,\mathrm{c}}$  as 
\begin{align}\langle \cA_h^c(\vec{\bu}_h),\vec{\bv}_h  \rangle:=
a_1^h(u_h,v_h)-a_2^h(p_h,v_h)+a_2^h(q_h,u_h)+a_3^h(p_h,q_h)\label{Ah} \end{align}
for $\vec{\bu}_h= (u_h,p_h),\vec{\bv}_h = (v_h,q_h)\in \bH_\epsilon^{h,\mathrm{c}}$. We also define the linear and bounded functional $\cF_h^c:\bH_\epsilon^{h,\mathrm{c}} \to \mathbb{R}$ as
\[
\vec{\bv}_h \mapsto \cF_h^c(\vec{\bv}_h)  :=  (\tilde{f}_h,v_h)_{\Omega} +  (\tilde{g}_h,q_h)_{\Omega},
\]
and therefore  problem \eqref{eq:weak_discrete} is recast as: Find $\vec{\bu}_h^c \in \bH_\epsilon^{h,\mathrm{c}}$ such that 
\begin{equation}\label{eq:operator_discrete}
\langle \cA_h^c(\vec{\bu}_h^c),\vec{\bv}_h \rangle = \cF_h^c(\vec{\bv}_h) \qquad \forall\; \vec{\bv}_h \in \bH_\epsilon^{h,\mathrm{c}}.\end{equation}
Similarly  we define $\bH_\epsilon^{h,\mathrm{nc}}:=\vhnc\times \qhnc$, the discrete  operators  $\cA_h^{\mathrm{nc}}$ and $\cF_h^{\mathrm{nc}}$, and seek 
$\vec{\bu}_h^{\mathrm{nc}} \in \bH_\epsilon^{h,\mathrm{nc}}$ such that 
\begin{equation}\label{eq:operator_discrete_nc}
\langle \cA_h^{\mathrm{nc}}(\vec{\bu}_h),\vec{\bv}_h \rangle = \cF_h^{\mathrm{nc}}(\vec{\bv}_h) \qquad \forall\; \vec{\bv}_h \in \bH_\epsilon^{h,\mathrm{nc}}.
\end{equation}
Define the piecewise version $\|\cdot\|_{\bH_\epsilon^h}$ of the norm $\|\cdot\|_{\bH_\epsilon}$ for $\vec{\bu}=(u,p)\in H^2(\cT_h)\times H^1(\cT_h)$ as 
\[\| \vec{\bu} \|_{\bH_\epsilon^h}^2 : = 
\| u \|_{\Omega}^2+| u |_{2,h}^2
+ \beta \|{p} \|_{\Omega}^2
+ \gamma| {p} |_{1,h}^2:=\sum_{K\in\cT_h}(\| u \|_{K}^2+| u |_{2,K}^2
+ \beta \|{p} \|_{K}^2
+ \gamma| {p} |_{1,K}^2).
\]
The following result yields the solvability of
the discrete problems.
\begin{theorem}\label{theo:iso_discrete}
Problem \eqref{eq:operator_discrete} (resp. \eqref{eq:operator_discrete_nc}) is well-posed in the space $\bH_\epsilon^{h,c}$ (resp. $\bH_\epsilon^{h,\mathrm{nc}}$) equipped with the norm \eqref{eq:product-norm} (resp. $\|\cdot\|_{\bH_\epsilon^{h}}$). 
\end{theorem} 
\begin{proof}
The boundedness of $\cA_h^c$ and $\cA_h^{\mathrm{nc}}$  clearly follows from the stability of the $L^2$-projection operators $ \Pi_{k-2},\Pi_k, \Pi_{\ell-1},$ and $\Pi_\ell$ for $k\geq 2$ and $\ell\geq 1$, and from \eqref{s1}-\eqref{s4}. For $\vec{\bv}_h\in\bH_\epsilon^{h,c}$ or $\bH_\epsilon^{h,\mathrm{nc}}$, the definition \eqref{Ah} implies $\langle \cA_h(\vec{\bv}_h),\vec{\bv}_h\rangle = a_1^h(v_h,v_h)+a_3^h(q_h,q_h)$. The definition \eqref{a1h} of $a_1^h$  and the lower bounds of stabilisation terms \eqref{s1}-\eqref{s2} lead to
\begin{align*}a_1^h(v_h,v_h) &\gtrsim \|\Pi_k v_h\|_{\Omega}^2+\|(1-\Pi_k)v_h\|_\Omega^2 +\|\Pi_{k-2}(\nabla^2 v_h)\|_\Omega^2+|(1-\pd_k)v_h|_{2,h}^2 \\&
	\gtrsim \|v_h\|_\Omega^2+|v_h|_{2,\Omega}^2,
\end{align*}
where we have employed    $\|(1-\Pi_{k-2})(\nabla^2v_h)\|_\Omega\leq |(1-\pd_k)v_h|_{2,h} $ and  triangle inequalities in the last step. Analogously we can prove that $a_3^h$ is coercive, and consequently $\cA_h^c$ (also $\cA_h^{\mathrm{nc}}$) is coercive with respect to the weighted norm $\|\cdot\|_{\bH_\epsilon^{h}}$. Hence the Lax--Milgram  lemma concludes  the proof. 
\end{proof}
\section{Error analysis for conforming VEM}\label{sec:error-c}
This section recalls the standard conforming interpolation estimates and
establishes the a priori error estimates in the energy norm $\|\cdot\|_{\bH_\epsilon}$ (cf. Theorem~\ref{theo:cv-c}).
\begin{proposition}[Conforming interpolation \cite{cangiani17,chen2022}]\label{prop:inter_vc}
There exists an  interpolation operator $\vec{I}_h^c:(V\cap H^s(\Omega))\times (Q\cap H^r(\Omega)) \to \vhc\times \qhc$  such that, for $v\in V\cap  H^s(\Omega)$ with $2\leq s\leq k+1$ and $q\in Q\cap  H^r(\Omega)$ with $1\leq r\leq \ell+1$, $\vec{I}_h^cv:=(v_I^c,q_I^c)$ and
\[|v-v_I^c|_{j,h}\lesssim h^{s-j}|v|_{s,\Omega}\quad\text{for}\;0\leq j \leq{2}\quad\text{and}\quad|q-q_I^c|_{j,h}\lesssim h^{r-j}|q|_{r,\Omega}\quad\text{for}\;0\leq j \leq {1}.\]
\end{proposition}
Throughout this paper, the oscillations of $\tilde{f},\tilde{g}\in L^2(\Omega)$ for $k\geq 2$ and $\ell\geq 1$ are defined as 
\[ \mathrm{osc}_2(\tilde{f},\cT_{h}):= \Big(\sum_{K\in\cT_{h}}\|h_K^2(1-\Pi_k)\tilde{f}\|_K^2\Big)^{1/2},\;\mathrm{osc}_1(\tilde{g},\cT_{h}):= \Big(\sum_{K\in\cT_{h}}\|h_K(1-\Pi_\ell)\tilde{g}\|_K^2\Big)^{1/2}.\]
\begin{theorem}\label{theo:cv-c}
	Given $\vec{\bu}:=(u,p)\in (V\cap H^s(\Omega))\times (Q\cap H^r(\Omega))$ for $s\geq 2$ and  $r\geq 1$,  the unique solution $\vec{\bu}_h^{c}=(u_h^{c},p_h^{c})\in \bH_\epsilon^{h,c}=\vhc\times\qhc$ for $k\geq 2$ and $1\leq \ell\leq k$ to  \eqref{eq:operator_discrete} satisfies 
	\begin{align*}
		\|\vec{\bu}-\vec{\bu}_h^{c}\|_{\bH_\epsilon}&\lesssim \|\vec{\bu}-\vec{I}_h^c\vec{\bu}\|_{\bH_\epsilon}+\|\vec{\bu}-\vec{\Pi}_h \vec{\bu}\|_{\bH_\epsilon^h}+\alpha^{1/2}|u-\pg_\ell u|_{1,h}+\mathrm{osc}_2(\tilde{f},\cT_h)&\\&\quad+\mathrm{osc}_1(\tilde{g},\cT_h)\lesssim h^{\min\{k-1,s-2,\ell,r-1\}}(\|\tilde{f}\|_{s-4,\Omega}+\|\tilde{g}\|_{r-2,\Omega}),
	\end{align*}
	for $\vec{\Pi}_h\vec{\bu} = (\pd_k u,\pg_\ell p)$.
\end{theorem}
\begin{proof}
We drop the superscript $c$ (denoting the conforming case) in the proof just for the sake of notational simplicity. Let $\vec{\be}_h:=(e_h^u,e_h^p)=(u_I-u_h,p_I-p_h)=\vec{\bu}_I-\vec{\bu}_h\in \bH_\epsilon^{h,c}$ for $\vec{\bu}_I=(u_I,p_I)$.
The coercivity of $\cA_h$ from Theorem~\ref{theo:iso_discrete} and the discrete problem \eqref{eq:operator_discrete} in the first step, and an elementary algebra in the second step lead to
\begin{align}
	\|\vec{\be}_h\|^2_{\bH_\epsilon}&\lesssim \cA_h(\vec{\bu}_I,\vec{\be}_h)-\cF_h(\vec{\be}_h)=(a_1^h(u_I-\pd_ku,e_h^u)+a_1^\pw(\pd_ku-u,e_h^u))\nonumber\\&\quad+(a_2(p,e_h^u) -a_2^h(p_I,e_h^u))+(a_3^h(p_I-\pg_\ell p,e_h^p)+a_3^\pw(\pg_\ell p-p,e_h^p))\nonumber\\&\quad+(a_2^h(e_h^p,u_I)-a_2(e_h^p,u))+((\tilde{f}-\tilde{f}_h,e_h^u)_\Omega+(\tilde{g}-\tilde{g}_h,e_h^p)_\Omega)\nonumber\\&\quad=:T_1+T_2+T_3+T_4+T_5.\label{eq:4.1}
\end{align}
The continuity of $a_1^h$ and $a_3^h$ from Theorem~\ref{theo:iso_discrete}, and the Cauchy--Schwarz inequality for $a_1^\pw$ and $a_3^\pw$ show 
\begin{align}
	T_1+T_3&\lesssim (\|u_I-\pd_ku\|_{\Omega}+\|u-\pd_ku\|_{\Omega})\|e_h^u\|_{\Omega}+ (|u_I-\pd_ku|_{2,h}\nonumber\\&\qquad+|u-\pd_ku|_{2,h}) |e_h^u|_{2,\Omega}+\beta^{1/2}(\|p_I-\pg_\ell p\|_{\Omega}+\|p-\pg_\ell p\|_{\Omega})\beta^{1/2}\|e_h^p\|_\Omega\nonumber\\&\qquad+ \gamma^{1/2}(\|p_I-\pg_\ell p\|_{1,h}+\|p-\pg_\ell p\|_{1,h})\gamma^{1/2}|e_h^p|_{1,h}\nonumber\\&\lesssim (\|\vec{\bu}-\vec{I}_h^c\vec{\bu}\|_{\bH_\epsilon}+\|\vec{\bu}-\vec{\Pi}_h \vec{\bu}\|_{\bH_\epsilon^h})\|\vec{\be}_h\|_{\bH_\epsilon}\nonumber\\&\lesssim h^{\min\{k-1,s-2,\ell,r-1\}}(|u|_{s,\Omega}+|p|_{r,
		\Omega})\|\vec{\be}_h\|_{\bH_\epsilon},\label{t1-t3}
\end{align}
with triangle inequalities in the second step, and Propositions~\ref{prop:poly}-\ref{prop:inter_vc} in the last step. Algebraic manipulations and the $L^2$-orthogonality of  $\Pi_{\ell-1}$ imply that 
\begin{align}
	\alpha^{-1}(T_2+T_4)&=(\nabla p-\Pi_{\ell-1}\nabla p_I,\nabla e_h^u)_\Omega+(\Pi_{\ell-1}\nabla p_I,(1-\Pi_{k-1})\nabla e_h^u)_\Omega\nonumber\\&\quad+(\Pi_{\ell-1}\nabla e_h^p,\Pi_{k-1}\nabla u_I-\nabla u)_\Omega+((\Pi_{\ell-1}-1)\nabla e_h^p,(1-\Pi_{\ell-1})\nabla u)_\Omega.\label{eq:4.3}
\end{align}
In addition, triangle inequalities and the $L^2$-orthogonality of $\Pi_{\ell-1}$ provide
\[
\|\nabla p-\Pi_{\ell-1}\nabla p_I\|_\Omega\leq |p-p_I|_{1,\Omega}+|p_I-\pg_\ell p|_{1,h}\lesssim |p-p_I|_{1,\Omega}+|p-\pg_\ell p|_{1,h}.
\]
The second term in \eqref{eq:4.3} vanishes because of the $L^2$-orthogonality of $\Pi_{k-1}$ and assumption $\ell\leq k$. Similarly, the third term in \eqref{eq:4.3} reduces to $(\Pi_{\ell-1}\nabla e_h^p,\Pi_{k-1}\nabla u_I-\nabla u)_\Omega=(\Pi_{\ell-1}\nabla e_h^p,\nabla u_I-\nabla u)_\Omega$. The Cauchy--Schwarz inequality in combination with  the previous bounds in \eqref{eq:4.3} and $\alpha^{1/2}\leq 1 \leq \gamma^{1/2}$ result in 
	\begin{align*}
		T_2+T_4&\lesssim \gamma^{1/2}(|p-p_I|_{1,\Omega}+|p-\pg_\ell p|_{1,h})|e_h^u|_{1,\Omega}\nonumber\\&\quad +\alpha^{1/2}(|u-u_I|_{1,\Omega}+|u-\pg_\ell u|_{1,h})\gamma^{1/2}|e_h^p|_{1,\Omega}\nonumber\\&\lesssim (\|\vec{\bu}-\vec{I}_h^c\vec{\bu}\|_{\bH_\epsilon}+\|\vec{\bu}-\vec{\Pi}_h \vec{\bu}\|_{\bH_\epsilon^h}+\alpha^{1/2}|u-\pg_\ell u|_{1,h}) (|e_h^u|_{2,\Omega}+\gamma^{1/2}|e_h^p|_{1,\Omega})\nonumber\\&\lesssim h^{\min\{\ell,r-1,k,s-1\}}(|e_h^u|_{2,\Omega}+\gamma^{1/2}|e_h^p|_{1,\Omega}),
	\end{align*}
	where $|u-u_I|_{1,\Omega}\lesssim |u-u_I|_{2,\Omega}$ and $|e_h^u|_{1,\Omega}\lesssim |e_h^u|_{2,\Omega}$ from \eqref{norm-equiv}, and Propositions~\ref{prop:poly}-\ref{prop:inter_vc} were used for the last two inequalities. The $L^2$-orthogonality of $\Pi_k$ and $\Pi_\ell$ together with Proposition~\ref{prop:poly} and $\gamma\geq 1 $ allow us to assert  that
\begin{align}
	T_5&=(h_{\cT_h}^2(\tilde{f}-\Pi_k\tilde{f}),h_{\cT_h}^{-2}(1-\Pi_k)e_h^u)_\Omega+(h_{\cT_h}(\tilde{g}-\Pi_\ell\tilde{g}),h_{\cT_h}^{-1}(1-\Pi_\ell)e_h^p)_\Omega\nonumber\\&\lesssim \mathrm{osc}_2(\tilde{f},\cT_h)|e_h^u|_{2,\Omega}+\mathrm{osc}_1(\tilde{g},\cT_h)|e_h^p|_{1,\Omega}\nonumber\\&\lesssim h^{\min\{k+3,s-2,\ell+2,r-1\}}(\|\tilde{f}\|_{s-4,\Omega}+\|\tilde{g}\|_{r-2,\Omega})(|e_h^u|_{2,\Omega}+\gamma^{1/2}|e_h^p|_{1,\Omega}).\label{t5}
\end{align}
The estimates \eqref{t1-t3}-\eqref{t5} in \eqref{eq:4.1} show that \[\|\vec{\be}_h\|_{\bH_\epsilon}\lesssim h^{\min\{k-1,s-2,\ell,r-1\}}(|u|_{s,\Omega}+|p|_{r,\Omega}+\|\tilde{f}\|_{s-4,\Omega}+\|\tilde{g}\|_{r-2,\Omega}).\] This and Proposition~\ref{prop:inter_vc} in the triangle inequality $\|\vec{\bu}-\vec{\bu}_h^c\|_{\bH_\epsilon}\leq \|\vec{\bu}-\vec{\bu}_I\|_{\bH_\epsilon}+\|\vec{\be}_h\|_{\bH_\epsilon}$ followed by the regularity estimates conclude the proof of the theorem.
\end{proof}
\section{Error analysis for nonconforming VEM}\label{sec:error-nc} 
Since the nonconforming discrete spaces $\vhnc$ and $\qhnc$ need not be subsets of continuous spaces $V$ and $Q$, this section   explains the different constructions (at least two) of  conforming companion operators which connect nonconforming VE spaces to  continuous Sobolev spaces. The two crucial ideas in the design are 
\begin{itemize}
\item first to map a nonconforming VE space to a conforming VE space of one degree higher, and 
\item second to modify the linear operator constructed in the first step
through standard bubble-function techniques to achieve additional orthogonal properties (in particular, $L^2$-orthogonality).
\end{itemize}
\subsection{Construction of companion operators}\label{subsec:companion}
Let $\mathrm{dof}_i^{\ell,\mathrm{c}}$ for $i=1,\dots,\mathrm{N}_1^{\ell,\mathrm{c}}$ and $\mathrm{dof}_j^{\ell,\mathrm{nc}}$ for $j=1,\dots,\mathrm{N}_1^{\ell,\mathrm{nc}}$ be the linear functionals associated with DoFs of the VE spaces $\qhc$ and $\qhnc$ of dimensions $\mathrm{N}_1^{\ell,\mathrm{c}}$ and $\mathrm{N}_1^{\ell,\mathrm{n      c}}$ for $\ell\geq 1$. Let $\mathrm{dof}_i^{k,\mathrm{c}}$ for $i=1,\dots,\mathrm{N}_2^{k,\mathrm{c}}$ and $\mathrm{dof}_j^{k,\rm{nc}}$ for $j=1,\dots,\mathrm{N}_2^{k,\mathrm{nc}}$ be the linear functionals associated with DoFs of the VE spaces $\vhc$ and $\vhnc$ of dimensions $\mathrm{N}_2^{k,\mathrm{c}}$ and $\mathrm{N}_2^{k,\mathrm{nc}}$ for $k\geq 2$. 
\begin{theorem}\label{thm:J1}
There exists a linear operator  $J_1:\vhnc\to V_h^{k+1,\mathrm{c}}$ satisfying the following properties:
\begin{enumerate}[before=\leavevmode,label=\upshape(\alph*),ref=\thetheorem (\alph*)]
	\item\label{4.1.a} $\mathrm{dof}_j^{k,\mathrm{nc}} (J_1v_h) = \mathrm{dof}_j^{k,\mathrm{nc}} (v_h)$ for all $j=1,\dots,{\mathrm{N}_2^{k,\mathrm{nc}}}$,
	\item\label{4.1.b} $a^{\pw}(v_h-J_1v_h,\chi)=0$ for all $\chi\in\p_k(\cT_h)$,
	\item\label{4.1.b'}$\nabla(v_h-J_1v_h)\perp (\p_{k-3}(\cT_h))^2$ in $(L^2(\Omega))^2$ for $k\geq 3$,
	\item\label{4.1.c} $\displaystyle \sum_{j=0}^2 h^{j-2}|v_h-J_1v_h|_{j,h}\lesssim \inf_{\chi\in \p_k(\cT_h)}|v_h-\chi|_{2,h}+\inf_{v\in V}|v_h-v|_{2,h}.$
\end{enumerate}
\end{theorem} 
\noindent {\it{Construction of $J_1$}}.
First we observe that DoFs of $\vhnc$ is a subset of DoFs of $V_h^{k+1,c}$.  Next we define a linear operator $J_1:\vhnc\to V_h^{k+1,c}$ through DoFs of $V_h^{k+1,c}$, for $v_h\in \vhnc$, by
\begin{align*}
\mathrm{dof}_j^{k,\mathrm{nc}} (J_1v_h) &= \mathrm{dof}_j^{k,\mathrm{nc}} (v_h)\quad\forall\;j=1,\dots,\mathrm{N}_2^{k,\mathrm{nc}},\\
\nabla J_1v_h(z) &= \frac{1}{|\cT_z|}\sum_{K\in\cT_z}\nabla\pd_kv_h|_K(z)\quad\forall\;z\in\V^i,\\
\dashint_K J_1v_h\chi\dx&=\dashint_K\pd_kv_h\chi\dx\quad\forall\;\chi\in\M^*_{k-3}(K),
\end{align*}
where the set $\cT_z:=\{K\in\cT_h : z\in K\}$ of cardinality $|\cT_z|$ contains the neighbouring polygons $K$ sharing the vertex $z$. We assign $\nabla J_1v_h(z)=\cero$ for the boundary vertices $z\in\V^s$ if $z$ is a corner (the angle at $z$ is not equal to $0,\pi,2\pi$) and for all $z\in\V^c$. If the angle at $z\in\V^s$ is equal to $0,\pi,2\pi$, then we assign
\[\partial_{\bt}(J_1v_h)(z)=0\quad\text{and}\quad\partial_{\nn}(J_1v_h)(z)=\frac{1}{|\cT_z|}\sum_{K\in\cT_z}\partial_{\nn}(\pd_kv_h)|_K(z).\]
\begin{proof}[Proof of Theorem \ref{4.1.a}]
This is an immediate consequence of the definition of $J_1$.
\end{proof}
\begin{proof}[Proof of Theorem~\ref{4.1.b}]
Let $\chi\in\p_k(K)$ and set the notation $ T(\chi):=\partial_{\nn}(\Delta\chi+\partial_{\btau\btau}\chi)$, and $[M_{\nn\btau}(\chi)]_{z_j}:=\partial_{\nn\tau}(\chi)|_{e_{j-1}}(z_j)-\partial_{\nn\tau}(\chi)|_{e_{j}}(z_j)$ for $j=1,\dots,N_K$ with $e_0=e_{N_K}$. Since $\chi\in H^4(K)$ and $v_h-J_1v_h\in H^2(K)$, an integration by parts  leads to 
\begin{align*}
	a^K(v_h-J_1v_h,\chi)&= \int_K \Delta^2\chi (v_h-J_1v_h)\dx+ \int_{\partial K}\partial_{\nn\nn}\chi\partial_{\nn}(v_h-J_1v_h)\ds\\&\qquad-\int_{\partial K}T(\chi)(v_h-J_1v_h)\ds+\sum_{j=1}^{N_K}[M_{\nn\btau}\chi]_{z_j}(v_h-J_1v_h)(z_j)=0,
\end{align*}
with part \ref{4.1.a} being used in the last step. This  holds for any $K\in\cT_h$ and concludes the proof of Theorem \ref{4.1.b}. 
\end{proof}
\begin{proof}[Proof of Theorem~\ref{4.1.b'}]
For any $v_h\in \vhnc, \chi\in\p_{k-2}(K)$ with  $k\geq 3$ and $K\in\cT_h$, an integration by parts and  Theorem~\ref{4.1.a}  show that
\[
(\nabla(v_h-J_1v_h),\nabla \chi)_K=-(v_h-J_1v_h,\Delta\chi)_K+(v_h-J_1v_h,\partial_{\nn}\chi)_{\partial K}=0.
\]	
This proves \ref{4.1.b'}.
\end{proof}
\begin{proof}[Proof of Theorem~\ref{4.1.c}]
Since $(\pd_kv_h-J_1v_h)|_K\in V_h^{k+1,\rm{c}}(K)$, the norm equivalence found in, e.g., \cite[Lemma~3.6]{huang21} shows that 
\begin{align}
	|\pd_kv_h-J_1v_h|_{2,K}\simeq h_K^{-1}\|\mathrm{Dof}^{k+1,\mathrm{c}}(\pd_kv_h-J_1v_h)\|_{\ell^2},\label{4.1}
\end{align}
for the vector $\mathrm{Dof}^{k+1,\mathrm{c}}$ with arguments as the local DoFs of $V_h^{k+1,\rm{c}}(K)$.
Let $z$ be an interior vertex in $\V^i\cap\V_K$ belonging to an edge $e\in\e_K$. The equality $J_1v_h(z)=v_h(z)$ from Theorem \ref{4.1.a} and the inverse estimate for polynomials imply
\begin{align*}
	&|(\pd_kv_h-J_1v_h)|_K(z)|\leq\|(\pd_kv_h-v_h)\|_{\infty,e}\lesssim h_e^{-1/2}\|v_h-\pd_kv_h\|_e\nonumber\\&\qquad\lesssim h_e^{-1}\|v_h-\pd_kv_h\|_K+|v_h-\pd_kv_h|_{1,K}\lesssim h_K|v_h-\pd_kv_h|_{2,K}.
\end{align*}
The  third step follows from the trace inequality, and the last step from \eqref{PF} and \ref{M2}. Let $z$ be an interior vertex in $\V^i\cap\V_K$ or a boundary vertex in $\V^s\cap\V_K$ with angle at $z$ equal to $\pi$,  and  polygons $K_1=K,\dots,K_{|\cT_z|}$ share the node $z$. Suppose $(\pd_kv_h)_i=\pd_kv_h|_{K_i},$ and $K_i$ and $K_{i+1}$ are two neighbouring polygons. Then
\begin{align}
	\nabla(\pd_kv_h-J_1v_h)_1(z)&=\frac{1}{|\cT_z|}\sum_{j=2}^{|\cT_z|}(\nabla(\pd_kv_h)_1-\nabla(\pd_kv_h)_j)(z)\nonumber\\&=\frac{1}{|\cT_z|}\sum_{j=2}^{|\cT_z|}\sum_{i=1}^{j-1}(\nabla(\pd_kv_h)_i-\nabla(\pd_kv_h)_{i+1})(z).\label{4.3}
\end{align}
A consequence of the mesh regularity assumptions \ref{M1}-\ref{M2} is that $|\cT_z|$ is uniformly bounded for any $z\in\mathcal{V}$. Hence it suffices to bound the term $(\nabla(\pd_kv_h)_1-\nabla(\pd_kv_h)_{2})(z)$ for $z\in e$ and an edge $e\in K_1\cap K_2$. In addition, the inverse estimate for polynomials leads to
\begin{align}
|(\nabla(\pd_kv_h)_1-\nabla(\pd_kv_h)_{2})(z)|&\leq \|\nabla(\pd_kv_h)_1-\nabla(\pd_kv_h)_{2}\|_{\infty,e}\nonumber\\&\lesssim h_e^{-1/2}\|[\nabla\pd_kv_h]_e\|_e.\label{4.4}
\end{align}
Let $v\in V$ be an arbitrary function and $a_e:=\dashint_e\nabla(v-v_h)\ds$. Since $a_e$ is uniquely defined from the definition of $v_h\in \vhnc$, rewrite $h_e^{-1/2}\|[\nabla\pd_kv_h]_e\|_e= h_e^{-1/2}\|[\nabla\pd_kv_h-\nabla v+a_e]_e\|_e$. Let $\omega_e$ denote the edge patch of $e$. Then the trace inequality and the triangle inequality show
\begin{align*}
	&h_e^{-1/2}\|[\nabla\pd_kv_h-\nabla v+a_e]_e\|_e\lesssim h_e^{-1}\|\nabla\pd_kv_h-\nabla v+a_e\|_{\omega_e}+|\pd_kv_h-v|_{2,\omega_e}\\&\lesssim h_e^{-1}(\|\pd_kv_h-v_h\|_{1,\omega_e}+ \|
	\nabla(v_h-v)+a_e\|_{\omega_e})+|\pd_kv_h-v_h|_{2,\omega_e}+|v_h-v|_{2,\omega_e}.
\end{align*}
Since $\dashint_e\nabla(v_h-v)+a_e\ds=0$, the Poincar\'e--Friedrichs inequality and \ref{M2} imply $\|
\nabla(v_h-v)+a_e\|_{\omega_e}\lesssim h_e|v_h-v|_{2,\omega_e}$. This and \eqref{PF} in the above displayed estimate provide
\begin{align}h_e^{-1/2}\|[\nabla\pd_kv_h]_e\|_e\lesssim |\pd_kv_h-v_h|_{2,\omega_e}+|v_h-v|_{2,\omega_e}.\label{4.5}
\end{align}
The combination \eqref{4.3}-\eqref{4.5} results in
\[
|h_z\nabla(\pd_kv_h-J_1v_h)|_K(z)|\lesssim h_K(|v_h-\pd_kv_h|_{2,\omega_e}+|v-v_h|_{2,\omega_e}).
\]
Theorem \ref{4.1.a}, the Cauchy--Schwarz inequality,  the trace inequality lead for any $\chi\in\M_{k-2}(e)$ and $e\in\e_K\setminus\e^c$ to
\begin{align*}
	&\int_e\partial_{\nn}(\pd_kv_h-J_1v_h)\chi\ds\leq h_e^{1/2}\|\partial_{\nn}(\pd_kv_h-v_h)\|_e\\&\qquad\lesssim |v_h-\pd_kv_h|_{1,\omega_e}+h_e|v_h-\pd_kv_h|_{2,\omega_e}\lesssim h_e|v_h-\pd_kv_h|_{2,\omega_e},
\end{align*}
with \eqref{PF} in the end. Analogously we can prove for any $\chi\in\M_{k-3}(e)$ and $e\in\e_K\cap\e^i$ that
\[
\dashint_e(\pd_kv_h-J_1v_h)\chi\ds\lesssim h_e|v_h-\pd_kv_h|_{2,\omega_e}.
\]
Again Theorem \ref{4.1.a}, the Cauchy--Schwarz inequality, and \eqref{PF} show for any $\chi\in\M_{k-4}(K)$ that
\[
\dashint_K(\pd_kv_h-J_1v_h)\chi\dx\lesssim h_K|v_h-\pd_kv_h|_{2,K},
\]
and $\dashint_K(\pd_kv_h-J_1v_h)\chi\dx=0$ for $\chi\in\M^*_{k-3}(K)$. The definition of $\pd_k$ from \eqref{pd1} implies $\|v_h-\pd_kv_h\|_{2,h}\leq\inf_{\chi\in\p_k(\cT_h)}|v_h-\chi|_{2,h}$.  The previous  estimates in \eqref{4.1} prove that \[|\pd_kv_h-J_1v_h|_{2,h}\lesssim \inf_{\chi\in\p_k(\cT_h)}|v_h-\chi|_{2,h}+\inf_{v\in V}|v_h-v|_{2,h}.\]
Hence the triangle inequality $|v_h-J_1v_h|_{2,h}\leq|v_h-\pd_kv_h|_{2,h}+|\pd_kv_h-J_1v_h|_{2,h}$ and \eqref{PF} prove the estimate  in  Theorem \ref{4.1.c} for the term $|v_h-J_1v_h|_{2,h}$.  The Poincar\'e--Friedrichs inequality implies $\sum_{j=0}^{1}h^{j-2}|v_h-J_1v_h|_{j,h}\lesssim |v_h-J_1v_h|_{2,h}$.
\end{proof}
The following theorem establishes the construction of the second companion operator which will be used in the sequel.

\begin{theorem}\label{thm:J2}
There exists a linear operator  $J_2:\vhnc\to V$ such that it satisfies   Theorem~\ref{4.1.a}-\ref{4.1.c} and in addition the $L^2$-orthogonality property. In particular,
\begin{enumerate}[before=\leavevmode,label=\upshape(\alph*),ref=\thetheorem (\alph*)]
	\item\label{4.2.a} $\mathrm{dof}_j^{k,\mathrm{nc}} (J_2v_h) = \mathrm{dof}_j^{k,\mathrm{nc}} (v_h)$ for all $j=1,\dots,{\mathrm{N}_{2}^{k,\mathrm{nc}}}$,
	\item\label{4.2.b} $a_{\rm{pw}}(v_h-J_2v_h,\chi)=0$ for all $\chi\in\p_k(\cT_h)$,
	\item\label{4.2.c'} $\nabla(v_h-J_2v_h)\perp (\p_{k-3}(\cT_h))^2$ in $(L^2(\Omega))^2$ for $k\geq 3$,
	\item\label{4.2.c} $v_h-J_2v_h\perp \p_k(\Omega)$ in $L^2(\Omega)$,
	\item\label{4.2.d} $\displaystyle \sum_{j=0}^2 h^{j-2}|v_h-J_2v_h|_{j,h}\lesssim \inf_{\chi\in \p_k(\cT_h)}|v_h-\chi|_{2,\pw}+\inf_{v\in V}|v_h-v|_{2,h}.$
\end{enumerate}
\end{theorem}
\noindent\textit{Construction of $J_2$}. Let $b_K\in H^2_0(K)$ be a bubble-function  supported in $K$ and $v_K\in\p_k(K)$ be the Riesz representative of the linear functional $\p_k(K)\to\mathbb{R},$ defined by, $w_k\mapsto(v_h-J_1v_h,w_k)_K$, for $w_k\in\p_k(K)$ in the Hilbert space $\p_k(K)$ endowed with the weighted scalar product $(b_K\bullet,\bullet)_K$. Given $v_h\in\vhnc$,  the function $\tilde{v}_h\in\p_k(\cT_h)$ with $\tilde{v}_h|_K:=v_K$  and the bubble-function $b_{h}|_K:=b_K\in H^2_0(\Omega)$ satisfy
\begin{align}
(b_h\tilde{v}_h,w_k)_\Omega=(v_h-J_1v_h,w_k)_\Omega\qquad\forall\;w_k\in\p_k(\cT_h),\label{4.10}
\end{align}
and define
\begin{align}
J_2v_h:=J_1v_h+b_h\tilde{v}_h\in V.\label{def:J2}
\end{align}
\begin{proof}[Proof of Theorem \ref{4.2.a}]
Since $b_K=0=\partial_{\nn}(b_K)$ on $\partial K$ for any $K\in \cT_h$, there holds, for any $v_h\in\vhnc$, 
\begin{align*}
	J_2v_h(z)&=J_1v_h(z)=v_h(z) &\text{for any }\;z\in\mathcal{V},\\
	\int_e\partial_{\nn}(J_2v_h)\chi\ds&=\int_e\partial_{\nn}(J_1v_h)\chi\ds=\int_e\partial_{\nn}(v_h)\chi\ds &\text{for}\;\chi\in\M_{k-2}(e)\;\text{and}\;e\in\mathcal{E},\\
	\dashint_eJ_2(v_h)\chi\ds&=\dashint_eJ_1(v_h)\chi\ds=\dashint_e v_h\chi\ds&\text{for}\;\chi\in\M_{k-3}(e)\;\text{and}\;e\in\mathcal{E}.
\end{align*}
For $\chi\in\mathcal{M}_{k-4}(K)$ and $K\in\cT_h$, the definition \eqref{def:J2} of $J_2$ and \eqref{4.10} show $\dashint_KJ_2v_h\chi\dx=\dashint_K(J_1v_h+b_Kv_K)\chi\dx=\dashint_Kv_h\chi\dx$. This concludes the proof of Theorem \ref{4.2.a}.
\end{proof}
\begin{proof}[Proof of Theorem \ref{4.2.b}-\ref{4.2.c'}]
This results from Theorem \ref{4.2.a} and it follows as in the proof of Theorem~\ref{4.1.b}-\ref{4.1.b'}.
\end{proof}
\begin{proof}[Proof of Theorem \ref{4.2.c}]
This is an immediate consequence of the definition \eqref{def:J2} of $J_2$ and 
\eqref{4.10}.
\end{proof}
\begin{proof}[Proof of Theorem \ref{4.2.d}]
The Poincar\'e--Friedrichs inequality implies $\sum_{j=0}^1h^{j-2}|v_h-J_2v_h|_{j,h}\lesssim|v_h-J_2v_h|_{2,h}$. Hence it remains to bound the term $|v_h-J_2v_h|_{2,h}$. The  triangle inequality and \eqref{def:J2} lead to
\begin{align}
	|v_h-J_2v_h|_{2,h}\leq |v_h-J_1v_h|_{2,h}+|b_h\tilde{v}_h|_{2,h}.\label{4.12}
\end{align}
For any $\chi\in\p_k(K)$ and $K\in\cT_h$, there exist  inverse estimates 
\begin{align}
	\|\chi\|^2_{K}\lesssim (b_K,\chi^2)_{K}\lesssim \|\chi\|^2_{K}\quad\text{and}\quad
	\|\chi\|_{K}\lesssim&\sum_{m=0}^2 h_K^m|b_K\chi|_{m,K}\lesssim \|\chi\|_{K}.\label{bubble-inverse}
\end{align}
This implies \begin{align}|b_K v_K|_{2,K}\lesssim h_K^{-2}\|v_K\|_K.\label{4.14}\end{align} The first inequality in \eqref{bubble-inverse}, and \eqref{4.10} with  $w_k=v_K\in\p_k(K)$ result in
\begin{align*}
	\|v_K\|_K^2\lesssim(b_Kv_K,v_K)_K=(v_h-J_1v_h,v_K)_K.
\end{align*}
Hence $\|v_K\|_K\lesssim \|v_h-J_1v_h\|_K$. This, the estimates \eqref{4.12} and \eqref{4.14}, and Theorem~\ref{4.1.c} conclude the proof of Theorem \ref{4.2.d}. 
\end{proof}
\begin{corollary}\label{cor:disc-norm-equiv}
	The piecewise $H^2$-seminorm forms a norm on $\vhnc$, and it is in turn equivalent to the piecewise $H^2$-norm. That is, for any $v_h\in \vhnc$, there holds
	\[\|v_h\|_{2,h}\lesssim |v_h|_{2,h}.\]
\end{corollary}
\begin{proof}
	Recall that $J_2v_h\in V$ for $v_h\in \vhnc$. Then the triangle inequality leads to
	\begin{align*}
		\|v_h\|_{2,h} &\leq \|v_h-J_2v_h\|_{2,h}+\|J_2v_h\|_{2,\Omega}\lesssim |v_h|_{2,h}+|J_2v_h|_{2,\Omega}\\&\lesssim |v_h|_{2,h}+|J_2v_h-v_h|_{2,h}+|v_h|_{2,h}\lesssim |v_h|_{2,h},
	\end{align*}
	with Theorem~\ref{4.2.d} and \eqref{norm-equiv} in the second step, again  the triangle inequality  and Theorem~\ref{4.2.d} in the last  two steps.
\end{proof}
The same idea follows for the second-order VE space $\qhnc$ and the following two theorems similarly construct $J_3$ (as $J_1$) and modify $J_3$ to obtain $J_4$ (as $J_2$) with the $L^2$-orthogonality. We prefer to highlight only the main steps in the construction of $J_3$ to avoid the repetition of the arguments.
\begin{theorem}\label{thm:J3}
There exists a linear operator  $J_3:\qhnc\to Q_h^{\ell+1,c}$ satisfying the following properties:
\begin{enumerate}[before=\leavevmode,label=\upshape(\alph*),ref=\thetheorem (\alph*)]
	\item\label{4.3.a} $\mathrm{dof}_j^{\ell,\mathrm{nc}} (J_3q_h) = \mathrm{dof}_j^{\ell,\mathrm{nc}} (q_h)$ for all $j=1,\dots,{\mathrm{N}_1^{\ell,\mathrm{nc}}}$,
	\item\label{4.3.b} $(\nabla_{\pw}(q_h-J_3q_h),\nabla_{\pw}\chi)_\Omega=0$ for all $\chi\in\p_\ell(\cT_h)$,
	\item\label{4.3.c} $\displaystyle \sum_{j=0}^1 h^{j-1}|q_h-J_3q_h|_{j,h}\lesssim \inf_{\chi\in \p_\ell(\cT_h)}|q_h-\chi|_{1,h}+\inf_{q\in Q}|q_h-q|_{1,h}.$
\end{enumerate}
\end{theorem}
\noindent {\it{Construction of $J_3$}}.
First we observe that the DoFs of $\qhnc$ constitute a subset of the DoFs of $Q_h^{\ell+1,\mathrm{c}}$.  We define a linear operator $J_3:\qhnc\to Q_h^{\ell+1,\mathrm{c}}$ through DoFs of $Q_h^{\ell+1,\mathrm{c}}$, for $q_h\in \qhnc$, by
\begin{align*}
\mathrm{dof}_j^{\ell,\mathrm{nc}} (J_3q_h) &= \mathrm{dof}_j^{\ell,\mathrm{nc}} (q_h)\quad\forall\;j=1,\dots,\mathrm{N}_1^{\ell,\mathrm{nc}},\\
J_3q_h(z) &= \frac{1}{|\cT_z|}\sum_{K\in\cT_z}\pg_\ell q_h|_K(z)\quad\forall\;z\in\V^i\cup\V
^c,\\
\dashint_K J_3v_h\chi\dx&=\dashint_K\pg_\ell q_h\chi\dx\quad\forall\;\chi\in\M^*_{\ell-1}(K).
\end{align*}
\begin{proof}[Proof of Theorem \ref{4.3.a}]
This is an immediate consequence  of the definition of $J_3$. 
\end{proof}
\begin{proof}[Proof of  Theorem \ref{4.3.b}]
An integration by parts and Theorem \ref{4.3.a} prove, for any $\chi\in\p_\ell(K)$ and $K\in\cT_h$, that
\begin{align*}
	(\nabla(q_h-J_3q_h),\nabla\chi)_K=-(q_h-J_3q_h,\Delta\chi)_K +(q_h-J_3q_h,\partial_{\nn}\chi)_{\partial K}=0.  
\end{align*}
This concludes the proof of Theorem \ref{4.3.b}.
\end{proof}
\begin{proof}[Proof of Theorem \ref{4.3.c}]
This follows analogously as  the proof of Theorem~\ref{4.1.c} with obvious modifications. 
\end{proof}
\begin{theorem}\label{thm:J4}
There exists a linear operator  $J_4:\qhnc\to Q$ such that it satisfies   Theorem~\ref{4.3.a}-\ref{4.3.c} and in addition the $L^2$-orthogonality property. In particular,
\begin{enumerate}[before=\leavevmode,label=\upshape(\alph*),ref=\thetheorem (\alph*)]
	\item\label{4.4.a} $\mathrm{dof}_j^{\ell,\mathrm{nc}} (J_4q_h) = \mathrm{dof}_j^{\ell,\mathrm{nc}} (q_h)$ for all $j=1,\dots,{\mathrm{N}_1^{\ell,\mathrm{nc}}}$,
	\item\label{4.4.b} $(\nabla_{\pw}(q_h-J_3q_h),\nabla_{\pw}\chi)_\Omega=0$ for all $\chi\in\p_\ell(\cT_h)$,
	\item\label{4.4.c} $q_h-J_4q_h\perp \p_\ell(\Omega)$ in $L^2(\Omega)$,
	\item\label{4.4.d} $\displaystyle \sum_{j=0}^1 h^{j-1}|q_h-J_4q_h|_{j,h}\lesssim \inf_{\chi\in \p_\ell(\cT_h)}|q_h-\chi|_{1,h}+\inf_{q\in Q}|q_h-q|_{1,h}.$
\end{enumerate}
\end{theorem}
\subsection{Energy error estimate}\label{subsec:5.2}
This section proves the energy error estimate for the nonconforming case invoking the companion operators constructed in Subsection~\ref{subsec:companion}.
\begin{proposition}[Nonconforming interpolation]\label{prop:inter_v}
There exists an  interpolation operator $\vec{I}_h^{\mathrm{nc}}:(V\cap H^s(\Omega))\times (Q\cap H^r(\Omega)) \to \vhnc\times \qhnc$  such that, for $v\in V\cap  H^s(\Omega)$ with $2\leq s\leq k+1$ and $q\in Q\cap  H^r(\Omega)$ with $1\leq r\leq \ell+1$, $\vec{I}_h^{\mathrm{nc}}v:=(v_I^{\mathrm{nc}},q_I^{\mathrm{nc}})$ satisfies
\[|v-v_I^{\mathrm{nc}}|_{j,h}\lesssim h^{s-j}|v|_{s,\Omega}\;\text{for}\;0\leq j \leq 2\;\text{and}\;|q-q_I^{\mathrm{nc}}|_{j,h}\lesssim h^{r-j}|q|_{r,\Omega}\;\text{for}\;0\leq j \leq 1.\]
\end{proposition}
\begin{theorem}\label{theo:cv-nc}
Given $\vec{\bu}:=(u,p)\in (V\cap H^s(\Omega))\times (Q\cap H^r(\Omega)) $ for $s\geq 2$ and for $r\geq 1$,  the unique solution $\vec{\bu}_h^{\mathrm{nc}}=(u_h^{\mathrm{nc}},p_h^{\mathrm{nc}})\in \bH_\epsilon^{h,\mathrm{nc}}=\vhnc\times\qhnc$ for $k\geq 2$ and $\ell\geq 1$ to  \eqref{eq:operator_discrete_nc} satisfies 
	\begin{align*}
		\|\vec{\bu}-\vec{\bu}_h^{\mathrm{nc}}\|_{\bH_\epsilon^h}&\lesssim \|\vec{\bu}-\vec{I}_h^{\mathrm{nc}}\vec{\bu}\|_{\bH_\epsilon^h}+\|\vec{\bu}-\vec{\Pi}_h \vec{\bu}\|_{\bH_\epsilon^h}+\alpha^{1/2}(|u-\pg_\ell u|_{1,h}&\\&\quad+h|p-\pg_{k-2}p|_{1,h})+\beta h\|p-\pg_\ell p\|_\Omega+\mathrm{osc}_2(\tilde{f},\cT_h)+\mathrm{osc}_1(\tilde{g},\cT_h)&\\&\lesssim h^{\min\{k-1,\ell,s-2,r-1\}}(\|\tilde{f}\|_{s-4,\Omega}+\|\tilde{g}\|_{r-2,\Omega}).
\end{align*}
\end{theorem}
\begin{proof}
Let $\vec{\bu}_I:=(u_I,p_I)\in\vhnc\times\qhnc$ be an interpolation of $\vec{\bu}$ and $\vec{\be}_h=(e_h^u,e_h^p):=(u_I-u_h,p_I-p_h)$. The coercivity of $\cA_h$ from Theorem~\ref{theo:iso_discrete} and the discrete problem \eqref{eq:operator_discrete_nc} lead to
\begin{align*}
	\|\vec{\be}_h\|^2_{\bH_\epsilon^h}&\lesssim\cA_h(\vec{\be}_h,\vec{\be}_h)=\cA_h(\vec{\bu}_I,\vec{\be}_h)-\cF_h(\vec{\be}_h)\\&=a_1^h(u_I,e_h^u)-a_2^h(p_I,e_h^u)+a_2^h(e_h^p,u_I)+a_3^h(p_I,e_h^p)-(\tilde{f}_h,e_h^u)_\Omega-(\tilde{g}_h,e_h^p)_\Omega\\&=(a_1^h(u_I-\pd_ku,e_h^u)+a_1^\pw(\pd_ku-u,e_h^u))+(a_1^\pw(u,e_h^u)-(\tilde{f},e_h^u)_\Omega)\\&\quad+(\tilde{f}-\tilde{f}_h,e_h^u)_\Omega+(-a_2^h(p_I,e_h^u)+a_2^h(e_h^p,u_I))+(a_3^h(p_I-\pg_\ell p,e_h^p)\\&\quad+a_3^\pw(\pg_\ell p-p,e_h^p))+(a_3^\pw(p,e_h^p)-(\tilde{g},e_h^p)_\Omega)+(\tilde{g}-\tilde{g}_h,e_h^p)_\Omega\\&=:T_1+T_2+T_3+T_4+T_5+T_6+T_7,\label{4.13}
\end{align*}
with an elementary algebra in the last two steps.  The boundedness of $\cA_h$ from Theorem~\ref{theo:iso_discrete} and the Cauchy--Schwarz inequality for $a_1^\pw$ and $a_3^\pw$ show the chain of bounds 
\begin{align} 
	T_1+T_5&\lesssim (\|u_I-\pd_k u\|_{\Omega}+\|u-\pd_k u\|_{\Omega})\|e_h^u\|_{\Omega}+(|u_I-\pd_k u|_{2,h}\nonumber\\&\quad+|u-\pd_k u|_{2,h}) |e_h^u|_{2,h}+\beta^{1/2}(\|p_I-\pg_\ell p\|_{\Omega}+\|p-\pg_\ell p\|_{\Omega})\beta^{1/2}\|e_h^p\|_\Omega\nonumber\\&\quad+\gamma^{1/2}(|p_I-\pg_\ell p|_{1,h}+|p-\pg_\ell p|_{1,h})\gamma^{1/2}|e_h^p|_{1,h}\nonumber\\&\lesssim (\|\vec{\bu}-\vec{\bu}_I\|_{\bH_\epsilon^h}+\|\vec{\bu}-\vec{\Pi}_h \vec{\bu}\|_{\bH_\epsilon^h})\|\vec{\be}_h\|_{\bH_\epsilon^h}\nonumber\\&\lesssim h^{\min\{k-1,s-2,\ell,r-1\}}(|u|_{s,\Omega}+|p|_{r,\Omega})\|\vec{\be}_h\|_{\bH_\epsilon^h},
\end{align}
with the triangle inequality in the second step, and Propositions~\ref{prop:poly}-\ref{prop:inter_v} in the last step. Taking $\vec{\bv}=(J_2e_h^u,J_4e_h^p)$ in the continuous problem \eqref{eq:operator} allows us to assert that 
\begin{align*}
	&T_2+T_4+T_6 \nonumber \\
	&\quad = a_1^\pw(u,e_h^u-J_2e_h^u)+a_2(p,J_2e_h^u)+(\tilde{f},J_2e_h^u-e_h^u)_\Omega+a_3^\pw(p,e_h^p-J_4e_h^p)\nonumber\\&\qquad-a_2(J_4e_h^p,u)+(\tilde{g},J_4e_h^p-e_h^p)_\Omega-(a_2^h(p_I,e_h^u)-a_2^h(e_h^p,u_I))\nonumber\\&\quad =\biggl(a_1^\pw(u-\pd_k u,e_h^u-J_2e_h^u)+(\tilde{f}-\Pi_k\tilde{f},J_2e_h^u-e_h^u)_\Omega+a_3^\pw(p-\pg_\ell p,e_h^p-J_4e_h^p)\nonumber\\&\qquad+(\tilde{g}-\Pi_\ell\tilde{g},J_4e_h^p-e_h^p)_\Omega\biggr)+\biggl(a_2(p,J_2e_h^u)-a_2^h(p_I,e_h^u)+a_2^h(e_h^p,u_I)-a_2(J_4e_h^p,u)\biggr)\nonumber\\&\quad=:T_8+T_9.
\end{align*}
The last step follows from Theorems~\ref{4.2.b}-\ref{4.2.c} and \ref{4.4.b}-\ref{4.4.c}. The Cauchy--Schwarz inequality and Theorems~\ref{4.2.d} and \ref{4.4.d} show for $T_8$ that
\begin{align*}
	T_8&\lesssim (|u-\pd_k u|_{2,h}+\mathrm{osc}_2(\tilde{f},\cT_h))|e_h^u|_{2,h}+(\beta h\|p-\pg_\ell p\|_\Omega+\gamma^{1/2}|p-\pg_\ell p|_{1,h} \nonumber\\&\quad +\mathrm{osc}_1(\tilde{g},\cT_h))\gamma^{1/2}|e_h^p|_{1,h}\\&\lesssim h^{\min\{k-1,s-2,\ell,r-1\}}(|u|_{s,\Omega}+|p|_{r,\Omega}+|\tilde{f}|_{s-4,\Omega}+|\tilde{g}|_{r-2,\Omega})\|\vec{\be}_h\|_{\bH_\epsilon^h}.
\end{align*}
Next, an elementary algebraic manipulation  for $T_9$ provides
\begin{align}
	\alpha^{-1}T_9	&=(\nabla p-\Pi_{\ell-1}\nabla p_I,\nabla J_2e_h^u)_\Omega+(\Pi_{\ell-1}\nabla p_I,\nabla J_2e_h^u-\Pi_{k-1}\nabla e_h^u)_\Omega\nonumber\\&\quad+(\Pi_{\ell-1}\nabla e_h^p,\Pi_{k-1}\nabla u_I-\nabla u)_\Omega+(\Pi_{\ell-1}\nabla e_h^p-\nabla J_4e_h^p,\nabla u-\nabla\pg_\ell u)_\Omega,\label{4.18}
\end{align}
with the $L^2$-orthogonality of $\Pi_{\ell-1}$ and Theorem~\ref{4.4.b} in the last term. The Cauchy--Schwarz inequality, the triangle inequality $\|\nabla p-\Pi_{\ell-1}\nabla p_I\|_\Omega\leq |p-p_I|_{1,h}+\|(1-\Pi_{\ell-1})\nabla_\pw p_I\|_\Omega$   for the first term in \eqref{4.18} lead to 
\begin{align}
	(\nabla p-\Pi_{\ell-1}\nabla p_I,\nabla J_2e_h^u)_\Omega&\leq (|p-p_I|_{1,h}+\|(1-\Pi_{\ell-1})\nabla_\pw p_I\|_\Omega) |J_2e_h^u|_{1,\Omega}\nonumber\\&\lesssim  (|p-p_I|_{1,h}+\|\nabla p-\Pi_{\ell-1}\nabla p\|_\Omega) |J_2e_h^u|_{2,\Omega}\nonumber\\&\lesssim h^{\min\{\ell,r-1\}}|p|_{r,\Omega}|e_h^u|_{2,h},\label{4.19}
\end{align} 
having employed $\|\nabla p_I-\Pi_{\ell-1}\nabla p_I\|_K\leq \|\nabla p_I-\Pi_{\ell-1}\nabla p\|_K$ for any $K\in\cT_h$ followed by the triangle inequality in the second step, and Propositions~\ref{prop:poly}-\ref{prop:inter_v} and the stability of $J_2$ from Theorem \ref{4.2.d}   in the last step.  For $k=2$, the Cauchy--Schwarz inequality and the $L^2$-stability of $\Pi_{\ell-1}$  for  the second term in \eqref{4.18} imply
\begin{align*}&(\Pi_{\ell-1}\nabla p_I,\nabla J_2e_h^u-\Pi_{1}\nabla e_h^u)_\Omega\leq | p_I|_{1,h}\|\nabla J_2e_h^u-\Pi_{1}\nabla e_h^u\|_\Omega\\&\qquad\leq (|p_I-p|_{1,h}+|p|_{1,\Omega})(| e_h^u- J_2e_h^u|_{1,h}+|e_h^u-\Pi_{1} e_h^u|_{1,h})\lesssim h|p|_{1,\Omega}|e_h^u|_{2,h}.
\end{align*}
The second step results from the triangle inequality, and the last step from Propositions~\ref{prop:poly}-\ref{prop:inter_v}, and Theorem~\ref{4.2.d}. 
Theorem~\ref{4.2.c'} in the second term of \eqref{4.18} for $k\geq 3$ leads to 
\begin{align*}
	&(\Pi_{\ell-1}\nabla p_I,\nabla J_2e_h^u-\Pi_{k-1}\nabla e_h^u)_\Omega=(\Pi_{\ell-1}\nabla p_I-\nabla_\pw (\pg_{k-2} p),\nabla J_2e_h^u-\Pi_{k-1}\nabla e_h^u)_\Omega\nonumber\\&\leq (\|\nabla p-\Pi_{\ell-1}\nabla p\|_\Omega+| p-p_I|_{1,h}+|p-\pg_{k-2}p|_{1,h})\nonumber\\&\quad\times(| e_h^u-J_2e_h^u|_{1,h}+\|(1-\Pi_{k-1})\nabla_\pw e_h^u\|_\Omega)\lesssim h^{\min\{\ell,r,k-1\}}|p|_{r,\Omega}|e_h^u|_{2,h},
\end{align*} 
where we have used the bound $\|\nabla p_I-\Pi_{\ell-1}\nabla p_I\|_K\leq \|\nabla p_I-\Pi_{\ell-1}\nabla p\|_K$ for any $K\in\cT_h$ and the triangle inequality in the second step, and Propositions~\ref{prop:poly}-\ref{prop:inter_v}  and Theorem~\ref{4.2.d}  in the last step. Similarly  the remaining two terms  in \eqref{4.18} are handled as
\begin{subequations}\begin{align}
		(\Pi_{\ell-1}\nabla e_h^p,\Pi_{k-1}\nabla u_I-\nabla u)_\Omega&\lesssim h^{\min\{k-1,s-1\}}|u|_{s,\Omega}|e_h^p|_{1,h},\label{4.21}\\
		(\Pi_{\ell-1}\nabla e_h^p-\nabla J_4e_h^p,\nabla u-\nabla_\pw(\pg_\ell u))_\Omega&\lesssim h^{\min\{\ell,s-1\}}|u|_{s,\Omega}|e_h^p|_{1,h}.\label{4.22}
\end{align}\end{subequations}
The combination \eqref{4.19}-\eqref{4.22} and the observation $\alpha\leq 1\leq \gamma $ in \eqref{4.18}  prove that 
\begin{align*}
		T_9&\lesssim (\|\vec{\bu}-\vec{\bu}_I\|_{\bH_\epsilon^h}+\|\vec{\bu}-\vec{\Pi}_h \vec{\bu}\|_{\bH_\epsilon^h})|e_h^u|_{2,h}\\&\quad+\alpha^{1/2}(|u-\pg_\ell u|_{1,h}+h|p-\pg_{k-2}p|_{1,h})\gamma^{1/2} |e_h^p|_{1,h}\\&\lesssim h^{\min\{k-1,s-1,\ell,r-1\}} (|u|_{s,\Omega}+|p|_{r,\Omega})(|e_h^u|_{2,h}+\gamma^{1/2} |e_h^p|_{1,h}).
\end{align*}
The $L^2$-orthogonality of $\Pi_k$ and $\Pi_\ell$, and Proposition~\ref{prop:poly} result in
\begin{align*}
	T_3+T_7&=(h_{\cT_h}^2(1-\Pi_k)\tilde{f},h_{\cT_h}^{-2}(1-\Pi_k)e_h^u)_\Omega+(h_{\cT_h}(1-\Pi_\ell)\tilde{g},h_{\cT_h}^{-1}(1-\Pi_\ell)e_h^p)_\Omega\nonumber\\&\lesssim \mathrm{osc}_2(\tilde{f},\cT_h)|e_h^u|_{2,h}+\mathrm{osc}_1(\tilde{g},\cT_h)|e_h^p|_{1,h}\nonumber\\&\lesssim h^{\min\{k+1,s-2,\ell+1,r-1\}}(|\tilde{f}|_{s-4,\Omega}+|\tilde{g}|_{r-2,\Omega})(|e_h^u|_{2,h}+\gamma^{1/2}|e_h^p|_{1,h}),
\end{align*}
with $\gamma\geq 1 $ in the end. The previous estimates in \eqref{4.13} readily prove that 
\[\|\vec{\be}_h\|_{\bH_\epsilon^h}\lesssim h^{\min\{k-1,s-2,\ell,r-1\}} (|u|_{s,\Omega}+|p|_{r,\Omega}+|\tilde{f}|_{s-4,\Omega}+|\tilde{g}|_{r-2,\Omega}).\] This and Proposition~\ref{prop:inter_v} in the triangle inequality $\|\vec{\bu}-\vec{\bu}_h\|_{\bH_\epsilon^h}\leq \|\vec{\bu}-\vec{\bu}_I\|_{\bH_\epsilon^h}+\|\vec{\be}_h\|_{\bH_\epsilon^h}$ followed by regularity estimates conclude the proof of the theorem.
\end{proof}
\begin{remark}[Best-approximation for lowest-order case]
If we reconstruct $J_2$ for $k=2$ with an additional $H^1$-orthogonality in Theorem~\ref{4.2.c'} as $\nabla(v_h-J_2v_h)\perp (\p_0(\cT_h))^2$ in $(L^2(\Omega))^2$  (see Subsection~\ref{subsec:modified_comp} for a definition), then the error estimate in Theorem~\ref{theo:cv-nc} for $k=2$ and $\ell=1$ can be written in the best-approximation form 
\begin{align*}
		\|\vec{\bu}-\vec{\bu}_h^{\mathrm{nc}}\|_{\bH_\epsilon^h}&\lesssim \|\vec{\bu}-\vec{I}_h^{\mathrm{nc}}\vec{\bu}\|_{\bH_\epsilon^h}+\|\vec{\bu}-\vec{\Pi}_h \vec{\bu}\|_{\bH_\epsilon^h}+\alpha^{1/2}|u-\pg_\ell u|_{1,h}+\beta h\|p-\pg_\ell p\|_\Omega&\\&\quad+\mathrm{osc}_2(\tilde{f},\cT_h)+\mathrm{osc}_1(\tilde{g},\cT_h).
\end{align*}
\end{remark}
\section{A posteriori error analysis}\label{sec:apost}
This section contains the derivation of a posteriori error indicators and the proof of their robustness. We provide the details for the nonconforming case and a remark for the conforming case to avoid the repetition of arguments. 

\subsection{Preliminaries}\label{subsec:pre}
We collect here the following local estimates, proven in \cite[Lemma 3.2]{huang21} and \cite[Lemmas 3.3-3.4]{huang21}, respectively.
\begin{lemma}\label{lem:huang32}
For any $\epsilon >0$, there exists a positive constant $c(\epsilon)$ such that
\[|v|_{1,K} \lesssim \epsilon h_K |v|_{2,K} + c(\epsilon) h_K^{-1} \|v\|_{K} \qquad \forall v \in H^2(K).\]
\end{lemma}
\begin{lemma}\label{lem:huang33-34}
For every $v \in H^2(K)$ such that $\Delta^2 v \in \mathbb{P}_{k-4}(K)$, there exists a polynomial $p \in \mathbb{P}_k(K)$ satisfying 
\[ \Delta^2 v = \Delta^2 p \qquad \text{in } K.\]
Moreover, the following estimates hold
\[
|p|_{2,K}  \lesssim h_K^2\|\Delta^2 v\|_{K} \lesssim |v|_{2,K},\qquad 
|p|_{2,K}  \lesssim h_K^{-2}\|v\|_{K},\qquad 
\|p\|_{K}  \lesssim \|v\|_{K}.
\]
\end{lemma}
\subsection{Standard estimates}\label{subsec:std_est}
We start with technical results (inverse estimates and norm equivalences) that are required in the analysis of the nonconforming formulations. These tools are available in the literature only for the conforming case. There are two terminologies, namely original and enhanced VE spaces, in the VE literature (see \cite{ahmad13} for more details). This paper utilises the enhanced versions, but we first  prove the results for the original space and then build for the enhanced space.  Let us denote  the original  local nonconforming VE space for deflections by $\tvhnc(K)$  and define by
\[\tvhnc(K):=\begin{dcases}
\begin{rcases}
	&v_h\in H^2(K)\cap C^0(\partial K): \Delta^2v_h\in \p_{k-4}(K),\;v_h|_e\in\p_k(e)\;\text{and}\\&\;\Delta v_h|_e\in\p_{k-2}(e)\quad\forall\;e\in\e_K,\;v_h|_{s_j}\in C^1(s_j),\\&\;\int_{e_{m_j}}(v_h-\pd_k v_h)\chi\ds=0\quad\forall\;\chi\in\p_{k-2}(e_{m_j}),\;\text{and}\\
	&\int_{e_{m_j+i}}(v_h-\pd_k v_h)\chi\ds=0\quad\forall\;\chi\in\p_{k-3}(e_{m_j+i})\\&\;\text{for}\;i=1,\dots,n_j;
	\;j=1,\dots,\tilde{N}_K
\end{rcases}.
\end{dcases}
\]
\begin{lemma}[Inverse estimates]\label{lem:inverse}
For any $\tilde{v} \in \tvhnc(K)$ and $K\in\cT_h$, there holds
\begin{equation}\label{eq:new-inv}
	|\tilde{v}|_{2,K}  \lesssim h_K^{-2}\|\tilde{v}\|_{K} \quad \text{and} \quad |\tilde{v}|_{1,K}  \lesssim h_K^{-1}\|\tilde{v}\|_{K}.\end{equation}
\end{lemma}
\begin{proof}
Given $\tilde{v} \in \tvhnc(K)$, $\Delta^2\tilde{v}\in \p_{k-4}(K)$ and consequently, we can choose a polynomial   $p \in \mathbb{P}_k(K)$ from Lemma~\ref{lem:huang33-34}. The triangle inequality and the second bound in Lemma~\ref{lem:huang33-34}   assert that 
\[ |\tilde{v}|_{2,K} \leq |\tilde{v}-p|_{2,K} + |p|_{2,K} \lesssim |\tilde{v}-p|_{2,K}+ h_K^{-2}\|\tilde{v}\|_{K}.\]
If we prove  $|\tilde{v}-p|_{2,K}  \lesssim h_K^{-2}\|\tilde{v}-p\|_{K}$, then the triangle inequality together with the third bound in  Lemma~\ref{lem:huang33-34} will provide 
\[
|\tilde{v}-p|_{2,K}  \lesssim h_K^{-2}\|\tilde{v}-p\|_{K} 
\lesssim h_K^{-2}(\|\tilde{v}\|_{K}+\|p\|_{K}) 
\lesssim h_K^{-2}\|\tilde{v}\|_{K}.
\]
Hence we concentrate on showing  $|\tilde{v}-p|_{2,K}  \lesssim h_K^{-2}\|\tilde{v}-p\|_{K}$. First we note that since $\Delta^2\tilde{v} = \Delta^2 p$ and $\tilde{v} - p \in \tvhnc(K)$, without loss of generality we can assume that $\Delta^2\tilde{v} = 0$. Then we define, for a fixed $\tilde{v}$, the following set
\begin{equation}\label{def:Sk}
	S(K):= \{ w\in H^2(K): w = \tilde{v} \text{on $\partial K$,} \int_e\partial_{\nn}(\tilde{v}-w)\chi \ds =0 \quad \forall \chi \in \mathbb{P}_{k-2}(e),\, e\in \mathcal{E}_K\}\end{equation}
and the fact  $a^K(\tilde{v},\tilde{v}-w) =0$ for $w\in S(K)$ leads to
\begin{align} |\tilde{v}|_{2,K} \leq |w|_{2,K} \quad \forall w\in S(K).\label{eq:6.3}\end{align}
Next,  we define  $Q_K\tilde{v} \in \widetilde{V}_h^{k+1,\rm{c}}(K)$ for $\tilde{v} \in \tvhnc(K)$ through the DoFs as
\begin{equation}
	\label{eq:aux09}
	\mathrm{Dof}_{\partial K}^{k+1,\mathrm{c}}(Q_K\tilde{v}) = \mathrm{Dof}_{\partial K}^{k+1,\mathrm{c}}(\tilde{v}), \quad \text{and }  \mathrm{Dof}_{K}^{k+1,\mathrm{c}}(Q_K\tilde{v}) = 0,\end{equation}
where
\[ \mathrm{Dof}^{k+1,\mathrm{c}}(\bullet) = \underbrace{\mathrm{Dof}_{\partial K}^{k+1,\mathrm{c}}(\bullet)}_{\text{boundary DoFs}} \cup \underbrace{\mathrm{Dof}_K^{k+1,\mathrm{c}}(\bullet)}_{\text{interior DoFs}}.\]
Observe that, for $\tilde{v} \in H^2(K)$, its tangential derivative $\partial_{\bt} \tilde{v}$ is well-defined along $\partial K$. If $z$ is not a corner in $\partial K$, then we can assign that $\partial_{\nn}\tilde{v}(z) = 0$, and if $z$ is a corner then the two tangential derivatives at $z$ will suffice to define $\nabla \tilde{v}(z)$ uniquely. This implies that $Q_K\tilde{v}$ is well-defined. In addition, since $Q_K\tilde{v}$ is uniquely determined by boundary DoFs of  $\tilde{v}$ and $\tilde{v}|_e\in \mathbb{P}_k(e)$ for all $e\in \mathcal{E}_K$, we have
\[Q_K\tilde{v} = \tilde{v} \quad \text{on $\partial K$}.\]
This and \eqref{eq:aux09} show that $Q_K\tilde{v} \in S(K)$ and, consequently \eqref{eq:6.3} imply the first inequality in
\begin{align*} |\tilde{v}|_{2,K}  \leq |Q_K\tilde{v}|_{2,K}  &\lesssim h_K^{-2} \|Q_K\tilde{v}\|_{K} \\&\lesssim h_K^{-1} \|\mathrm{Dof}^{k+1,\mathrm{c}}(Q_K\tilde{v})\|_{\ell^2}  =  h_K^{-1} \|\mathrm{Dof}_{\partial K}^{k+1,\mathrm{c}}(Q_K\tilde{v})\|_{\ell^2} \end{align*} 
with the inverse estimate and the norm equivalence available for conforming VE functions in the next two  inequalities,  and  \eqref{eq:aux09} in the last equality. 

Let us examine each contribution to the DoFs in the expression above. Firstly, for any $z\in \cV_K$ it can be inferred that 
\begin{align*}
	|\tilde{v}(z)| & \leq \| \tilde{v}\|_{L^\infty(\partial K)} \lesssim h_K^{-1/2}\|\tilde{v}\|_{L^2(\partial K)}\\
	& \lesssim h_K^{-1}\|\tilde{v}\|_{K} + |\tilde{v}|_{1,K} \lesssim (1 + c(\epsilon)) h_K^{-1}\|\tilde{v}\|_{K} + \epsilon h_K |\tilde{v}|_{2,K},
\end{align*} 
where we have used the inverse estimate for polynomials in 1d in the second step, and  the trace inequality and Lemma~\ref{lem:huang32} in the last two steps.

Secondly,  similar arguments show that
\begin{align*}
	h_z |\nabla \tilde{v}(z)| & = h_z|\nabla Q_K\tilde{v}(z)| \leq h_z | Q_K \tilde{v}|_{1,\infty,\partial K} \lesssim \| Q_K \tilde{v}\|_{\infty,\partial K}\\
	&\lesssim h_K^{-1/2}\|\tilde{v}\|_{\partial K}\lesssim h_K^{-1} \|\tilde{v}\|_{K} + \epsilon h_K |\tilde{v}|_{2,K}.
\end{align*}
Note that the inverse inequality for polynomials  is suitable here since $\partial_{\nn}Q_K\tilde{v}|_e \in \mathbb{P}_{k-1}(e)$. For the remaining boundary moments, the Cauchy--Schwarz inequality and the inverse estimate lead to
\begin{align*}
	\Big|\int_e \partial_{\nn} \tilde{v}\chi_{k-2} \ds\Big| & = \Big|\int_e \partial_{\nn} (Q_K \tilde{v}) \chi_{k-2}\ds\Big| \leq \|\partial_{\nn}(Q_K \tilde{v})\|_{e}h_e^{1/2}\\
	& \lesssim h_e^{-1/2} \|Q_K\tilde{v}\|_{e} =h_e^{-1/2} \|\tilde{v}\|_{e} \lesssim h_K^{-1} \|\tilde{v}\|_{K}  + \epsilon h_K |\tilde{v}|_{2,K} 
\end{align*}
with \eqref{eq:6.3} and  Lemma~\ref{lem:huang32} once more in the last two steps.  Similarly, we can prove that
\[ \dashint_e \tilde{v} \chi_{k-3}\ds \lesssim h_K^{-1} \|\tilde{v}\|_{K}  + \epsilon h_K |\tilde{v}|_{2,K} .\]
These bounds allow us to write 
$  \|\mathrm{Dof}_{\partial K}^{k+1,\mathrm{c}}(Q_K\tilde{v})\|_{\ell^2}  \lesssim h_K^{-1} \|\tilde{v}\|_{K} + \epsilon h_K |\tilde{v}|_{2,K},$
which in turn proves that 
\[  |\tilde{v}|_{2,K} \lesssim h_K^{-2}\|\tilde{v}\|_{K} + \epsilon  |\tilde{v}|_{2,K}\]
and absorbing the $\epsilon$ term on the left-hand side we immediately have the first bound in \eqref{eq:new-inv}. For the second bound it suffices to combine the first bound with  Lemma~\ref{lem:huang32}.
\end{proof}
\begin{figure}[H]
\begin{center}
	\includegraphics[width=0.25\textwidth]{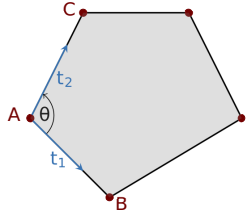}
\end{center}

\vspace{-4mm}
\caption{Sketch of a polygonal domain $K$ and three consecutive vertices $A,B,C$. The unit vectors $\bt_1,\bt_2$ form an angle $\theta$ on $A$.}\label{fig:diagram2}
\end{figure}
\begin{lemma}[Poincar\'e-type inequality]\label{lem:poin}
Let $K$ be a polygonal domain and $v\in H^2(K)$.  If $v(A) = v(B) = v(C)$ for any three non-collinear consecutive vertices $A,B,C$ of $K$ (see the diagram in Figure~\ref{fig:diagram2}), then there exists a positive constant $C_{\rm{P}}$ depending only on the mesh regularity parameter $\rho$, such that  
\[ |v|_{1,K} \leq C_{\rm{P}} h_K |v|_{2,K}.\] 
\end{lemma}
\begin{proof}

With respect to Figure~\ref{fig:diagram2}, let $\bt_1,\bt_2$ be two tangential unit vectors along the sides $AB$ and $AC$, respectively, oriented as in the diagram (moving away from the vertex $A$), forming an angle $\theta \in (0,\pi)$, and 
$|\bt_1 \cdot \bt_2| = |\cos\theta|$.  Owing to the transformation stability result from \cite{carstensen12} we know that 
\[ \min_{\ba\in \mathbb{R}^2\setminus\{\cero\}} \frac{(\ba\cdot\bv)^2 + (\ba\cdot\bu)^2}{|\ba|^2} = 1- |\bv \cdot \bu|, \]
for linearly independent unit vectors $\bv,\bu\in\mathbb{R}^2$. We use this result in our context with $\ba = \nabla v(\bx)$ for an interior point $\bx$ of $K$, $\bv = \bt_1$, $\bu = \bt_2$, giving 
\[ (1- |\cos \theta|)|\nabla v(\bx)|^2 \leq (\nabla v(\bx)\cdot \bt_1)^2 + (\nabla v(\bx)\cdot \bt_2)^2.\]  
Now we define $f_j = \nabla v \cdot \bt_j$ for $j=1,2$. An integration over $K$ leads to 
\[|v|^2_{1,K} \leq \frac{1}{1-|\cos\theta|} (\|f_1\|^2_{K} + \|f_2\|^2_{K}).\]
On the other hand, note that since $\int_A^Bf_1\ds = 0 = \int_A^Cf_2\ds$ from the assumption $v(A)=v(B)=v(C)$, we can apply the Poincar\'e--Friedrichs inequality to $f_1$ and $f_2$ (see, for example, \cite{carstensen22}). We are then left with 
\[ \|f_i\|_{K} \leq C_{\rm{PF}} h_K |f_i|_{1,K}\quad \text{for } i = 1,2.\]
Hence
\[ |v|^2_{1,K} \leq \frac{C^2_{\rm{PF}}}{1-|\cos\theta|} h_K^2(|f_1|^2_{1,K} + |f_2|^2_{1,K}) \leq \frac{2C^2_{\rm{PF}}}{1-|\cos\theta|} h_K^2|v|^2_{2,K},\]
which proves the sought bound with $C_{\rm{P}}: = C_{\rm{PF}} \sqrt{ \frac{2}{1-|\cos\theta|}}$. 
\end{proof}
\begin{lemma}[Local norm equivalence]\label{lem:norm-equiv}
For $\tilde{v}\in \tvhnc(K)$, there holds 
\[ \| \tilde{v} \|_{K}  \simeq h_K \| \mathrm{Dof}^{k,\mathrm{nc}}(\tilde{v})\|_{\ell^2}. \] 
\end{lemma}
\begin{proof} \textbf{Step 1.} 
Proceeding as in the proof of Lemma~\ref{lem:inverse}, for the DoFs we have 
\[
|\tilde{v}(z)|  \lesssim h_K^{-1}\|\tilde{v}\|_{K} + |\tilde{v}|_{1,K} 
\lesssim h_K^{-1} \|\tilde{v}\|_{K}\] 
with the inverse inequality applied to $\tilde{v}$ from Lemma~\ref{lem:inverse}. Using the Cauchy--Schwarz inequality, the trace inequality, and Lemma~\ref{lem:inverse} again,  we arrive at  
\[
\Big|\int_e \partial_{\nn}\tilde{v}\chi_{k-2}\ds \Big|  \leq \|\partial_{\nn}\tilde{v}\|_{e}h_e^{1/2}
\lesssim |\tilde{v}|_{1,K} +h_K |\tilde{v}|_{2,K}
\lesssim h_K^{-1} \|\tilde{v}\|_{K}.
\]
In addition, 
\[
\Big|\dashint_e \tilde{v}\chi_{k-3}\ds \Big|\leq h_e^{-1/2} \|\tilde{v}\|_{e} \lesssim  h_K^{-1} \| \tilde{v} \|_{K}.
\]
The Cauchy--Schwarz inequality for the cell moments proves that
\[
\Big|\dashint_K \tilde{v}\chi_{k-4}\dx\Big| \leq |K|^{-1/2}\|\tilde{v}\|_{K}  = h_K^{-1} \| \tilde{v} \|_{K},
\]
and combining these bounds together we readily obtain 
\begin{equation}\label{eq:step21}
	\| \mathrm{Dof}^{k,\mathrm{nc}}(\tilde{v})\|_{\ell^2} \lesssim h_K^{-1} \| \tilde{v} \|_{K}.
\end{equation} 

\medskip 
\noindent\textbf{Step 2.} 
On the other hand, let us consider the problem of finding $\tilde{v}_2 \in H_0^1(K) \cap H^2(K)$ such that 
\[ \Delta^2\tilde{v}_2 = \Delta^2\tilde{v} \quad \text{in $K$}; \qquad \partial_{\nn\nn}(\tilde{v}_2) = \partial_{\nn\nn}\tilde{v}\quad \text{on $\partial K$}.\]
Let $\tilde{v}_1 = \tilde{v}-\tilde{v}_2$. Then, it follows that
\[ \Delta^2\tilde{v}_1 = 0 \quad \text{in $K$}; \qquad \partial_{\nn\nn}(\tilde{v}_1) = 0 \ \text{and} \  \tilde{v}_1 = \tilde{v}\quad \text{on $\partial K$}.\]
Next we recall that for any $w\in S(K)$ (cf. \eqref{def:Sk}) we have $a^K(\tilde{v}_1,\tilde{v}_1-w) = 0$. From the proof of Lemma~\ref{lem:inverse} we 
also recall that $Q_K\tilde{v}_1$ is well-defined and 
\[Q_K\tilde{v}_1 = \tilde{v}_1 = \tilde{v} \quad \text{on $\partial K$}.\]
The triangle inequality and  Lemma~\ref{lem:poin} for $\tilde{v}_1 - Q_K \tilde{v}_1$ (applies from the definition of $Q_K$) result in
\begin{align*}
	\|\tilde{v}_1\|_{K}  \leq \|\tilde{v}_1 - Q_K\tilde{v}_1\|_{K} + \|Q_K\tilde{v}_1\|_{K} & \lesssim h_K^2 |\tilde{v}_1 - Q_K \tilde{v}_1|_{2,K}+ \|Q_K\tilde{v}_1\|_{K} \\&\lesssim h_K^2 |\tilde{v}_1 |_{2,K}+ \|Q_K\tilde{v}_1\|_{K}. 
\end{align*}
From the proof of Lemma~\ref{lem:inverse}, we can infer that $|\tilde{v}_1|_{2,K} \lesssim h_K^{-3/2}\|\tilde{v}_1\|_{\partial K}$ and $\|Q_K\tilde{v}_1\|_{K} \lesssim \|\tilde{v}_1\|_{\partial K}$. This in the previous bound results in $\|\tilde{v}_1\|_{K}  \lesssim h_K^{1/2} \|\tilde{v}_1\|_{\partial K}$.

Recall that $\Delta^2 \tilde{v}_2 = \Delta^2\tilde{v} =: g_1 \in \mathbb{P}_{k-4}(K)$ and 
$\partial_{\nn\nn}(\tilde{v}_2)|_e = \partial_{\nn\nn}\tilde{v}|_e =: g_2|_e \in \mathbb{P}_{k-2}(e)$ for all $e\in \mathcal{E}_K$. Therefore, after 
expanding $g_1 = \sum_{|\alpha|\leq k-4}g^\alpha_1m_\alpha$ and   $g_2|_e = \sum_{|\beta|\leq k-2}g^\beta_2m_\beta^e$ in terms of the scaled monomials $m_\alpha\in\M_{k-4}(K)$ and $m_\beta^e\in\M_{k-2}(e)$, 
an integration by parts provides
\begin{align*} 
	|\tilde{v}_2|^2_{2,K} &= (\Delta^2\tilde{v}_2,\tilde{v}_2)_{K} +  (\partial_{\nn\nn}(\tilde{v}_2),\partial_{\nn}\tilde{v}_2)_{\partial K} \\&= (g_1, \tilde{v})_{K} -  (g_1, \tilde{v}_1)_{K} +  (g_2, \partial_{\nn}\tilde{v})_{\partial K} - (g_2, \partial_{\nn}\tilde{v}_1)_{\partial K} \\
	& =  \sum_{|\alpha|\leq k-4}g^\alpha_1(m_\alpha,\tilde{v})_{K} + \sum_{|\beta|\leq k-2}g^\beta_2(m_\beta,\partial_{\nn}\tilde{v})_{\partial K} -  (g_1, \tilde{v}_1)_{K} - (g_2, \partial_{\nn}\tilde{v}_1)_{\partial K}.
\end{align*} 
Set the notation $\vec{g}_1 = (g_1^\alpha)_\alpha$, $\vec{g}_2 = (g_2^\beta)_\beta$ and recall from \cite[Lemma~4.1]{chen18} that $h_K\|\vec{g}_1\|_{\ell^2}\lesssim \|g_1\|_K$ and $h_K^{1/2}\|\vec{g}_2\|_{\ell^2}\lesssim \|g_2\|_{\partial K}$. Hence the Cauchy--Schwarz inequality in the previous bound and the definition of DoFs show
\begin{align*} 
	|\tilde{v}_2|^2_{2,K} &\leq \|\vec{g}_1\|_{\ell^2}|K|\|\mathrm{Dof}_K^{k,\mathrm{nc}}(\tilde{v})\|_{\ell^2} + \|\vec{g}_2\|_{\ell^2}\|\mathrm{Dof}_{\partial K}^{k,\mathrm{nc}}(\tilde{v})\|_{\ell^2} \\&\quad+ \|g_1\|_{K}  \|\tilde{v}_1\|_{K}  + \|g_2\|_{\partial K}   \|\partial_{\nn}\tilde{v}_1\|_{\partial K}\\
	&\lesssim h_K \|g_1\|_{K}\|\mathrm{Dof}_K^{k,\mathrm{nc}}(\tilde{v})\|_{\ell^2} + h_K^{-1/2} \|g_2\|_{\partial K}\|\mathrm{Dof}_{\partial K}^{k,\mathrm{nc}}(\tilde{v})\|_{\ell^2} \\
	& \quad +  \|g_1\|_{K}  \|\tilde{v}_1\|_{K}  + \|g_2\|_{\partial K} ((1+\epsilon)h_K^{1/2} |\tilde{v}_1|_{2,K} + C(\epsilon)h_K^{-3/2}\|\tilde{v}_1\|_K)
\end{align*}
with $\|\partial_{\nn}\tilde{v}_1\|_{\partial K}\lesssim h_K^{-1/2} |\tilde{v}_1|_{1,K} + h_K^{1/2} |\tilde{v}_1|_{2,K} \lesssim (1+\epsilon)h_K^{1/2} |\tilde{v}_1|_{2,K} + C(\epsilon)h_K^{-3/2}\|\tilde{v}_1\|_K$ from the trace inequality and Lemma~\ref{lem:huang32} in the last step. Note also that, thanks to \cite{chen20}, we can assert that 
\[ \|g_1\|_{K} = \| \Delta^2 \tilde{v}_2 \|_{K} \lesssim h_K^{-2} |\tilde{v}_2|_{2,K},\]
and using the inverse inequality on $g_2|_e \in \mathbb{P}_{k-2}(e)$, the trace inequality, and Lemma~\ref{lem:poin}, we are left with 
\[
\|g_2\|_{\partial K}  
\lesssim h_K^{-1}  |\tilde{v}_2|_{1,\partial K} 
\lesssim h_K^{-3/2} |\tilde{v}_2|_{1,K} + h_K^{-1/2}  |\tilde{v}_2|_{2,K}
\lesssim h_K^{-1/2}  |\tilde{v}_2|_{2,K}.
\]
Hence, all the above bounds result in 
\[ |\tilde{v}_2|_{2,K} \lesssim h_K^{-1} \|\mathrm{Dof}^{k,\mathrm{nc}}(\tilde{v})\|_{\ell^2} + h_K^{-3/2} \|\tilde{v}\|_{\partial K}.\]
We can then invoke again Lemma~\ref{lem:poin} to obtain $\|\tilde{v}_2\|_{K} \lesssim h_K^2 |\tilde{v}_2|_{2,K}$, and so 
\[ \|\tilde{v}\|_{K} \leq \|\tilde{v}_1\|_{K} + \|\tilde{v}_2\|_{K} \lesssim h_K^{1/2} \|\tilde{v}\|_{\partial K} + h_K  \|\mathrm{Dof}^{k,\mathrm{nc}}(\tilde{v})\|_{\ell^2}.\]
Since $\tilde{v}$ is a polynomial along each $e\in \mathcal{E}_K$, then standard scaling arguments imply that 
$\|\tilde{v}\|_{\partial K} \simeq h_K^{1/2} \|\mathrm{Dof}_{\partial K}^{k,\mathrm{nc}}(\tilde{v})\|_{\ell^2},$
and therefore 
\begin{equation}\label{eq:step22}
	\| \tilde{v} \|_{K} \lesssim h_K \| \mathrm{Dof}^{k,\mathrm{nc}}(\tilde{v})\|_{\ell^2}.
\end{equation} 
Finally, the desired result follows from combining the estimates \eqref{eq:step21} (from Step 1) and \eqref{eq:step22} (from Step 2).  
\end{proof}
\begin{lemma}\label{lem:lem6}
The inverse estimates and norm equivalence results hold for any $v \in  \vhnc(K)$.
\end{lemma}
\begin{proof}
Given $v\in \vhnc(K)$, construct $\tilde{v}\in \tvhnc(K)$  with $\mathrm{Dof}^{k,\mathrm{nc}}(v) = \mathrm{Dof}^{k,\mathrm{nc}}(\tilde{v})$.  Such $\tilde{v}$ can be found since the DoFs of both local VE spaces $\vhnc(K)$ and $\widetilde{V}_h^{k,\rm{nc}}(K)$ coincide. Starting from the triangle inequality
$ |v|_{2,K}\leq |v-\tilde{v}|_{2,K}+|\tilde{v}|_{2,K},$
we apply integration by parts, the Cauchy--Schwarz inequality, and the inverse inequality on $\Delta^2(v-\tilde{v}) \in \mathbb{P}_k(K)$ to obtain 
\begin{align*}
	|v-\tilde{v}|^2_{2,K}  = a^K(v-\tilde{v},v-\tilde{v}) 
	&= \int_K \Delta^2(v-\tilde{v})\,(v-\tilde{v})\dx \\
	& \leq \|\Delta^2(v-\tilde{v})\|_{K}\|v-\tilde{v}\|_{K} \lesssim h_K^{-2} | v-\tilde{v}|_{2,K} \|v-\tilde{v}\|_{K}.
\end{align*}
This bound together with the triangle inequality show that $|v-\tilde{v}|_{2,K} \lesssim h_K^{-2}(\|v\|_{K} + \|\tilde{v}\|_{K})$. Combining this last bound with Lemma~\ref{lem:inverse}-\ref{lem:norm-equiv} readily implies that 
\[ |v|_{2,K} \lesssim h_K^{-2} (\|v\|_{K} + \|\tilde{v}\|_{K}), \qquad \|\tilde{v}\|_{K} \simeq h_K \| \mathrm{Dof}^{k,\mathrm{nc}}(\tilde{v})\|_{\ell^2} = h_K \| \mathrm{Dof}^{k,\mathrm{nc}}(v)\|_{\ell^2}.\]
Then we can follow the arguments developed in the proof of Lemma~\ref{lem:inverse} to get 
\[|v|_{2,K} \lesssim h_K^{-2} \|v\|_{K} + \epsilon |v|_{2,K} ,\]
and the proof of the inverse estimate is concluded after absorbing the $\epsilon$ term on the left-hand side.
Regarding the norm equivalence result, we note that 
steps 1 and 2 from the proof of Lemma~\ref{lem:norm-equiv} apply to any $v\in \vhnc(K)$, and  decompose $v= v_1+v_2$. The proof follows analogously  only differing in the presence of the 
term $\Delta^2{v}_2 =\Delta^2{v}=: g_1 \in \mathbb{P}_k(K)$, now written as 
\[ (\Delta^2v_2,v)_{K} = \sum_{|\alpha|\leq k-4}g^\alpha_1(m_\alpha,{v})_{K} + \sum_{k-4 \leq |\alpha|\leq k}g^\alpha_1(m_\alpha,v)_{ K}.\]
The first term on the right-hand side is treated just as in Lemma~\ref{lem:norm-equiv}. For the second term we can apply the definition of the space  $\vhnc(K)$ and the Cauchy--Schwarz inequality to obtain the estimate 
\begin{align*} \sum_{k-4\leq |\alpha|\leq k}g^\alpha_1(m_\alpha,v)_{ K} &=  \sum_{k-4 \leq |\alpha|\leq k}g^\alpha_1(m_\alpha,\pd_kv)_{K} \\&\leq h_K^{-1}\|g_1\|_{K}\|m_\alpha\|_K\|\pd_kv\|_{K}\approx\|g_1\|_{K}\|\pd_kv\|_{K}\end{align*}
with the observation $\|m_\alpha\|_K\approx h_K$ in the last step. Since $\pd_kv$ is uniquely determined by the DoFs of $v$ and $\mathrm{Dof}^{k,\mathrm{nc}}(v) = \mathrm{Dof}^{k,\mathrm{nc}}(\tilde{v})$, we can conclude that 
$ \pd_kv = \pd_k\tilde{v}$. This and the triangle inequality show $\|\pd_kv\|_{K} = \|\pd_k\tilde{v}\|_{K} \leq \|\pd_k\tilde{v} - \tilde{v}\|_{K} + \| \tilde{v}\|_{K}$. The Poincar\'e--Friedrichs inequality and the inverse inequality provide $\|\pd_k\tilde{v} - \tilde{v}\|_{K} \lesssim \|\tilde{v}\|_{K}$. These bounds together with Lemma~\ref{lem:norm-equiv}  establish 
\begin{align*}
	\|\pd_kv\|_{K} & \lesssim \| \tilde{v}\|_{K}  \lesssim h_K \|\mathrm{Dof}^{k,\mathrm{nc}}(\tilde{v})\|_{\ell^2} = h_K \|\mathrm{Dof}^{k,\mathrm{nc}}({v})\|_{\ell^2},
\end{align*}
and this observation concludes the proof of norm equivalence. 
\end{proof} 
\subsection{Modified companion map}\label{subsec:modified_comp}
In this subsection, we consider the lowest-order ($k=2$) nonconforming VE space $V_h^{2,\mathrm{nc}}$ and the aim is to modify the companion operator $J_1v_h$ for $v_h\in V_h^{2,\mathrm{nc}}$ from Theorem~\ref{thm:J1} so that the new companion $J_2^*v_h$ satisfies the $H^1$-orthogonality $\nabla (v_h-J_2^*v_h)\perp (\p_0(\cT_h))^2$ in addition to the $H^2$- and $L^2$-orthogonalities established in Theorem~\ref{4.2.b}-\ref{4.2.c}.

If  $e\in\e^i$ is an interior edge, we assume that it's shared by two triangles $T^+\subset K^+$ and $T^-\subset K^-$ inside two neighbouring polygons $K^+$ and $K^-$, and set $\omega_e:=K^+\cup K^-$, and if $e\in\e^b$ is a boundary edge, we assume that it only belongs to a triangle $T^+\subset K^+$ and set $\omega_e:=K^+$. Let $\psi_e,\phi_e\in H^2_0(T^+\cup T^-)$ be two edge bubble-functions from \cite{brenner2010posteriori,chen2022} satisfying the following properties:
\begin{itemize}
\item $\dashint_e\psi_e\;\ds=1, |\psi_e|_{2,T^{\pm}}\approx h_e^{-1}$,
\item $\phi_e\equiv 0$ on $\partial T^{\pm}$, $\int_e\partial_{\nn} \phi_e\ds = 1, |\phi_e|_{2,T^{\pm}}\approx h_e^{-1}$.
\end{itemize} 
\textbf{Step 1}. Given $J_1v_h$ and the bubble-function $\psi_e$, define
\begin{align}
J_1^*v_h := J_1v_h + \sum_{e\in\e}\Big(\dashint_e(v_h-J_1v_h)\ds\Big)\psi_e\in V. 
\end{align}
Observe from the definition of $J_1$ and $\psi_e$ that $J_1^*v_h(z) = v_h(z)$ for any $z\in \mathcal{V}$ and $\dashint_e (v_h-J_1^*v_h)\ds=0$ for any $e\in\e$.  The Cauchy--Schwarz inequality and the scaling of $\psi_e$ from above show that
\begin{align*}
&\Big|\Big(\dashint_e(v_h-J_1v_h)\ds\Big)\Big|\|\psi_e\|_{\omega_e}\lesssim h_e^{-3/2}\|v_h-J_1v_h\|_e\\&\qquad\lesssim h_e^{-2}\|v_h-J_1v_h\|_{\omega_e}+h_e^{-1}|v_h-J_1v_h|_{1,\omega_e}\lesssim |v_h-J_1v_h|_{2,\omega_e}
\end{align*}
with the trace inequality and Theorem~\ref{thm:J1} in the last two steps. This proves that $|v_h-J_1^*v_h|_{2,h}\lesssim |v_h-J_1v_h|_{2,h}$. Similarly the scaling $|\psi_e|_{1,T^{\pm}}\approx 1$ and Theorem~\ref{4.1.c} result in $|v_h-J_1^*v_h|_{1,h}\lesssim h|v_h-J_1v_h|_{2,h}$. 

\smallskip 
\noindent\textbf{Step 2}. Given $J_1^*v_h$ and the bubble-function $\phi_e$, define
\begin{align}
J_1^{**}v_h := J_1^*v_h + \sum_{e\in\e}\Big(\int_e\partial_{\nn}(v_h-J_1^*v_h)\ds\Big)\phi_e\in V. 
\end{align}
Since $\phi_e|_e\equiv 0$, we have $J_1^{**}v_h(z) = v_h(z)$ for any $z\in \mathcal{V}$ and $\dashint_e (v_h-J_1^{**}v_h)\ds=0$ for any $e\in\e$. Note from  $\int_e\partial_{\nn}\phi_e\ds=1$ that $\int_e\partial_{\nn}(v_h-J_1^{**}v_h)\ds=0$. Again as in Step 1, it is easy to prove that 
\begin{align*}\Big|\Big(\dashint_e\partial_{\nn}(v_h-J_1^*v_h)\ds\Big)\Big|\|\phi_e\|_{\omega_e}\lesssim |v_h-J_1^*v_h|_{2,\omega_e},\end{align*}
and consequently,
\[|v_h-J_1^{**}v_h|_{2,h}\lesssim |v_h-J_1^*v_h|_{2,h}\lesssim |v_h-J_1v_h|_{2,h}.\]
\textbf{Step 3}. Next we construct the operator $J_2^*$ and the design employs the tools from the construction of $J_2$.   Recall the element bubble-function $b_h|_K=b_K\in H^2_0(K)$ and suppose $v_2\in\p_2(K)$ is the Riesz representative of the linear functional $\p_2(K)\to\mathbb{R},$ defined by, $w_2\mapsto(v_h-J_1^{**}v_h,w_2)_K$, for $w_2\in\p_2(K)$ in the Hilbert space $\p_2(K)$ endowed with the weighted scalar product $(b_K\bullet,\bullet)_K$. Given $v_h\in V_h^{2,\mathrm{nc}}$,  the function $\tilde{v}_2\in\p_2(\cT_h)$ with $\tilde{v}_2|_K:=v_2$   satisfy $
(b_h\tilde{v}_2,w_2)_\Omega=(v_h-J_1^{**}v_h,w_2)_\Omega$ for all $w_2\in\p_2(\cT_h) $
and define
\begin{align}
J_2^*v_h:=J_1^{**}v_h+b_h\tilde{v}_2\in V.\label{def:J2*}
\end{align}
\begin{theorem}
\label{thm:J2*}
The  modified conforming companion operator $J_2^*$ satisfies the following properties.
\begin{enumerate}[before=\leavevmode,label=\upshape(\alph*),ref=\thetheorem (\alph*)]
	\item\label{6.1.a} $\mathrm{dof}_j^{2,\mathrm{nc}} (J_2^*v_h) = \mathrm{dof}_j^{2,\mathrm{nc}} (v_h)$ for all $j=1,\dots,N_2^{2,\mathrm{nc}}$,
	\item\label{6.1.b} $\nabla^2(v_h-J_2^*v_h)\perp(\p_0(\cT_h))^{2\times 2}$ in $(L^2(\Omega))^{2\times 2}$,
	\item\label{6.1.c} $\nabla(v_h-J_2^*v_h)\perp(\p_0(\cT_h))^2$ in $(L^2(\Omega))^2$,
	\item\label{6.1.d} $v_h-J_2^*v_h\perp \p_2(\cT_h)$ in $L^2(\Omega)$,
	\item\label{6.1.e} $\displaystyle \sum_{j=0}^2 h^{j-2}|v_h-J_2^*v_h|_{j,h}\lesssim \inf_{\chi\in \p_2(\cT_h)}|v_h-\chi|_{2,h}+\inf_{v\in V}|v_h-v|_{2,h}.$
\end{enumerate}
\end{theorem}
\begin{proof}
We provide only the proof of modified property \ref{6.1.c} and the  remaining properties  follow analogously as in the proof of Theorem~\ref{thm:J2}. Since $b_K =0 $ on $\partial K$, it results from that definition of $J_2^*$ and $J_1^{**}$ that $\dashint_e J_2^*v_h\ds= \dashint_eJ_1^{**}v_h\ds=\dashint_e v_h\ds$. Hence, for $\vec{\chi}\in(\p_0(K))^2$ and for any $K\in\cT_h$, an integration by parts   shows 
\[(\nabla(v_h-J_2^*v_h),\vec{\chi})_K=-(v_h-J_2^*v_h,\text{div}(\vec{\chi}))_K+(v_h-J_2^*v_h,\vec{\chi}\cdot\nn_K)_{\partial K}=0,\]
and therefore it concludes the proof of Theorem~\ref{6.1.c}.
\end{proof}
\subsection{Reliability}\label{subsec:rel}
Recall $(u_h,p_h)$ is the nonconforming VE solution to the discrete problem and the notation $  T(\bullet)=\partial_{\nn}(\Delta\bullet+\partial_{\btau\btau}\bullet)$. We proceed to define the local contributions for the estimator 
\begin{align*}
\eta_{1,K}^2&:=h_K^4\|\tilde{f}-\Pi_k\tilde{f}\|_K^2+h_K^4\|\tilde{f}-\Delta^2\pd_ku_h-\Pi_ku_h-\alpha\nabla\cdot\Pi_{\ell-1}\nabla p_h\|_K^2,\\
\eta_{2,K}^2&:=h_K^2\|\tilde{g}-\Pi_\ell\tilde{g}\|_K^2+h_K^2\|\tilde{g}+\gamma\nabla\cdot\Pi_{\ell-1}\nabla p_h-\beta\Pi_\ell p_h+\alpha\nabla\cdot\Pi_{k-1}\nabla u_h\|^2_K,\\
\eta_{3,K}^2&:=\sum_{e\in\e^i_K\cup\e^s_K}h_e\|[\partial_{\nn\nn}(\pd_ku_h)]\|^2_e,\\
\eta_{4,K}^2&:=\sum_{e\in\e^i_K}h_e^3\|[T(\pd_k u_h)+\alpha\Pi_{\ell-1}\nabla p_h\cdot\nn]\|^2_e,\\
\eta_{5,K}^2&:=\sum_{e\in\e^i_K\cup\e^c_K}h_e\|[\alpha\Pi_{k-1}\nabla u_h\cdot\nn+\gamma\Pi_{\ell-1}\nabla p_h\cdot\nn]\|^2_e,\\
\eta_{6,K}^2&:=(1+\alpha^{1/2}h_K+h_K^2)^2S_{\nabla^2}^K((1-\pd_k)u_h,(1-\pd_k)u_h)\\&\qquad+S_{\nabla}^K((1-\pg_\ell)p_h,(1-\pg_\ell)p_h)+S_{2,0}^K((1-\Pi_\ell)p_h,(1-\Pi_\ell)p_h),\\
\eta_{7,K}^2&:=\alpha\|\text{Dof}^{k,\rm{nc}}(u_h-\pg_\ell u_h)\|^2_{\ell^2},\\
\eta_{8,K}^2&:=\sum_{e\in\e_K}h_e^{-1}(\|[\nabla\pd_ku_h]\|^2_e+\|[\Pi_\ell p_h]\|^2_e).
\end{align*}
Denote $\eta_i^2:= \sum_{K\in\cT_h} \eta_{i,K}^2$ for $i=1,\dots,8$ and the following theorem shows that the sum of these contributions form an upper bound for the  error in energy norm.
\begin{theorem}\label{thm:reliability}
With the aforementioned notation, there holds, for $k=2$ and $\ell=1$, 
\[\|\vec{\bu}-\vec{\bu}_h\|_{\bH_\epsilon^h}^2 \lesssim \eta^2:=\sum_{i=1}^8\eta_i^2.\]
\end{theorem} 
\begin{proof}
Let $J\vec{\bu}_h:= (J_2^*u_h,J_4p_h)$ and $\vec{\be}:= \vec{\bu}-J\vec{\bu}_h:=(e^u,e^p)\in \bH_\epsilon$ with $\vec{\be}_I:=(e_I^u,e_I^p)$. Even though $J_2^*$  is constructed  for the lowest-order case ($k=2$), we prefer to write the proof with the  notation $k$ and $\ell$ to point out the challenges for general values.  The coercivity of the continuous bilinear form $\cA$ leads to 
\begin{align}
	\|\vec{\be}\|^2_{\bH_\epsilon}&\lesssim\cA(\vec{\bu},\vec{\be})-\cA(J\vec{\bu}_h,\vec{\be})=\cF(\vec{\be})-\cF_h(\vec{\be}_I)+\cA_h(\vec{\bu}_h,\vec{\be}_I)-\cA(J\vec{\bu}_h,\vec{\be})\nonumber.
\end{align}
The identities $(\Pi_k\tilde{f},e_I^u-J_2^*e_I^u)_\Omega=0=(\Pi_\ell\tilde{g},e_I^p-J_4e_I^p)_\Omega$ from  Theorems~\ref{thm:J2*}.d and \ref{thm:J4}.c, and the $L^2$-orthogonality of $\Pi_k$ and $\Pi_\ell$ show
\begin{align*}
	&\cF(\vec{\be})-\cF_h(\vec{\be}_I)=(\tilde{f},e^u)_\Omega-(\Pi_k\tilde{f},J_2^*e_I^u)_\Omega+(\tilde{g},e^p)_\Omega-(\Pi_\ell\tilde{g},J_4e_I^p)_\Omega\nonumber\\&=(\tilde{f},v)_\Omega+(\tilde{f}-\Pi_k\tilde{f},J_2^*e_I^u-\pd_ke_I^u)_\Omega+(\tilde{g},q)_\Omega+(\tilde{g}-\Pi_\ell\tilde{g},J_4e_I^p-\pg_\ell e_I^p)_\Omega
\end{align*}
with $v=e^u-J_2^*e_I^u$ and $q=e^p-J_4e_I^p$. The definitions of $\cA_h$ and $\cA$ provide
\begin{align*}
	&\cA_h(\vec{\bu}_h,\vec{\be}_I)-\cA(J\vec{\bu}_h,\vec{\be})=(a_1^h(u_h,e_I^u)-a_1(J_2^*u_h,e^u))+(a_2(J_4p_h,e^u)-a_2^h(p_h,e_I^u))\nonumber\\&\quad+(a_2^h(e_I^p,u_h)-a_2(e^p,J_2^*u_h))+(a_3^h(p_h,e_I^p)-a_3(J_4p_h,e^p))=:T_1+T_2+T_3+T_4.
\end{align*}
Theorem~\ref{6.1.b}-\ref{6.1.d} imply $(\Pi_{k-2}\nabla^2u_h,\Pi_{k-2}\nabla^2e_I^u)_\Omega=(\nabla^2_\pw\pd_ku_h,\nabla^2J_2^*e_I^u)_\Omega$ \\ and $(\Pi_ku_h,\Pi_ke_I^u)_\Omega=(\Pi_ku_h,J_2^*e_I^u)_\Omega$. This  results in 
\begin{align}
	T_1&=-(\Pi_ku_h,v)_\Omega+(\Pi_ku_h-J_2^*u_h,e^u)_\Omega-(\nabla^2_\pw\pd_ku_h,\nabla^2_\pw v)_\Omega\nonumber\\&\quad+(\nabla^2_\pw(\pd_ku_h-J_2^*u_h),\nabla^2e^u)_\Omega+S_{1,0}((1-\Pi_k)u_h,(1-\Pi_k)e_I^u)\nonumber\\&\quad+S_{\nabla^2}((1-\pd_k)u_h,(1-\pd_k)e_I^u). \nonumber
\end{align}
An integration by parts  and the fact that $v\in V$, lead to
\begin{align*}-(\nabla^2_\pw\pd_ku_h,\nabla^2_\pw v)_\Omega&=-\sum_{K\in\cT_h}(\Delta^2\pd_ku_h,v)_K-\sum_{e\in\e^i\cup \e^s}([\partial_{\nn\nn}(\pd_ku_h)],\partial_{\nn}v)_e\nonumber\\&\quad-\sum_{e\in\e^i}( [T(\pd_ku_h)],v)_e.\end{align*}
This simplifies $T_1$ to
\begin{align*}
	T_1&=\sum_{K\in\cT_h}\Big((-\Pi_ku_h-\Delta^2\pd_ku_h,v)_K+(\Pi_ku_h-J_2^*u_h,e^u)_K\nonumber\\&\qquad+(\nabla^2(\pd_ku_h-J_2^*u_h),\nabla^2 e^u)_K+S_{1,0}^K((1-\Pi_k)u_h,(1-\Pi_k)e_I^u)\nonumber\\&\qquad+S_{\nabla^2}^K((1-\pd_k)u_h,(1-\pd_k)e_I^u)\Big)-\sum_{e\in\e^i\cup \e^s}([\partial_{\nn\nn}(\pd_ku_h)],\partial_{\nn}v)_e\nonumber\\&\qquad-\sum_{e\in\e^i}( [T(\pd_ku_h)],v)_e.
\end{align*}
Theorem~\ref{6.1.c} and the $L^2$-orthogonality of $\Pi_{k-1}$ imply $(\Pi_{\ell-1}\nabla p_h,\nabla J_2^*e_I^u-\Pi_{k-1}\nabla e_I^u)_\Omega=0$. This and an integration by parts show
\begin{align}
	\alpha^{-1}T_2&=(\nabla J_4p_h-\Pi_{\ell-1}\nabla p_h,\nabla e^u)_\Omega-\sum_{K\in\cT_h}(\nabla\cdot\Pi_{\ell-1}\nabla p_h,v)_K+\sum_{e\in\e^i}([\Pi_{\ell-1}\nabla p_h\cdot\nn_e],v)_e.\label{6.7}
\end{align}
Theorem~\ref{4.4.b}, the $L^2$-orthogonality of $\Pi_{\ell-1}$, and again an integration by parts prove that
\begin{align}
	\alpha^{-1}T_3&=\sum_{K\in\cT_h}\Big((\Pi_{\ell-1}\nabla e_I^p-\nabla J_4e_I^p,\Pi_{k-1}\nabla u_h-\Pi_{\ell-1}\nabla u_h)_K+(q,\nabla\cdot \Pi_{k-1}\nabla u_h)_K\nonumber\\&\quad+(\nabla e^p,\Pi_{k-1}\nabla u_h-\nabla J_2^*u_h)_K\Big)-\sum_{e\in\e^i\cup\e^c}(q,[\Pi_{k-1}\nabla u_h\cdot\nn_e])_e.
\end{align}
The identities $(\Pi_{\ell-1}\nabla p_h, \Pi_{\ell-1}\nabla e_I^p)_\Omega= (\Pi_{\ell-1}\nabla p_h, \nabla J_4 e_I^p)_\Omega$ and $(\Pi_\ell p_h,\Pi_\ell e_I^p)_\Omega=(\Pi_\ell p_h,J_4 e_I^p)_\Omega$ follow from Theorem~\ref{4.4.b}-\ref{4.4.d}. This in the first step and  an integration by parts in the next step lead to
\begin{align*}
	T_4&=\beta((\Pi_\ell p_h-J_4p_h,e^p)_\Omega-(\Pi_\ell p_h,q)_\Omega)+\gamma((\Pi_{\ell-1}\nabla p_h-\nabla J_4 p_h,\nabla e^p)_\Omega\nonumber\\&\quad-(\Pi_{\ell-1}\nabla p_h,\nabla q)_\Omega)+S_{2,0}((1-\Pi_\ell)p_h,(1-\Pi_\ell)e_I^p)\nonumber\\&\quad+S_{\nabla}((1-\pg_\ell)p_h,(1-\pg_\ell)e_I^p)\nonumber\\&=\sum_{K\in\cT_h}\Big(\beta(\Pi_\ell p_h-J_4p_h,e^p)_K+\gamma(\Pi_{\ell-1}\nabla p_h-\nabla J_4 p_h,\nabla e^p)_K\nonumber\\&\quad+(\gamma\nabla\cdot\Pi_{\ell-1}\nabla p_h-\beta \Pi_\ell p_h,q)_K+S_{2,0}^K((1-\Pi_\ell)p_h,(1-\Pi_\ell)e_I^p)\nonumber\\&\quad+S_{\nabla}^K((1-\pg_\ell)p_h,(1-\pg_\ell)e_I^p)\Big)-\sum_{e\in\e^i\cup\e^c}\gamma([\Pi_{\ell-1}\nabla p_h\cdot\nn_e],q)_e.
\end{align*}
The rearrangement of the terms results in 
\begin{align}
	\|\vec{\be}\|^2_{\bH_\epsilon}&\lesssim  T_5+\dots+T_{10},\label{eq:6.9}
\end{align}
where
\begin{align*}
	T_5&:=(\tilde{f}-\Pi_k\tilde{f},J_2^*e_I^u-\pd_ke_I^u)_\Omega+(\tilde{f}-\Pi_kv_h-\Delta^2\pd_k u_h-\alpha \nabla\cdot\Pi_{\ell-1}\nabla p_h,v)_\Omega,\\T_6&:=(\tilde{g}-\Pi_\ell\tilde{g},J_4e_I^p-\pg_\ell e_I^p)_\Omega+(\tilde{g}+\gamma\nabla\cdot\Pi_{\ell-1}\nabla p_h-\beta \Pi_\ell p_h+\alpha \nabla\cdot \Pi_{k-1}\nabla u_h,q)_\Omega,\\T_7&:=(\Pi_ku_h-J_2^*u_h,e^u)_\Omega+(\nabla^2_\pw(\pd_ku_h-J_2^*u_h),\nabla^2e^u)_\Omega\\&\quad+\alpha(\nabla J_4p_h-\Pi_{\ell-1}\nabla p_h,\nabla e^u)_\Omega+\alpha(\Pi_{\ell-1}\nabla e_I^p-\nabla J_4e_I^p,\Pi_{k-1}\nabla u_h-\Pi_{\ell-1}\nabla u_h)_K\\&\quad+\alpha(\nabla e^p,\Pi_{k-1}\nabla u_h-\nabla J_2^*u_h)_\Omega+\beta(\Pi_\ell p_h-J_4p_h,e^p)_\Omega\\&\quad+\gamma(\Pi_{\ell-1}\nabla p_h-\nabla J_4 p_h,\nabla e^p)_\Omega,\\T_8&:=S_{1,0}((1-\Pi_k)u_h,(1-\Pi_k)e_I^u)+S_{\nabla^2}((1-\pd_k)u_h,(1-\pd_k)e_I^u)\nonumber\\&\quad+S_{2,0}((1-\Pi_\ell)p_h,(1-\Pi_\ell)e_I^p)+S_{\nabla}((1-\pg_\ell)p_h,(1-\pg_\ell)e_I^p),\\T_9&:=-\sum_{e\in\e^i\cup\e^c}(q,[\alpha\Pi_{k-1}\nabla u_h\cdot\nn_e+\gamma\Pi_{\ell-1}\nabla p_h\cdot\nn_e])_e,\\T_{10}&:=-\sum_{e\in\e^i\cup\e^s}([\partial_{\nn\nn}(\pd_ku_h)],\partial_{\nn}v)_e-\sum_{e\in\e^i}( [T(\pd_ku_h)+\alpha \Pi_{\ell-1}\nabla p_h\cdot\nn_e],v)_e.
\end{align*}
The Poincar\'e--Friedrichs inequality implies that $h_P^{-2}\|J_2^*e_I^u-\pd_ke_I^u\|_{L^2(K)}\lesssim |J_2^*e_I^u-\pd_ke_I^u|_{2,K}$ and $h_P^{-2}\|v\|_{L^2(K)}\lesssim |v|_{2,K}$. Then the triangle inequality $|J_2^*e_I^u-\pd_ke_I^u|_{2,K}$\\$\leq |J_2^*e_I^u-e_I^u|_{2,K}+|e_I^u-\pd_ke_I^u|_{2,K}$ and $|v|_{2,K}\leq |e^u-e_I^u|_{2,K}+|e_I^u-J_2^*e_I^u|_{2,K}$ followed by propositions~\ref{prop:inter_v} and \ref{prop:poly}, and Theorem~\ref{6.1.e} show
\begin{align}
	T_5\lesssim \Big(\sum_{K\in\cT_h}\eta_{1,K}\Big) |e^u|_{2,\Omega}.
\end{align}
Similarly we can prove that $T_6\lesssim \Big(\sum_{K\in\cT_h}\eta_{2,K}\Big) |e^p|_{1,\Omega}\leq \Big(\sum_{K\in\cT_h}\eta_{2,K}\Big) \gamma^{1/2}|e^p|_{1,\Omega}$. Then we proceed to rewrite the terms in $T_7$ using the $L^2$-orthogonality $\Pi_k,\Pi_\ell$ and Theorem~\ref{6.1.d}-\ref{4.4.c} as 
\begin{align*}
	(\Pi_ku_h-J_2^*u_h,e^u)_\Omega&=(\Pi_ku_h-J_2^*u_h,e^u-\Pi_ke^u)_\Omega=(\pd_ku_h-J_2^*u_h,e^u-\Pi_ke^u)_\Omega, \\
	(\Pi_\ell p_h-J_4p_h,e^p)_\Omega &= (\Pi_\ell p_h-J_4p_h,e^p-\Pi_\ell e^p)_\Omega=(\pg_\ell p_h-J_4p_h,e^p-\Pi_\ell e^p)_\Omega.
\end{align*}
Then we combine Cauchy--Schwarz  and triangle inequalities, which  results in
\begin{align}
	T_7&\lesssim (\|u_h-\pd_ku_h\|_{\Omega}+|u_h-\pd_ku_h|_{2,h}+\|u_h-J_2^*u_h\|_{\Omega}+|u_h-J_2^*u_h|_{2,h}\nonumber\\&\quad+|u_h-\pg_\ell u_h|_{1,h}+\|p_h-\pg_\ell p_h\|_{\Omega}+|p_h-\pg_\ell p_h|_{1,h}+\|p_h-J_4p_h\|_{\Omega}\nonumber\\&\quad+|p_h-J_4p_h|_{1,h})\|\vec{\be}\|_{\bH_\epsilon}\nonumber\\&\lesssim \sum_{K\in\cT_h}\Big( (1+\alpha^{1/2}h_K+h_K^2)|u_h-\pd_ku_h|_{2,K}+\alpha^{1/2}|u_h-\pg_\ell u_h|_{1,h}\nonumber\\&\qquad+\beta^{1/2}\|p_h-\pg_\ell p_h\|_{K}+\gamma^{1/2}|p_h-\pg_\ell p_h|_{1,K}+\eta_{8,K}\Big)\|\vec{\be}\|_{\bH_\epsilon}\nonumber\\&\lesssim \sum_{K\in\cT_h}(\eta_{6,K}+\eta_{7,K}+\eta_{8,K})\|\vec{\be}\|_{\bH_\epsilon}.\nonumber
\end{align}
The second step follows from the Poincar\'e--Friedrichs inequality and Theorem~\ref{thm:J1}-\ref{thm:J4}, and the last step from \eqref{s1}-\eqref{s3} and the equivalence $|u_h-\pg_\ell u_h|_{1,K}\approx \|\mathrm{Dof}^{k,\mathrm{nc}}(u_h-\pg_\ell u_h)\|_{\ell^2}$ (see \cite{huang21} for a proof). Cauchy--Schwarz inequalities for  inner products and \eqref{s2} lead to
\begin{align*}
		T_8&\lesssim \Big(\sum_{K\in\cT_h}\|u_h-\Pi_ku_h\|_{L^2(K)}+S_{\nabla^2}^{1/2}((1-\pd_k)u_h,(1-\pd_k)u_h)\\&\quad+S_{2,0}^{1/2}((1-\Pi_\ell)p_h,(1-\Pi_\ell)p_h)+S_{\nabla}^{1/2}((1-\pg_\ell)p_h,(1-\pg_\ell)p_h)\Big)\|\vec{\be}\|_{\bH_\epsilon}\\&\lesssim \Big(\sum_{K\in\cT_h}\eta_{6,K}\Big)\|\vec{\be}\|_{\bH_\epsilon},
\end{align*}
with $\|u_h-\Pi_ku_h\|_{K}\leq \|u_h-\pd_ku_h\|_{K}\lesssim h_K^2|u_h-\pd_ku_h|_{2,K}$ 
followed by \eqref{s1} in the last estimate. The trace inequality shows $\|q\|_{e}\lesssim h_e^{-1/2}\|q\|_{\omega_e}+h_e^{1/2}|q|_{1,\omega_e}\lesssim h_e^{1/2}|q|_{1,\omega_e}$ with the Poincar\'e--Friedrichs inequality in the last bound. This and the Cauchy--Schwarz inequality prove that 
\begin{align*}
	T_9\lesssim \Big(\sum_{K\in\cT_h}\eta_{5,K}\Big)\gamma^{1/2}|e^p|_{1,\Omega}.
\end{align*}
Finally, analogous arguments as those used in $T_9$ show that $T_{10}\lesssim \Big(\sum_{K\in\cT_h}(\eta_{3,K}+\eta_{4,K})\Big)|e^u|_{2,\Omega}$. The previous bounds in \eqref{eq:6.9} conclude the proof. 
\end{proof}
\subsection{Efficiency}\label{subsec:eff}
\begin{theorem}[Efficiency up to stabilisation and data oscillation] Under the assumption $\ell\leq k\leq \ell+2$, the local error estimators are bounded above as follows:
\begin{subequations}	\begin{align}
		\eta_{1,K}&\lesssim \|u-u_h\|_K+|u-u_h|_{2,K}+\gamma^{1/2}|p-p_h|_{1,K}+\eta_{6,K}+\mathrm{osc}_2(\tilde{f},K),\label{eff1}\\
		\eta_{2,K}&\lesssim |u-u_h|_{2,K}+\beta\|p-p_h\|_K+\gamma|p-p_h|_{1,K}+(\beta^{1/2}+\gamma^{1/2})\eta_{6,K}+\mathrm{osc}_1(\tilde{g},K),\label{eff2}\\
		\eta_{3,K} &\lesssim \sum_{e\in\e_K}\sum_{K'\in\omega_e}\Big(\|u-u_h\|_{K'}+|u-u_h|_{2,K'}+|p-p_h|_{1,K'}+\eta_{6,K'}+\mathrm{osc}_2(\tilde{f},K')\Big),\label{eff3}\\
		\eta_{4,K} &\lesssim \sum_{e\in\e_K}\sum_{K'\in\omega_e}\Big(\|u-u_h\|_{K'}+|u-u_h|_{2,K'}+|p-p_h|_{1,K'}+\eta_{6,K'}+\mathrm{osc}_2(\tilde{f},K')\Big),\label{eff4}\\
		\eta_{5,K}&\lesssim \sum_{e\in\e_k}\sum_{K'\in\omega_e}\Big(|u-u_h|_{1,K'}+\beta\|p-p_h\|_{K'}+\gamma|p-p_h|_{1,K'}+(\beta^{1/2}+\gamma^{1/2})\eta_{6,K'}\nonumber\\&\qquad\qquad+\mathrm{osc}_1(\tilde{g},K')\Big),\label{eff5}\\
		\eta_{7,K}&\lesssim  |u-u_h|_{1,K}+|u-\pg_\ell u|_{1,K},\label{eff6}\\
		\eta_{8,K}&\lesssim |u-u_h|_{2,K}+|p-p_h|_{1,K}+\eta_{6,K}.\label{eff7}
\end{align}\end{subequations}
\end{theorem}
\begin{proof}
Recall the element bubble-function $b_{2,K}:=b_K\in H^2_0(K)$ supported in $K\in\cT_h$ and $\ell\leq k$. Let \[v_k:=\Pi_k\tilde{f}-\Delta^2\pd_ku_h-\Pi_ku_h-\alpha\nabla\cdot\Pi_{\ell-1}\nabla p_h\in\p_k(K)\;\text{and}\; v:=v_kb_{2,K}\in H^2_0(\Omega)\subset V.\]
This in the first equation of the continuous problem \eqref{eq:weak:a}, and $a^K(\pd_ku_h,v) = (\Delta^2\pd_ku_h,v)_K$ and $(\Pi_{\ell-1}\nabla p_h,\nabla v)_K = -(\nabla \cdot \Pi_{\ell-1}\nabla p_h,v)_K$ from an integration by parts lead to 
\begin{align*}
	&(u-\Pi_ku_h,v)_K+a^K(u-\pd_ku_h,v)-\alpha(\nabla p- \Pi_{\ell-1}\nabla p_h,\nabla v)_K \\&\quad= (\tilde{f}-\Pi_k\tilde{f},v)_K + (v_k,v)_K.
\end{align*}
Hence $v=v_kb_{2,K}$, the inequalities \eqref{bubble-inverse}, and $\alpha\leq 1\leq \gamma$ show that
\begin{align*}
	h_K^2\|v_k\|_K&\lesssim h_K^2\|\tilde{f}-\Pi_k\tilde{f}\|_K+h_K^2\|u-\Pi_ku_h\|_K+|u-\pd_ku_h|_{2,K}\\&\quad+  \gamma^{1/2}h_K\|\nabla p - \Pi_{\ell-1}\nabla p_h\|_K\\&\lesssim \|u-u_h\|_K+ |u-u_h|_{2,K}+\gamma^{1/2}|p-p_h|_{1,K}+\eta_{6,K}+\mathrm{osc}_2(\tilde{f},K)
\end{align*}
with triangle inequalities and \eqref{s2}-\eqref{s4} in the last estimate. This and the triangle inequality $\eta_{1,K}\leq \mathrm{osc}_2(\tilde{f},K)+h_K^2\|v_k\|_K$ conclude the proof of \eqref{eff1}. The  bubble-function $b_{1,K}\in H^1_0(K)$ supported in $K$ is constructed as in \cite{carstensen22}   and it satisfies, for any $\chi\in \p_{\ell}(K)$, that
\begin{align}
	\|\chi\|_{K}\lesssim&\sum_{m=0}^1 h_K^m|b_{1,K}\chi|_{m,K}\lesssim \|\chi\|_{K}.\label{bubble-inverse-1}
\end{align}
Let
$q_\ell:=\Pi_\ell\tilde{g}+\gamma\nabla\cdot\Pi_{\ell-1}\nabla p_h-\beta\Pi_\ell p_h+\alpha\nabla\cdot\Pi_{k-1}\nabla u_h$ and $q:=q_\ell b_{1,K}.$
This in the second equation of the continuous problem \eqref{eq:weak:b}, and $(\nabla q,\Pi_{k-1}\nabla u_h)_K = -(q,\nabla\cdot \Pi_{k-1}\nabla u_h)_K$ and $(\Pi_{\ell-1}\nabla p_h,\nabla q)_K= -(\nabla\cdot \Pi_{\ell-1}\nabla p_h, q)_K$ show
\begin{align*}
	&\beta(p-\Pi_\ell p_h,q)_K+\alpha( q,\nabla u - \Pi_{k-1}\nabla u_h)_K+\gamma(\nabla p-\Pi_{\ell-1}\nabla p_h,\nabla q)_K \\&= (\tilde{g}-\Pi_{\ell}\tilde{g},q)_K+(q_\ell,q)_K.
\end{align*}
Hence \eqref{bubble-inverse-1} in the above equation allows us to assert that 
\begin{align*}
	&h_K\|q_\ell\|_K\lesssim |u-u_h|_{2,K}+\beta\|p-p_h\|_K+\gamma |p-p_h|_{1,K}+|u_h-\pd_k u_h|_{2,K}\\&\qquad+\gamma|p_h-\pg_\ell p_h|_{1,K}+\mathrm{osc}_1(\tilde{g},K)\\&\quad\lesssim |u-u_h|_{2,K}+\beta\|p-p_h\|_K+\gamma |p-p_h|_{1,K}+(\beta^{1/2}+\gamma^{1/2})\eta_{6,K}+\mathrm{osc}_1(\tilde{g},K).
\end{align*}
This concludes the proof of \eqref{eff2}. It follows from \cite{chen2022} that  $v:=\phi_e[\partial_{\nn\nn}(\pd_ku_h)]$ satisfies  the first inequality
\begin{align}
	\|[\partial_{\nn\nn}(\pd_ku_h)]\|_e^2&\lesssim ([\partial_{\nn\nn}(\pd_ku_h)],\partial_{\nn}v)_e=a^{\omega_e}(\pd_ku_h,v)-(\Delta^2\pd_ku_h,v)_{\omega_e}\nonumber\\&\hspace{-0.5cm}=a^{\omega_e}(\pd_ku_h-u,v)+(\tilde{f}-\Pi_k\tilde{f},v)_{\omega_e}+(v_k,v)_{\omega_e}+(\Pi_ku_h-u,v)_{\omega_e}\nonumber\\&\quad+\alpha(\nabla p-\Pi_{\ell-1}\nabla p_h,\nabla v)_{\omega_e}\label{6.15}
\end{align}
with the second equality from an integration by parts and the last equality from \eqref{eq:weak:a}. 
The Cauchy--Schwarz inequality in \eqref{6.15} and the inverse estimate  result in
\begin{align*}
	\|[\partial_{\nn\nn}(\pd_ku_h)]\|_e^2&\lesssim \sum_{K'\in\omega_e}\Big(h_{K'}^{-2}(|u-\pd_ku_h|_{2,K'}+\eta_{1,K})+\|u-\Pi_ku_h\|_{K'}\\&\quad+\alpha h_{K'}^{-1/2}\|\nabla p -\Pi_{\ell-1}\nabla p_h\|_{K'}\Big)\|v\|_{K'}.
\end{align*}
Refer to \cite{chen2022} for the estimate $\|v\|_{\omega_e}\lesssim h_e^{3/2}\|[\partial_{\nn\nn}(\pd_ku_h)]\|_e$. This and
\eqref{eff1}    conclude the proof of \eqref{eff3}. Since $[T(\pd_k u_h)+\alpha\Pi_{\ell-1}\nabla p_h\cdot\nn]$ is a polynomial along an edge $e$, the analogous arguments as in the bound of $\eta_{3,K}$ for $w:=\psi_e[T(\pd_k u_h)+\alpha\Pi_{\ell-1}\nabla p_h\cdot\nn]$ lead to 
\begin{align}
&([T(\pd_k u_h)+\alpha\Pi_{\ell-1}\nabla p_h\cdot\nn],w)_e=(u-\Pi_ku_h,w)_{\omega_e}-\alpha(\nabla p-\Pi_{\ell-1}\nabla p_h,\nabla w)_{\omega_e}\nonumber\\&-(\tilde{f}-\Pi_ku_h-\Delta^2\pd_ku_h-\alpha\nabla \cdot\Pi_{\ell-1}\nabla p_h,w)_{\omega_e}-a^{\omega_e}(\pd_ku_h-u,w)\nonumber\\&+([\partial_{\nn\nn}(\pd_ku_h)],\partial_{\nn}w)_e.\nonumber
\end{align}
The Cauchy--Schwarz inequality, the inverse  estimates $\sum_{m=0}^2h_K^{m-2}|w|_{m,K}\lesssim \|w\|_K\lesssim h_e^{-3/2}\|[T(\pd_k u_h)+\alpha\Pi_{\ell-1}\nabla p_h\cdot\nn]\|_e$, and \eqref{eff1}-\eqref{eff3} conclude the proof of \eqref{eff4}. Let $b_e\in H^1_0(\omega_e)$ be the edge-bubble function constructed as in \cite[Lemma~9]{cangiani17} with the estimates 
\begin{align}
	\|\chi\|^2_e\lesssim (b_e,\chi^2)_e\lesssim \|\chi\|^2_e\quad\text{and}\quad\sum_{m=0}^1h_K^{m-1/2}\|b_e\chi\|_{m,K}\lesssim \|\chi\|_e\label{6.17}
\end{align}
for $\chi\in\p_\ell(e)$ with the constant elongation of $\chi$ in the normal direction of $e\in\e_K$. The test function $q=b_e[\alpha \nabla\cdot\Pi_{k-1}\nabla u_h+\gamma\nabla\cdot\Pi_{\ell-1}\nabla p_h]$ in \eqref{eq:weak:b} and an integration by parts show
\begin{align*}
&\beta(p-\Pi_\ell p_h,q)_{\omega_e}+\alpha(\nabla q,\nabla u-\Pi_{k-1}\nabla u_h)_{\omega_e}+\gamma (\nabla p-\Pi_{\ell-1}\nabla p_h,\nabla q)_{\omega_e}\\&=(\tilde{g}-\Pi_\ell\tilde{g},q)_{\omega_e}+(\Pi_\ell\tilde{g}-\Pi_\ell p_h+\alpha \nabla\cdot\Pi_{k-1}\nabla u_h+\gamma\nabla\cdot\Pi_{\ell-1}\nabla p_h,q)_{\omega_e}\\&\quad-([\alpha \nabla\cdot\Pi_{k-1}\nabla u_h+\gamma\nabla\cdot\Pi_{\ell-1}\nabla p_h],q)_e.
\end{align*}
The Cauchy--Schwarz inequality, $\chi=[\alpha \nabla\cdot\Pi_{k-1}\nabla u_h+\gamma\nabla\cdot\Pi_{\ell-1}\nabla p_h]$ in \eqref{6.17}, and the estimate \eqref{eff2} for $\eta_{2,K}$ conclude the proof of \eqref{eff5}. Then we can invoke  the equivalence $|u_h-\pg_\ell u_h|_{1,K}\approx \|\text{Dof}^k_K(u_h-\pg_\ell u_h)\|_{\ell^2}$ again as in the reliability, and then the definition of $\pd_\ell$ and the triangle inequality  allow us to prove that 
\begin{align*}
	\eta_{7,K}\lesssim |u_h-\pg_\ell u|_{1,K}\leq |u-u_h|_{1,K}+|u-\pg_\ell u|_{1,K}.
\end{align*}
The estimate \eqref{eff7} immediately follows from the arguments involved in  the proof of Theorem~\ref{4.1.c}.
\end{proof}
\begin{remark}[Higher degrees $k\geq 3$ and $\ell\leq k\leq \ell+2$]{\label{rem:nc-k3}}
If we introduce $J\vec{\bu_h}=(J_2u_h,J_4p_h)$ in the   proof of Theorem~\ref{thm:reliability} for $k\geq 3$, then the proof follows analogously  with only difference in the estimate \eqref{6.7} of the term $T_2$. There the arguments utilise the $H^1$-orthogonality $\nabla_\pw (v_h-J_2^*v_h)\perp (\p_{\ell-1}(\cT_h))^2$ and hence the same estimator works if one can construct $J_2^*$ for $k\geq 3$ with this orthogonality (which is possible but not  trivial). Still for higher $k$, we can invoke the $H^1$-orthogonality Theorem~\ref{4.2.b} and this leads to 
\begin{align*}
	&\alpha(\Pi_{\ell-1}\nabla p_h,\nabla J_2e_I^u-\Pi_{k-1}\nabla e_I^u)_K=\alpha(\Pi_{\ell-1}\nabla p_h-\Pi_{k-3}\nabla p_h,\nabla J_2e_I^u-\Pi_{k-1}\nabla e_I^u)_K\\&\qquad\lesssim \alpha (|p_h-\pg_\ell p_h|_{1,K}+|p_h-\pg_{k-2}p_h|_{1,K})h_K|e_I^u|_{1,K}
\end{align*}
with the Cauchy--Schwarz inequality,  the triangle inequality,  Theorem~\ref{4.2.d}, and Proposition~\ref{prop:poly} in the last estimate.  Consequently, we assume that $k-2\leq 
\ell$ and obtain an additional contribution, say $\eta_{9,K}$, in the error estimator. The equivalence of norms show  \[\eta_{9,K}^2:= \alpha h_K^2\|\text{Dof}^{\,k}_K(p_h-\pg_{k-2}p_h)\|_{\ell^2}^2,\]
and also the efficiency 
\[\eta_{9,K}\lesssim h_K|p_h-\pg_{k-2}p_h|_{1,K}\leq h_K|p_h-\pg_{k-2}p|_{1,K}\leq |p-p_h|_{1,K}+h_K|p-\pg_{k-2}p|_{1,K}. \]
\end{remark}
\begin{remark}[Conforming VEM]
As companion operators are not required in the conforming case, the proof of reliability and efficiency follows analogously assuming $J= I$, where $I$ denotes the identity operator. Note that the local contributions  $\eta_{8,K}$ and $\eta_{9,K}$ arise due to the noncoformity of the method and hence the error estimator in the conforming case is 
\[\|u-u_h^c\|_{\bH_\epsilon}^2\lesssim \sum_{i=1}^7\eta_{i}^2.\]
\end{remark}
\begin{remark}[Choices of projection operators]
Note that the projection $\Pi_{k-2}\nabla v_h$ for $v_h\in\vhc$ (or $\vhnc$) is also computable in terms of the DoFs, and both a priori and a posteriori error analysis hold with this choice. We prefer to use $\Pi_{k-2}\nabla v_h$ instead of $\Pi_{k-1}\nabla v_h$ in the numerical experiments below. Also from the theoretical analysis,  observe that if we set $\ell=k$ (one degree higher for pressure) and modify the term $(\nabla p,\nabla u)_K\approx (\Pi_{k-1}\nabla p,\Pi_{k-1}\nabla u)_K$ for all $K\in\cT_h$ in the discrete approximation, then the error estimator component $\eta_7$ will  disappear. But higher approximation of pressure may not be a good choice from a numerical perspective.
\end{remark}
\section{Numerical results}\label{sec:results}
We now present a number of computational tests that confirm the theoretical a priori and a posteriori error estimates from Sections~\ref{sec:error-c}-\ref{sec:error-nc} and  Section~\ref{sec:apost}, and we also include typical benchmark solutions for Kirchhoff--Love plates that we modify to include the coupling with filtration in porous media. All meshes were generated with the library \texttt{PolyMesher} \cite{talischi12}. 
\subsection{Example 1: Accuracy verification with smooth solutions}
In order to investigate numerically the error decay predicted by Theorems~\ref{theo:cv-c} and \ref{theo:cv-nc}, we follow the approach of manufactured solutions. We set the parameters $\alpha=\beta=\gamma =1$ in all the examples below. 
\par We construct a transverse load and a source function  $f,g$, respectively, as well as homogeneous and non-homogeneous boundary data for $u$ and $p$, such that the problem has the following smooth deflection and fluid pressure moment exact solutions 
\[ u(x,y) = \sin^2(\pi x)\sin^2(\pi y), \quad p(x,y) = \cos(\pi xy),\]
on the square domain $\Omega=(0,1)^2$ with mixed boundary conditions $\Gamma_c:= \{x=0\} \cup \{y=0\} $ and $\Gamma_s:=\partial\Omega\setminus\Gamma_c$. Then we employ a sequence of successively refined meshes $\cT_h^i$ and compute the projected virtual element solution $(\pd_ku_h,\pg_\ell p_h)$ on each mesh refinement $h_i$, and monitor the norms  $|u-\pd_ku_h|_{2,h}$ for displacement approximation,  $|p-\pg_\ell p_h|_{1,h}$ for pressure approximation and the combined energy norm $\|\cdot\|_{\bH_\epsilon^h}$. The experimental order of convergence $\texttt{r}_i$ is computed from the formula 
\[ \texttt{r}_i= \frac{\log (\frac{\texttt{e}_{i+1}}{\texttt{e}_i})}{\log (\frac{h_{i+1}}{h_i})},\] 
where $\texttt{e}_{i}$ denotes a norm of the error on the mesh $\cT_h^{i}$.
\par We impose the appropriate boundary conditions for both clamped and simply supported boundaries. In case of the conforming VEM, note that the degrees of freedom include the gradient values at vertices and  $u(z)=0 =  \partial_{\nn\nn}u(z) =0$ implies $\nabla u(z) = \cero$ for a  corner $z$ along the boundary $\Gamma^s$. Hence,  we have to impose the zero gradient values at the corners on simply supported part in addition to the clamped part. 
\begin{figure}[H]
\centering
\begin{subfigure}{.5\textwidth}
	\centering
	\includegraphics[width=1\linewidth]{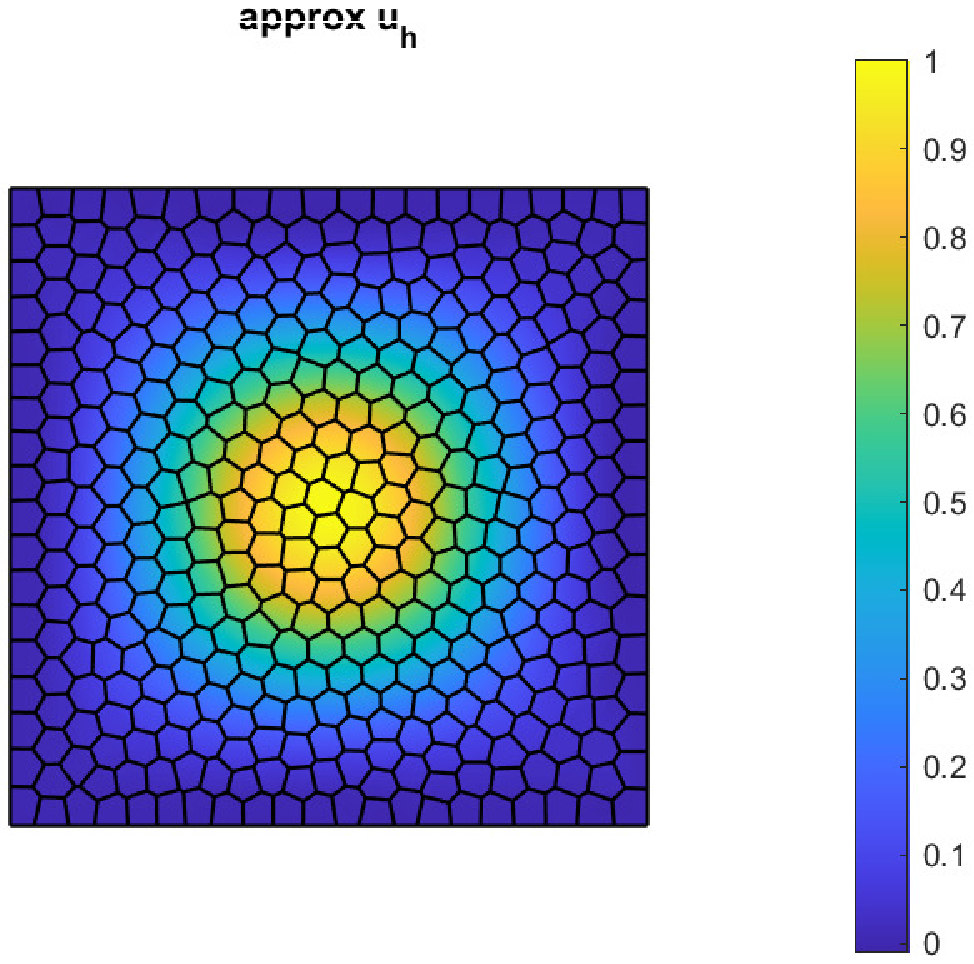}
\end{subfigure}%
\begin{subfigure}{.5\textwidth}
	\centering
	\includegraphics[width=1\linewidth]{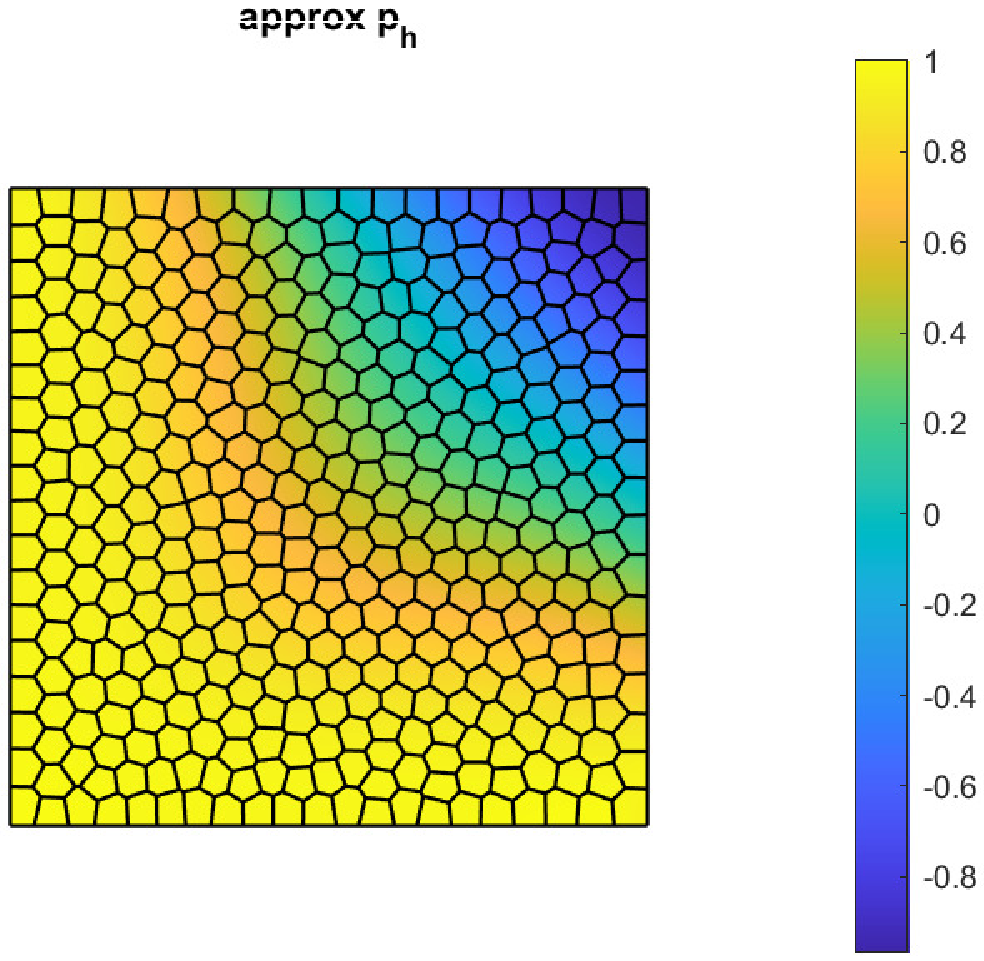}
\end{subfigure}
\caption{Approximation $u_h$  of displacement $u$ for $k=2$ (left) and $p_h$ of pressure $p$ for $\ell=1$ (right)   on a smooth Voronoi mesh of 400 elements.}
\label{fig:Ex1}
\end{figure}
We take $\ell=k-1$ in numerical experiments to obtain the expected  optimal convergence rates for both conforming and nonconforming VEM. See Table~\ref{table:ex1-nc-and-c} (resp. \ref{table:ex1_k3-c-and-nc})   for $k=2$ and $\ell=1$ (resp. for $k=3$ and $\ell=2$). Tables~\ref{table:ex1-nc-and-c}-\ref{table:ex1_k3-c-and-nc} display the errors and the convergence rates on a sequence of Voronoi meshes. 
\begin{table}[H]
\begin{center}
	\begin{tabular}{|c|c|c|c|c|c|c|}
		\hline
		$h$& $|u-\Pi_2^{\nabla^2}u_h|_{2,h}$&$\texttt{r}_i$&$|p-\Pi_1^\nabla p_h|_{1,h}$&$\texttt{r}_i$& $|\vec{\bu}-\vec{\Pi} \vec{\bu}_h|_{\bH_\epsilon^h}$& $\texttt{r}_i$\vphantom{$\displaystyle\int$}\\
		\hline
		\hline
		\multicolumn{7}{|c|}{Conforming VE discretisation}\\
		\hline
		0.6801&13.530&0.575&0.8743&0.106&14.848&0.551\\
		0.3124&8.6492&1.213&0.8051&1.247&9.6714&1.212\\
		0.1558&3.7205&0.988&0.3382&1.450&4.1621&1.036\\
		0.0795&1.9133&0.908&0.1275&1.328&2.0729&0.941\\
		0.0390&1.0023&1.148&0.0495&1.126&1.0608&1.151\\
		0.0193&0.4483&*        &0.0225&*&0.4734&*\\                \hline
		
		\hline
		\multicolumn{7}{|c|}{Nonconforming VE discretisation}\\
		\hline
		0.7147&24.738&0.966&5.1202&0.855&33.475&1.001\\
		0.3625&12.837&1.197&2.8642&1.360&16.960&1.223\\
		0.1862&5.7833&0.943&1.1574&1.798&7.5073&1.092\\
		0.0938&3.0282&1.181&0.3372&1.825&3.5485&1.261\\
		0.0476&1.3600&1.207&0.0978&1.632&1.5092&1.252\\
		0.0244&0.6077&*&0.0329&*&0.6544&*\\
		\hline
	\end{tabular}
\end{center}
\caption{Error $\vec{\bu}-\vec{\Pi}\vec{\bu}_h=(u-\Pi_2^{\nabla^2}u_h,p-\pg_1p_h)$ in the energy norm $\|\bullet\|_{\bH_\epsilon^h}$ 
	with $k=2$ and $\ell=1$ on a sequence of smooth Voronoi meshes of 5,25,100,400,1600, and 6400 elements.}
\label{table:ex1-nc-and-c}
\end{table}
\begin{table}[H]
\begin{center}
	\begin{tabular}{|c|c|c|c|c|c|c|}
		\hline
		$h$& $|u-\Pi_3^{\nabla^2}u_h|_{2,h}$&$\texttt{r}_i$&$|p-\Pi_2^\nabla p_h|_{1,h}$&$\texttt{r}_i$& $|\vec{\bu}-\vec{\Pi} \vec{\bu}_h|_{\bH_\epsilon^h}$& $\texttt{r}_i$\vphantom{$\displaystyle\int$}\\
		\hline
		\hline
		\multicolumn{7}{|c|}{Conforming VE discretisation}\\
		\hline
		0.6801&8.2153&1.986&0.4883&1.712&8.9142&1.984\\
		0.3096&1.7206&2.058&0.1269&2.861&1.8713&2.114\\
		0.1508&0.3916&2.219&0.0162&2.313&0.4091&2.226\\
		0.0794&0.0942&1.979&0.0036&2.069&0.0980&1.984\\
		0.0393&0.0234&2.109&0.0008&2.166&0.0243&2.111\\
		0.0203&0.0058&*&0.0002&*&0.0060&*\\
		\hline
		\multicolumn{7}{|c|}{Nonconforming VE discretisation}\\
		\hline
0.7147&10.748&0.7089&3.6144&-0.3735&15.115&0.3286\\
0.3406&6.3548&1.5061&4.7674&1.9128&11.847&1.7228\\
0.1875&2.5874&2.0609&1.5229&2.5617&4.2385&2.2546\\
0.0969&0.6639&1.9931&0.2808&2.8056&0.9570&2.2010\\
0.0484&0.1663&2.2325&0.0340&2.6055&0.2075&2.3010\\
0.0239&0.0345&2.0891&0.0064&2.2739&0.0410&2.1178\\
0.0123&0.0086&*&0.0014&*&0.0101&*\\
		\hline
	\end{tabular}
\end{center}
\caption{Error $\vec{\bu}-\vec{\Pi}\vec{\bu}_h=(u-\Pi_3^{\nabla^2}u_h,p-\pg_2p_h)$ in the energy norm $\|\bullet\|_{\bH_\epsilon^h}$
	with $k=3$ and $\ell=2$ on a sequence of smooth Voronoi meshes of 5,25,100,400,1600, and 6400 elements.}
\label{table:ex1_k3-c-and-nc}
\end{table}
	Figures~\ref{fig1:ErrEst}-\ref{fig1:ncErrEst} display the error and the error estimator convergence rates for both uniform and adaptive refinements. In this example, we choose a smooth Voronoi mesh of 25 elements as an initial partition and follow the standard adaptive algorithm 
	\[\text{SOLVE}\longrightarrow \text{ESTIMATE} \longrightarrow \text{MARK} \longrightarrow \text{REFINE} \]
	\par In all the adaptive experiments below, we first solve the discrete problem \eqref{eq:operator_discrete} (resp. \eqref{eq:operator_discrete_nc}) for conforming (resp. nonconforming), compute the upper bound $\eta$ in Theorem~\ref{thm:reliability}, consider the D\"{o}rfler marking strategy with $\theta = 0.5$, and  divide a marked polygon into quadrilaterals  by connecting vertices to the centroid of the respective polygon.  The same refinement strategy is utilised to divide all the elements in case of uniform refinement. The additional error estimator component $\eta_9$ from Remark~\ref{rem:nc-k3} is incorporated in the experiment of the nonconforming VEM with degree $k=3$ and $\ell=2$.	
	\begin{figure}[H]
		\centering
		\begin{subfigure}{.5\textwidth}
			\centering
			\includegraphics[width=0.95\linewidth]{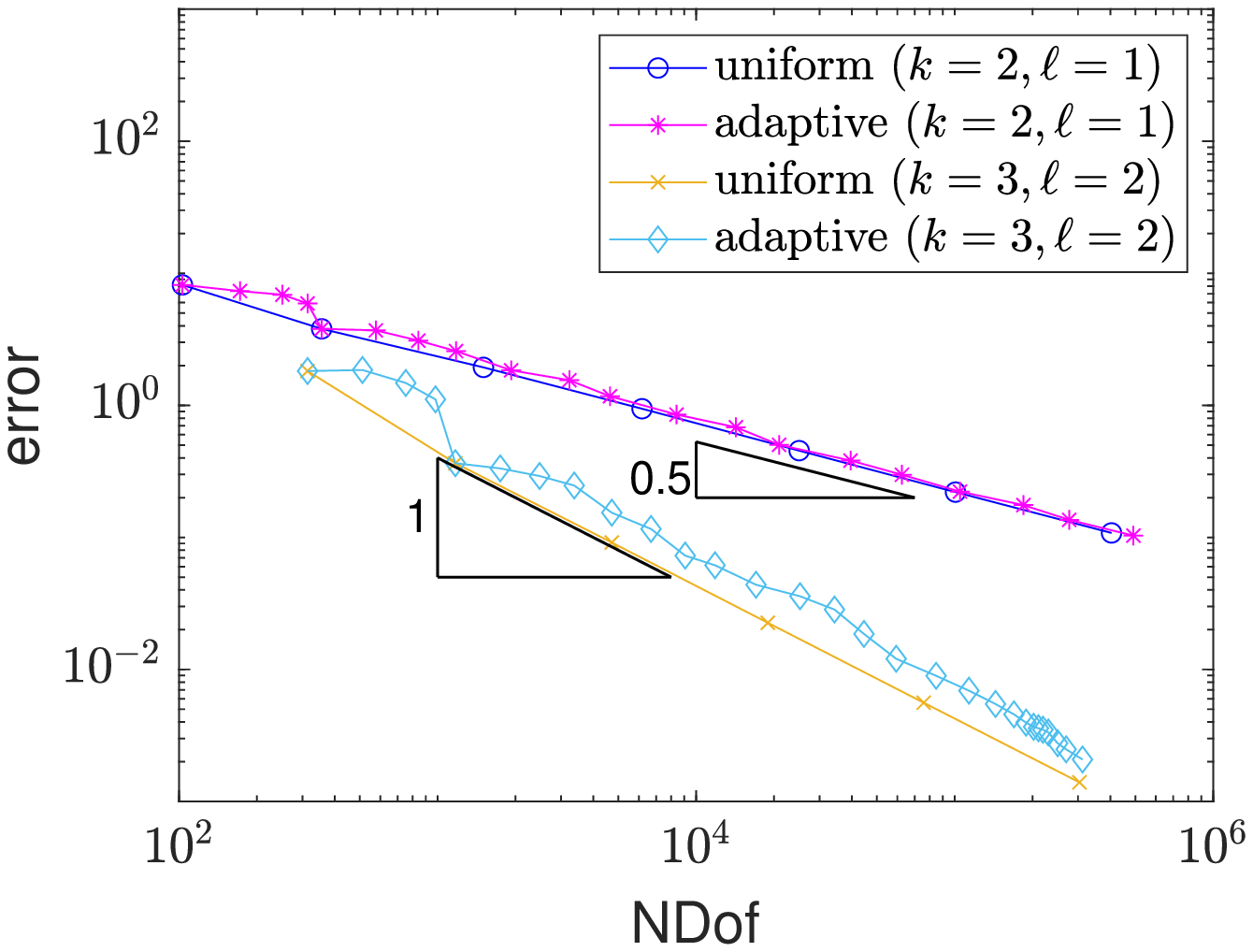}
		\end{subfigure}%
		\begin{subfigure}{.5\textwidth}
			\centering
			\includegraphics[width=0.95\linewidth]{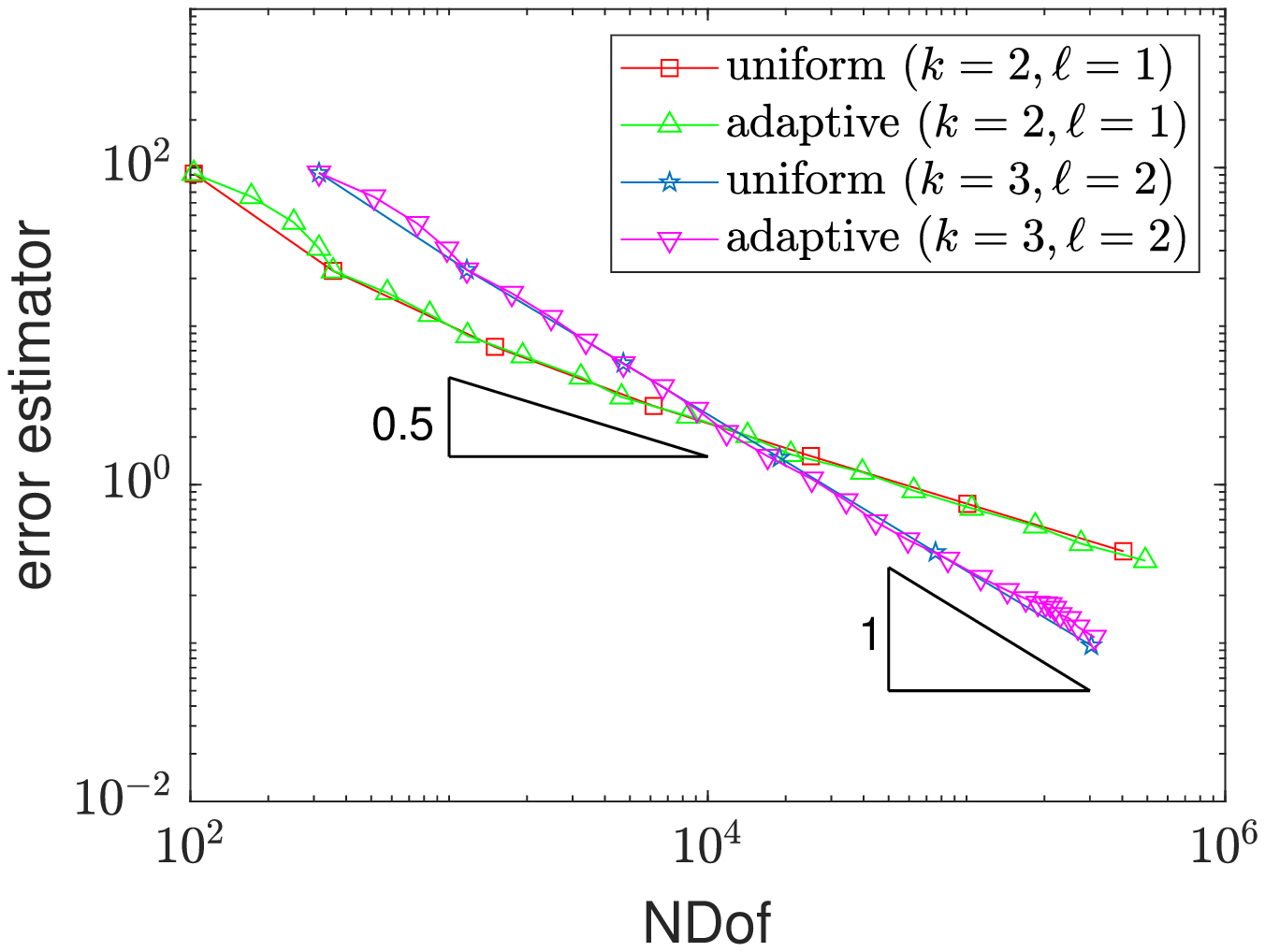}
		\end{subfigure}
		\caption{Left (resp. right) panel displays NDof vs error in energy norm (resp. error estimator) in both uniform and adaptive refinements for conforming VEM.}
		\label{fig1:ErrEst}
	\end{figure}
	\begin{figure}[H]
		\centering
		\begin{subfigure}{.5\textwidth}
			\centering
			\includegraphics[width=0.95\linewidth]{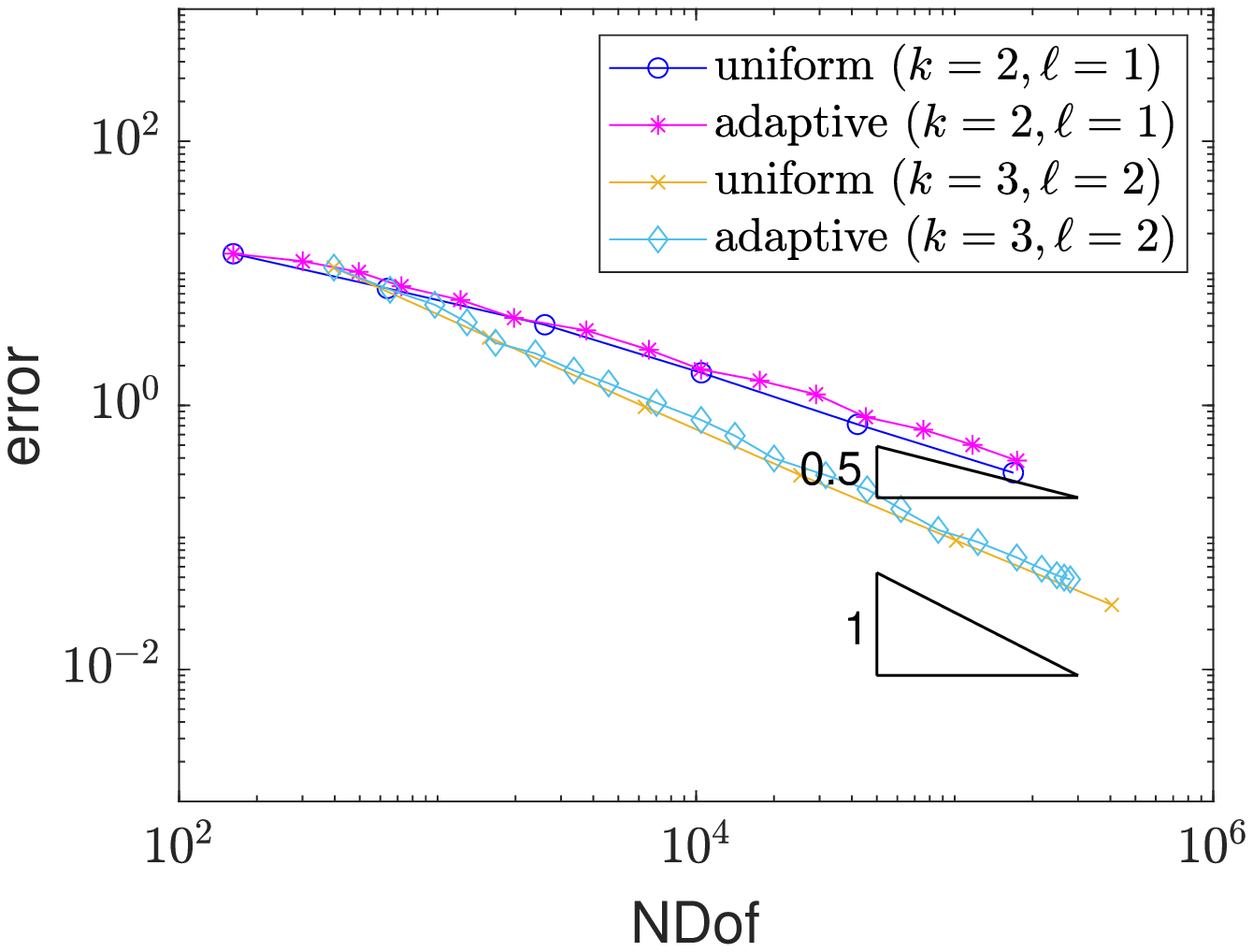}
		\end{subfigure}%
		\begin{subfigure}{.5\textwidth}
			\centering
			\includegraphics[width=0.95\linewidth]{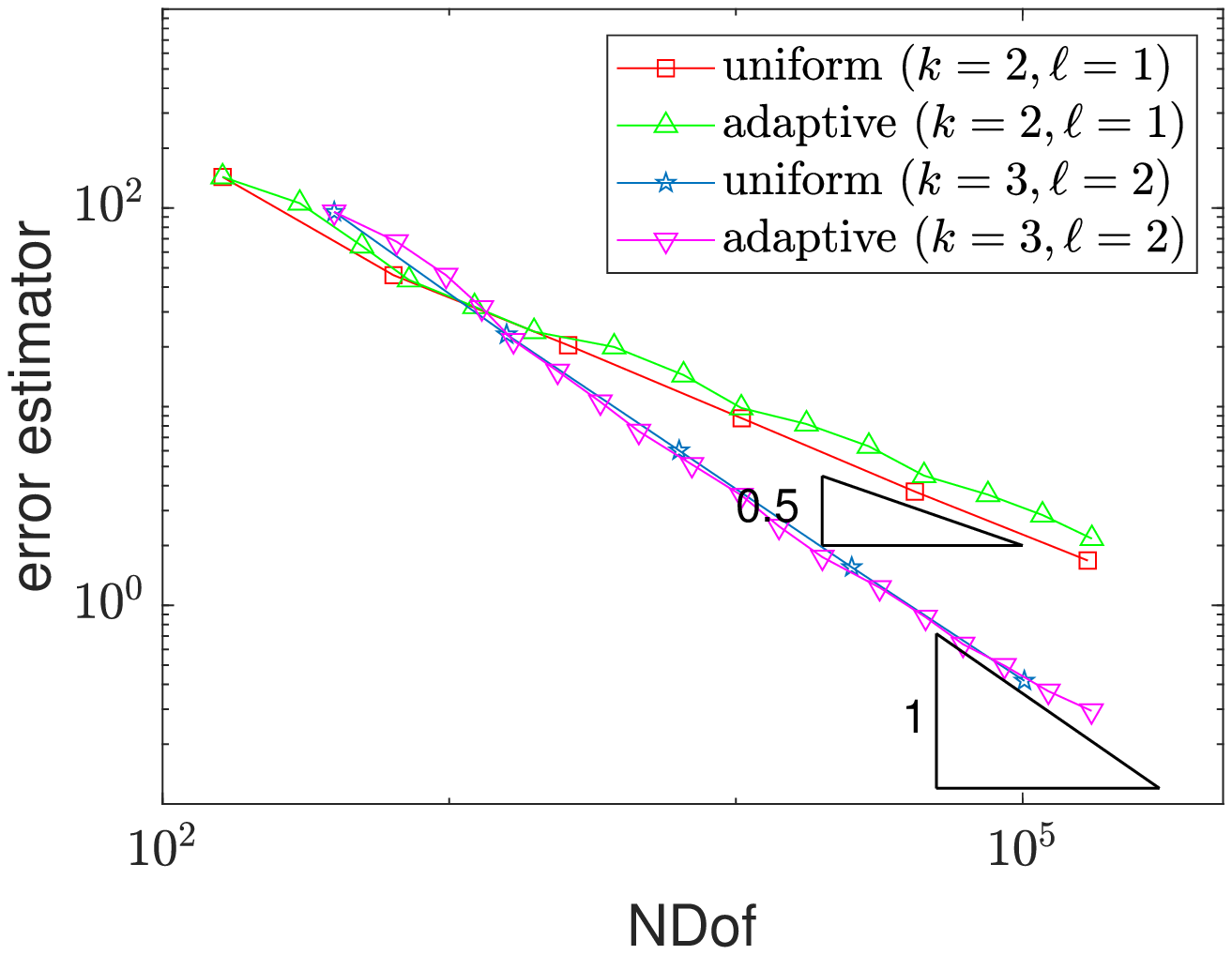}
		\end{subfigure}
		\caption{Left (resp. right) panel displays NDof vs error in energy norm (resp. error estimator) in both uniform and adaptive refinements for nonconforming VEM.}
		\label{fig1:ncErrEst}
	\end{figure}

\subsection{Example 2: Convergence rates with non-smooth solutions}
We consider the L-shaped domain $\Omega= (-1,1)^2\setminus ([0,1)\times[-1,0))$  and the exact solution
	\[u(r,\theta) = r^{5/3}\sin\Big(\frac{5\theta}{3}\Big),\quad
	p(r,\theta) = r^{2/3}\sin\Big(\frac{2\theta}{3}\Big)\]
	with clamped boundary conditions for $u$ and Dirichlet boundary condition for $p$ on $\partial \Omega$ (observe that we can take Dirichlet boundary condition instead of Neumann for $p$ on $\Gamma^c$ without affecting  the well-posedness and error analysis of the model problem). Since both the displacement $u\in H^{(8/3)-\epsilon}(\Omega)$ and the pressure $p\in H^{(5/3)-\epsilon}(\Omega)$ for all $\epsilon>0$ have corner singularities,  the lowest-order scheme $k=2$ and $\ell=1$ suffices to achieve the optimal convergence rates with respect to the regularity of $u$ and $p$. 
	\par When the adaptive algorithm is run, we see more refinement around the singular corner as displayed in Figures~\ref{fig1:Ex2}-\ref{fig2:Ex2}.  Figure~\ref{fig3:Ex2} shows  that the method with unifom refinement leads to suboptimal rates whereas adaptive refinement recovers the optimal convergence rates, and the error estimator mirrors the behaviour of the actual error. We observe from the plots of error estimator components  that $\eta_7$ (resp. $\eta_8$) dominates the  remaining contributions for the case of conforming (resp. nonconforming) VEM.
	\begin{figure}[H]
		\centering
		\begin{subfigure}{.33\textwidth}
			\centering
			\includegraphics[width=1.1\linewidth]{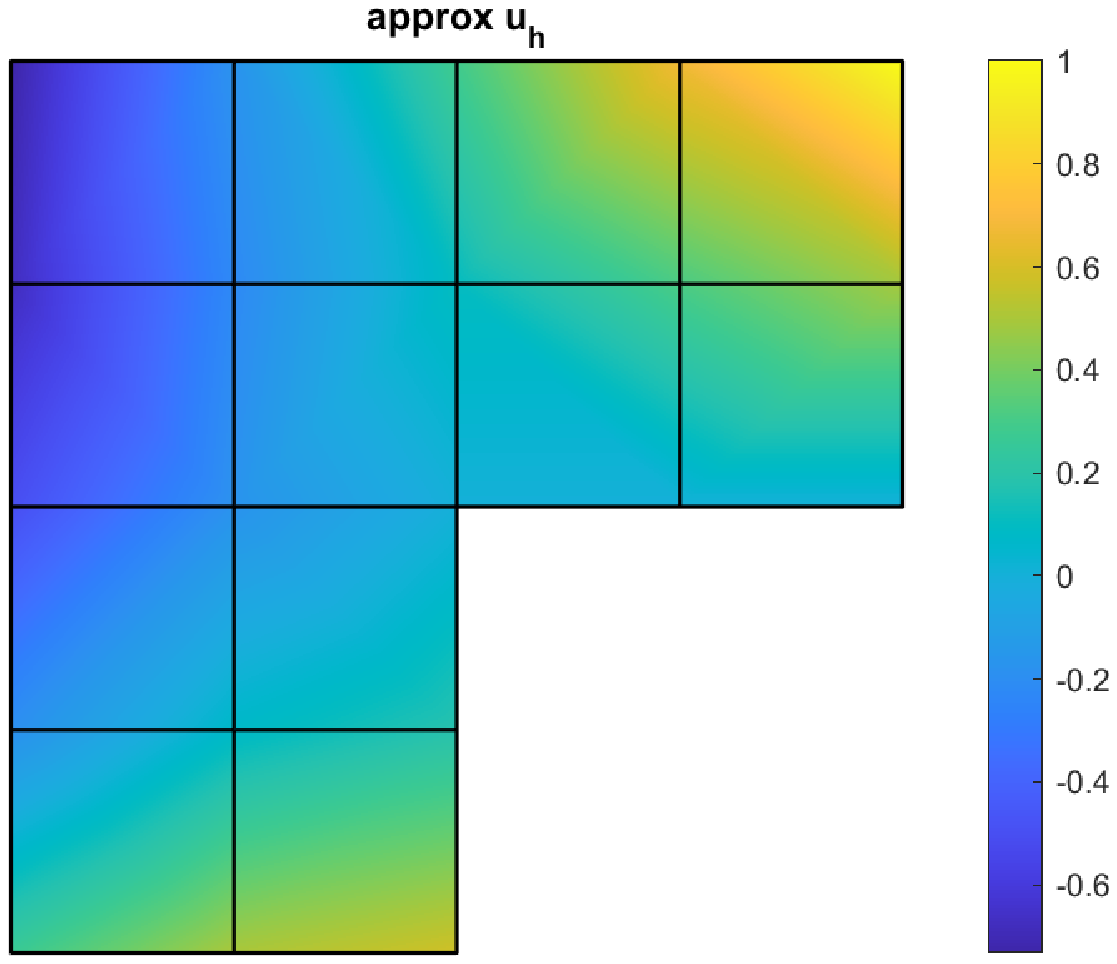}
		\end{subfigure}%
		\begin{subfigure}{.33\textwidth}
			\centering
			\includegraphics[width=1.1\linewidth]{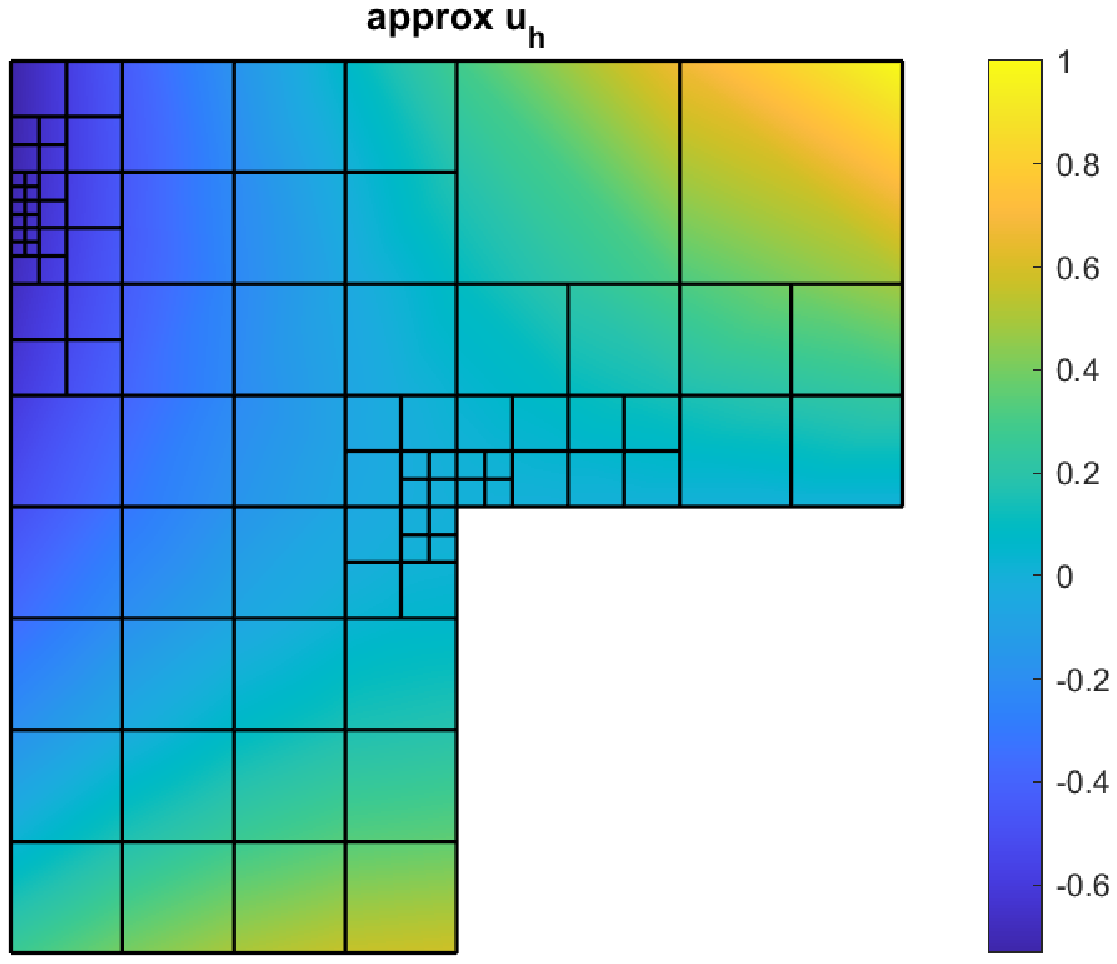}
		\end{subfigure}%
		\begin{subfigure}{.33\textwidth}
			\centering
			\includegraphics[width=1.1\linewidth]{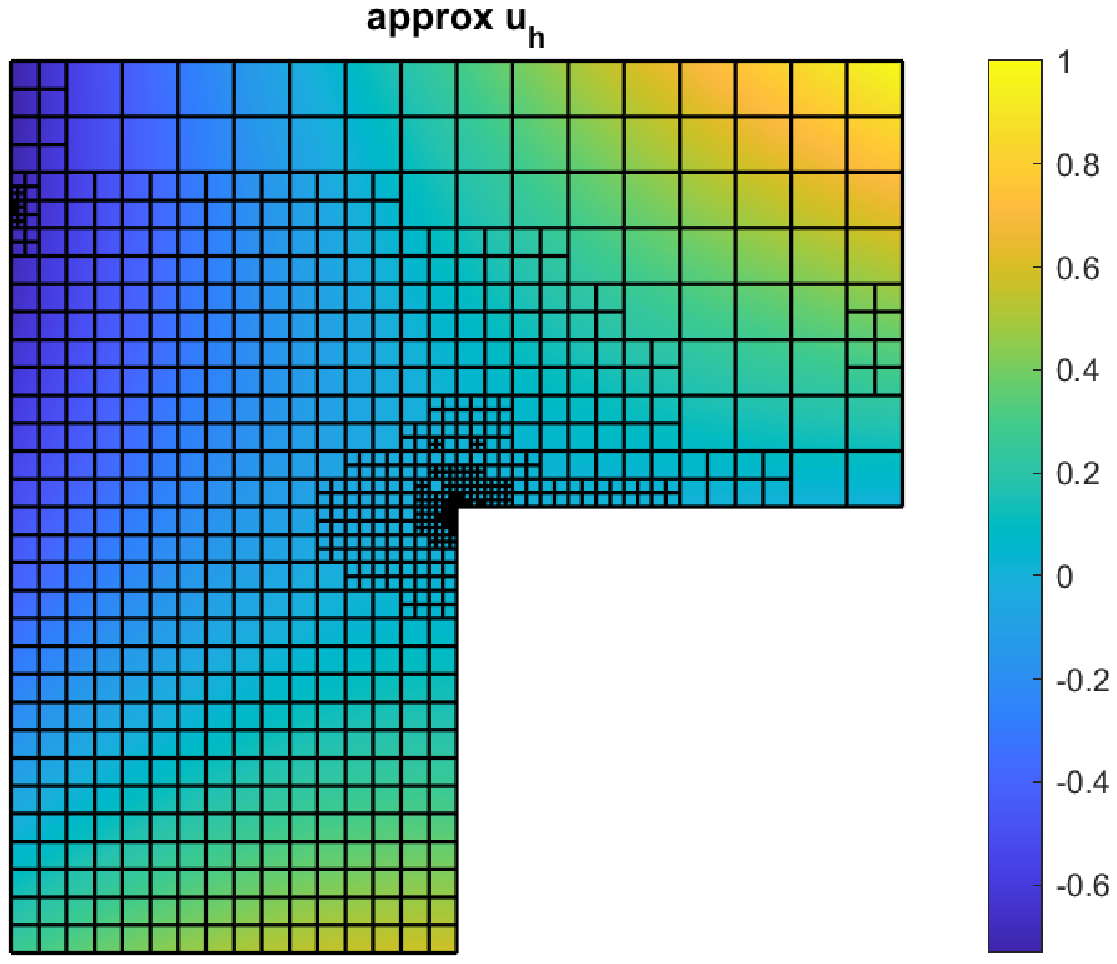}
		\end{subfigure}
		\caption{Approximation $u_h$ of displacement $u$  on adaptive meshes $\cT_1, \cT_5,\cT_{10}$.}
		\label{fig1:Ex2}
	\end{figure}
 \begin{figure}[H]
		\centering
		\begin{subfigure}{.33\textwidth}
			\centering
			\includegraphics[width=1.1\linewidth]{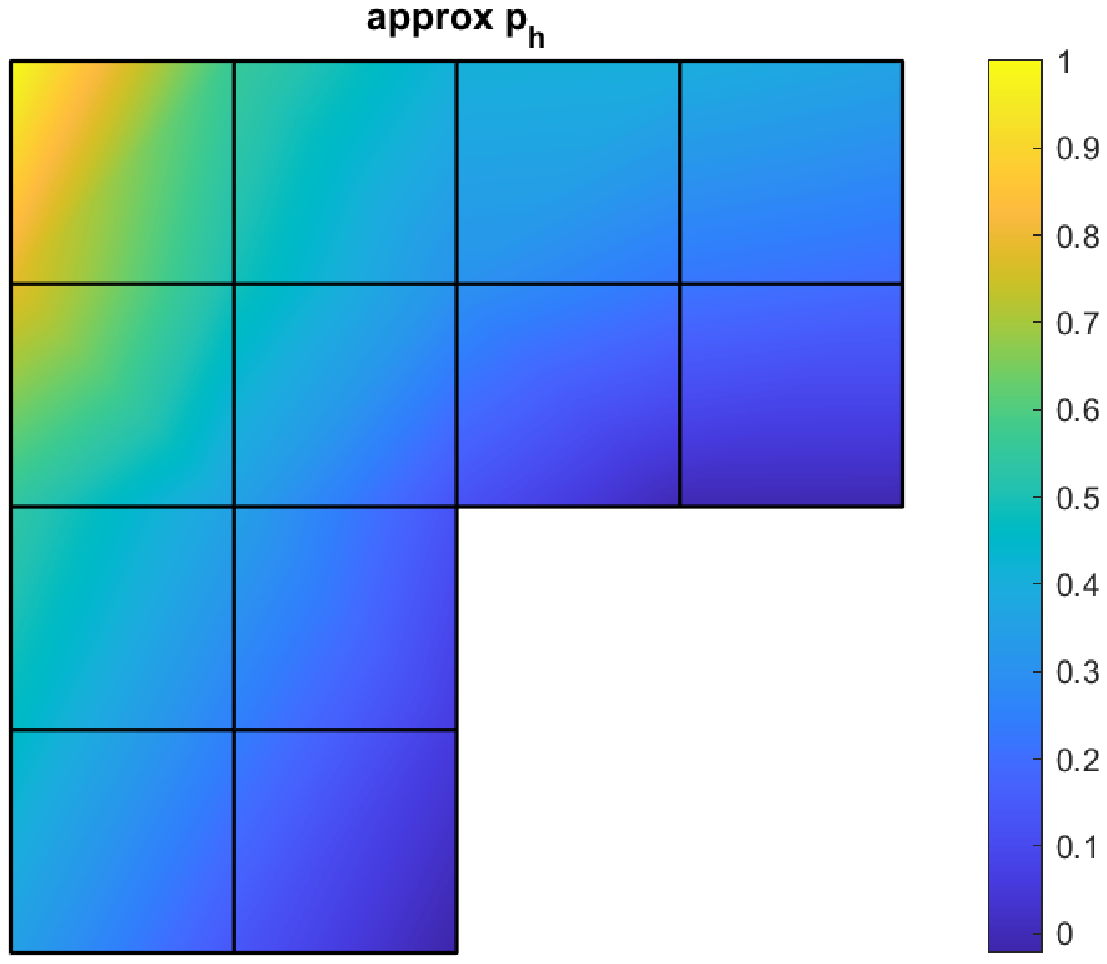}
		\end{subfigure}%
		\begin{subfigure}{.33\textwidth}
			\centering
			\includegraphics[width=1.1\linewidth]{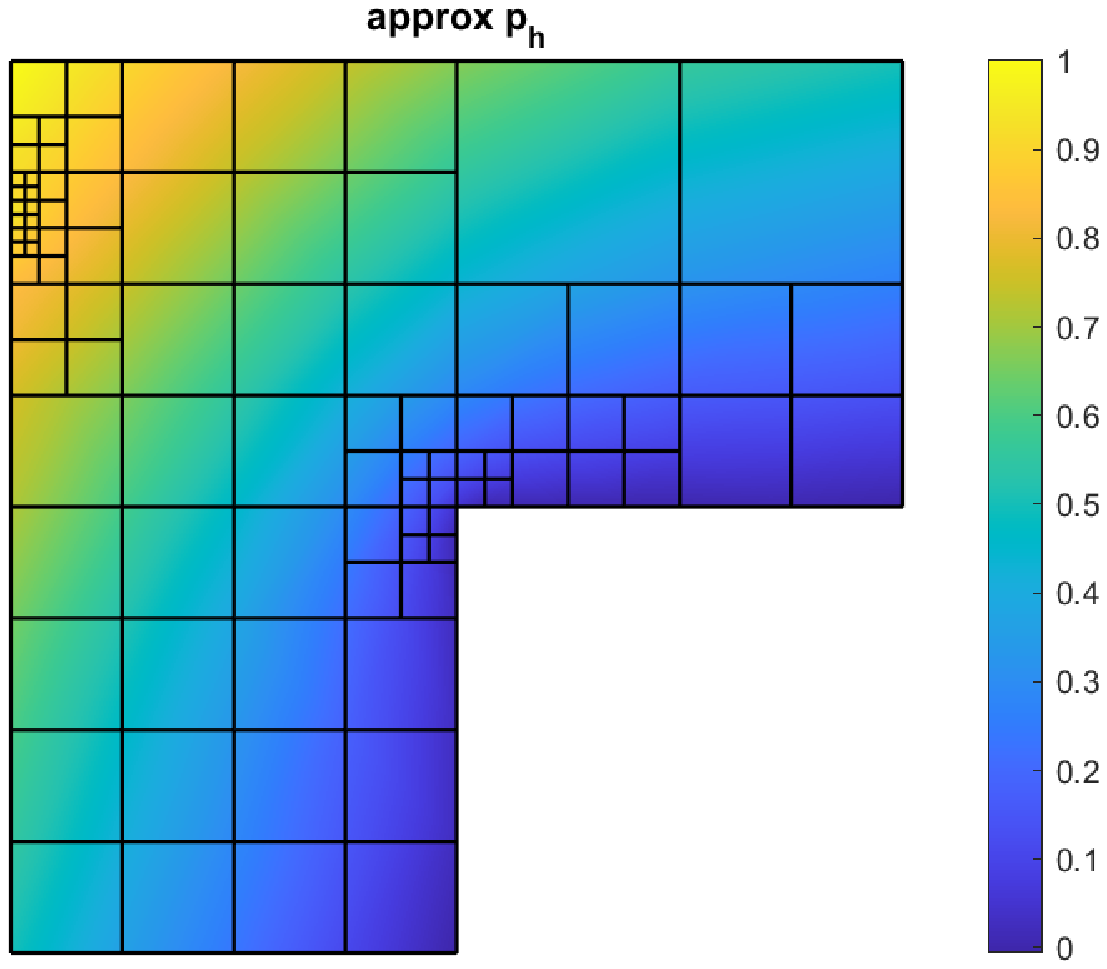}
		\end{subfigure}%
		\begin{subfigure}{.33\textwidth}
			\centering
			\includegraphics[width=1.1\linewidth]{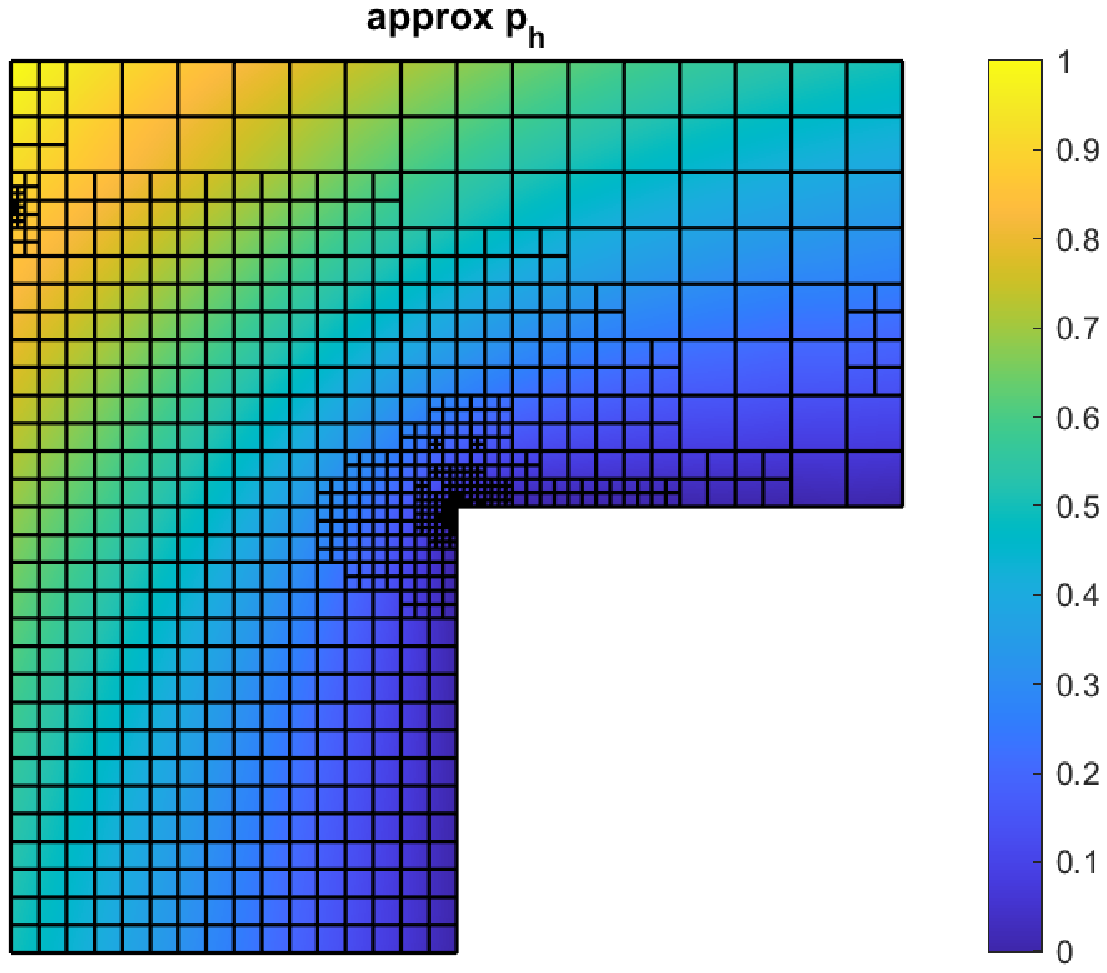}
		\end{subfigure}
		\caption{Approximation $p_h$ of pressure $p$  on adaptive meshes $\cT_1, \cT_5,\cT_{10}$.}
		\label{fig2:Ex2}
	\end{figure}
	\begin{figure}[H]
		\centering
		\begin{subfigure}{.5\textwidth}
			\centering
			\includegraphics[width=0.95\linewidth]{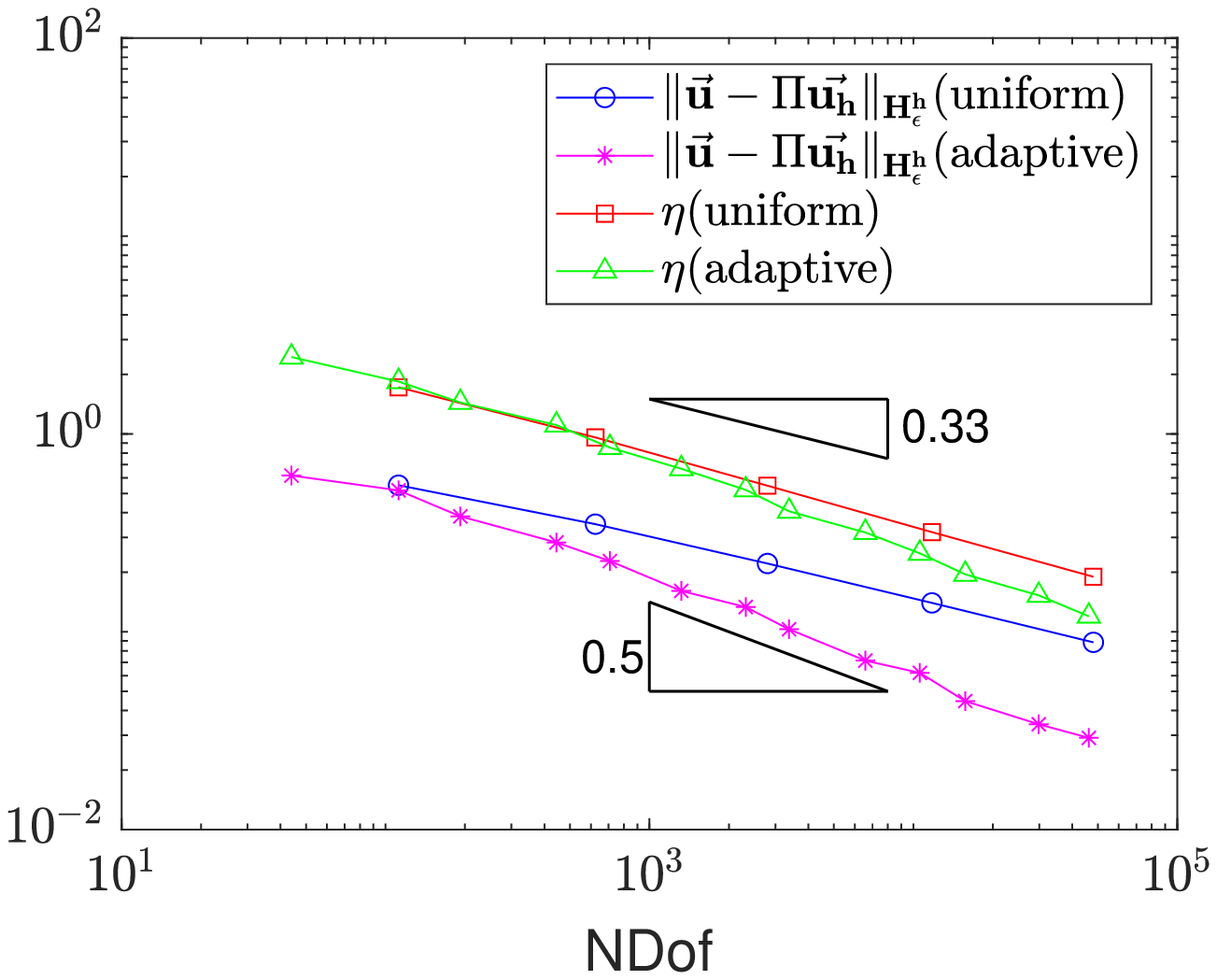}
		\end{subfigure}%
		\begin{subfigure}{.5\textwidth}
			\centering
			\includegraphics[width=0.95\linewidth]{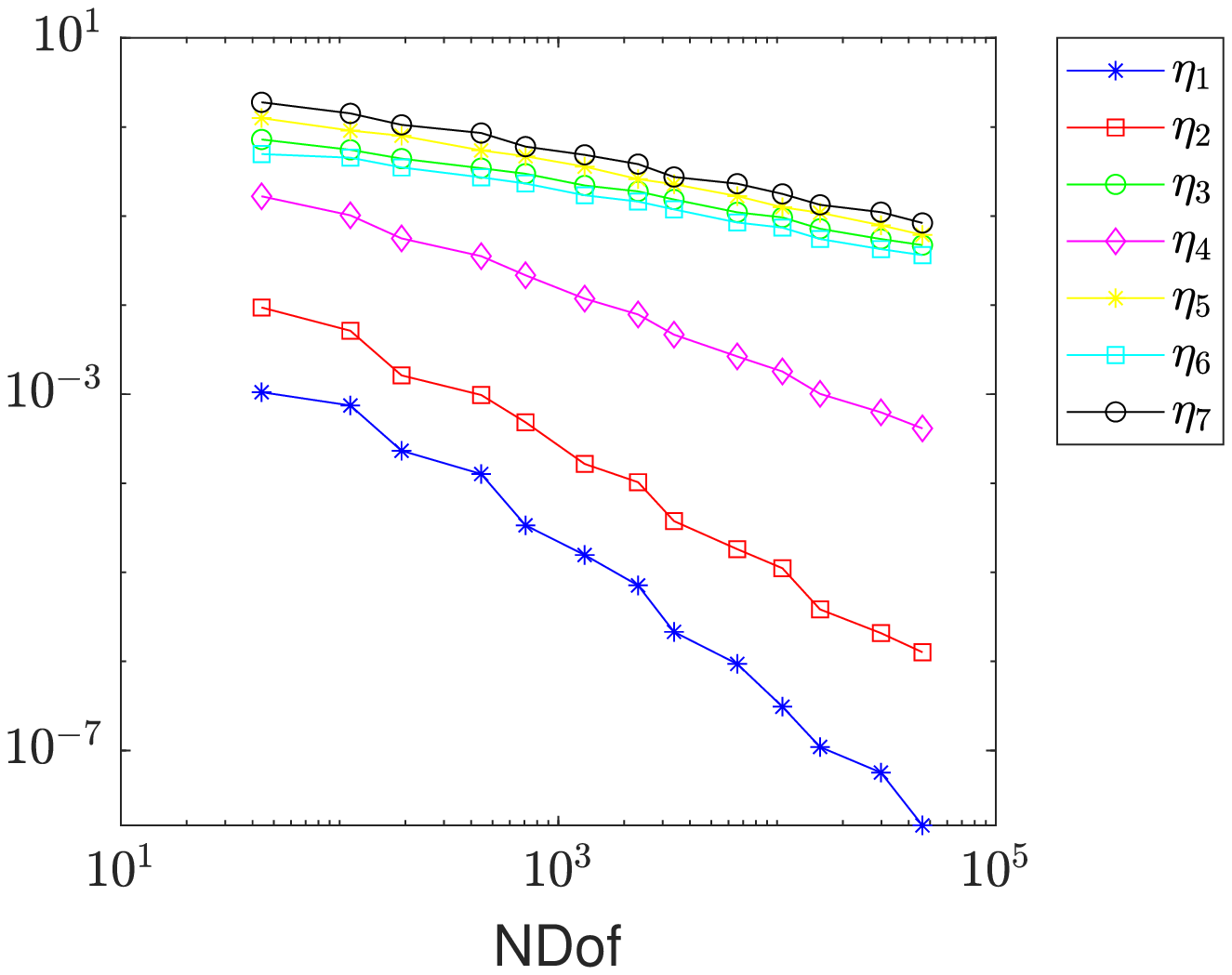}
		\end{subfigure}
		\caption{Left panel displays NDof vs error in energy norm and estimator in both uniform and adaptive refinements, and right panel displays estimator components in adaptive refinement for conforming VEM.}
		\label{fig3:Ex2}
	\end{figure}
	
	\begin{figure}[H]
		\centering
		\begin{subfigure}{.5\textwidth}
			\centering
			\includegraphics[width=0.95\linewidth]{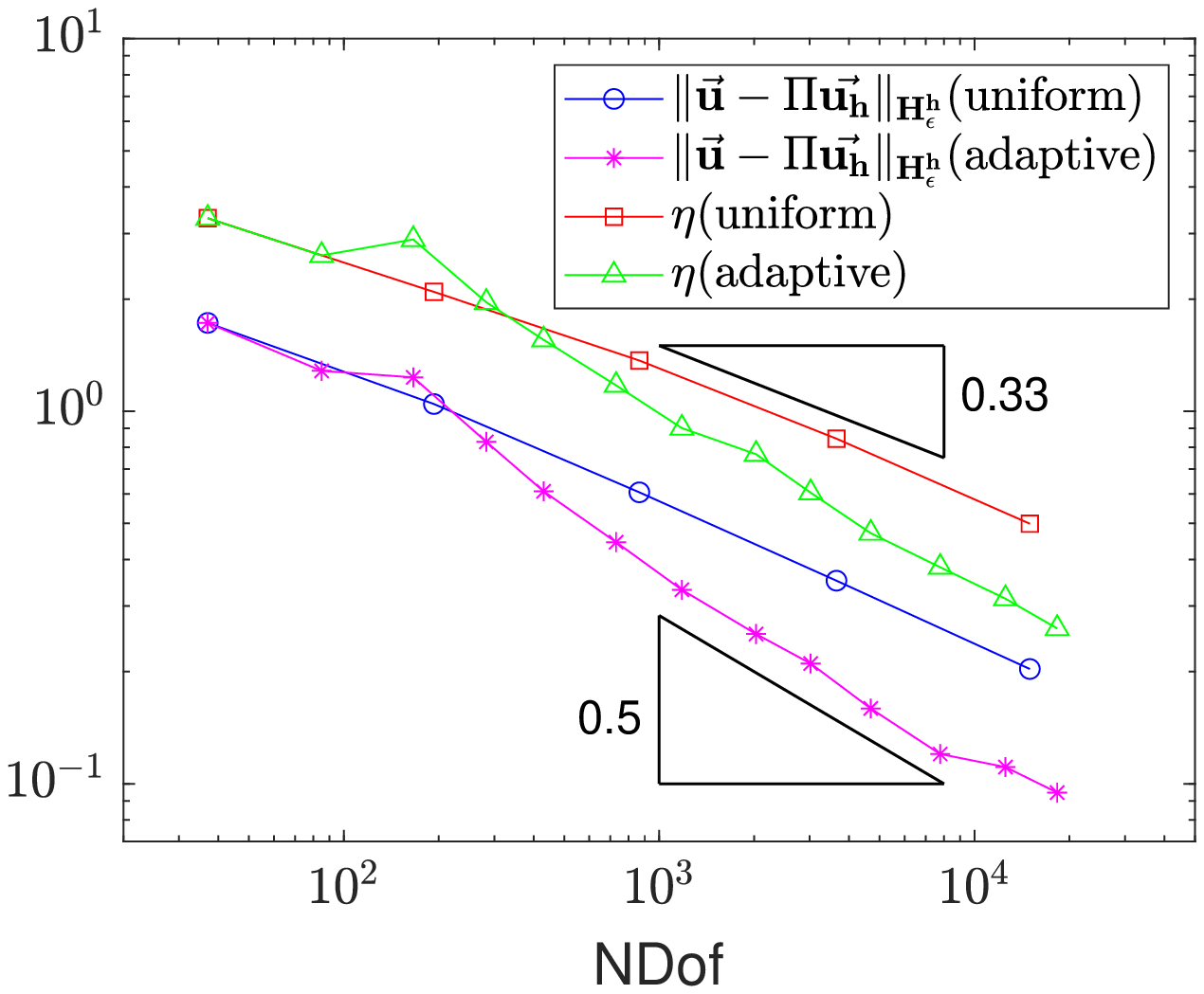}
		\end{subfigure}%
		\begin{subfigure}{.5\textwidth}
			\centering
			\includegraphics[width=0.95\linewidth]{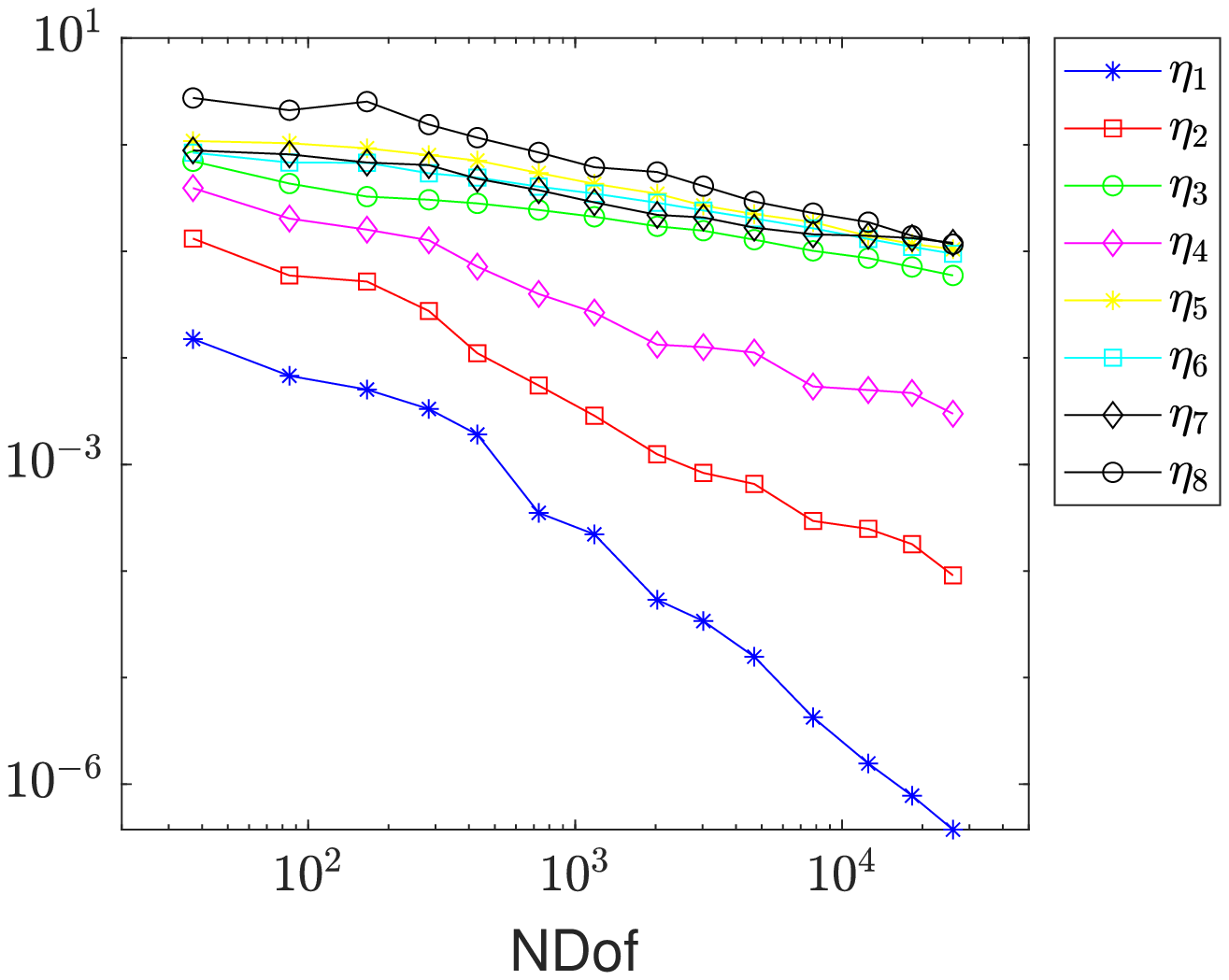}
		\end{subfigure}
		\caption{Left panel displays NDof vs error in energy norm and estimator in both uniform and adaptive refinements, and right panel displays estimator components in adaptive refinement for nonconforming VEM.}
		\label{fig4:Ex2}
	\end{figure}
\subsection*{Conclusions}
	This paper analyses  conforming and nonconforming VEMs for the poroelastic plate model \eqref{problem} referred from \cite{iliev16}. The well-posedness of both continuous and discrete problems follow straightforwardly. Theorem~\ref{theo:cv-c} provides the a priori error estimate in the best-approximation form for the conforming case. In the nonconforming case, we first developed a key tool, the so-called companion operator, which maps from the nonconforming VE space to the continuous space ($V$ for displacement space and $Q$ for pressure space). We stress that the construction of the companion operator in this paper is novel and substantially different from those already present in the literature \cite{huang21, carstensen22, carstensen22b}, in the sense that it is designed for general degree VE spaces satisfying  orthogonality and best-approximation properties. Furthermore, this operator is an independent vital technical argument that can be useful in other nonconforming VE methods for different second and fourth-order elliptic problems. The VE functions locally need not be polynomials and this paper contributes to the non-standard proofs of  basic estimates such as Poincar\'e-type inequalities, inverse estimates, and norm-equivalences. This paper also presents a residual-based reliable and efficient a posteriori error estimator, which is robust with respect to the model parameters. To support the robustness of the error analysis with respect to the model parameters, we also perform uniform and adaptive numerical tests for moderate and extreme values of model parameters  $\alpha=10^{-6}, \beta = 10^6= \gamma$ and observe from Figure~\ref{fig:scale} that although the magnitude of errors and error estimators become large for large values of $\beta$ and $\gamma$, the convergence rates remain optimal as predicted by the theory.
	\par We emphasise here that both conforming and nonconforming VEM have great advantages over standard FEM from the perspective that the $C^1$-conforming FEM  requires higher dofs compared to conforming VEM and nonconforming FEM for higher degrees can not be easily put in one framework as we can do in VEM. However since the model problem is a fourth-order PDE, we require $C^1$ continuity in the conforming case, and consequently the nonconforming VE can be implemented with less computational effort compared to its conforming counterpart.
	\par   Recall that the model problem \eqref{problem} analysed in this paper  is a scaled version of the original model problem \eqref{org-problem}. Certainly, one can start with \eqref{org-problem} instead of the normalised version, but the coercivity of $\cA$ (uniformly with respect to the parameters) is not straightforward and one may have to invoke a different abstract result (for example a global inf-sup condition or similar arguments) to prove the well-posedness of the problem. Note that our analysis (well-posedness as well as the robustness with respect to model parameters) majorly depends on having the same coefficient $\alpha$ for the coupled terms and the relation $\alpha\leq 1 \leq \gamma$. It will be interesting to track all model parameters, but the formulation is substantially modified leading to a completely different analysis and we keep it open for a future work. Also the limiting case $d\to 0$ indeed requires a change in the model to consider nonlinear effects, and altogether implying a different physical phenomenon to be regarded and studied separately. Finally, we mention that as an important extension of this work we are developing new mixed formulations for the coupling of Biot--Kirchhoff plates interacting with a bulk free flow.
	\begin{figure}[H]
		\centering
		\begin{subfigure}{.5\textwidth}
			\centering
			\includegraphics[width=0.95\linewidth]{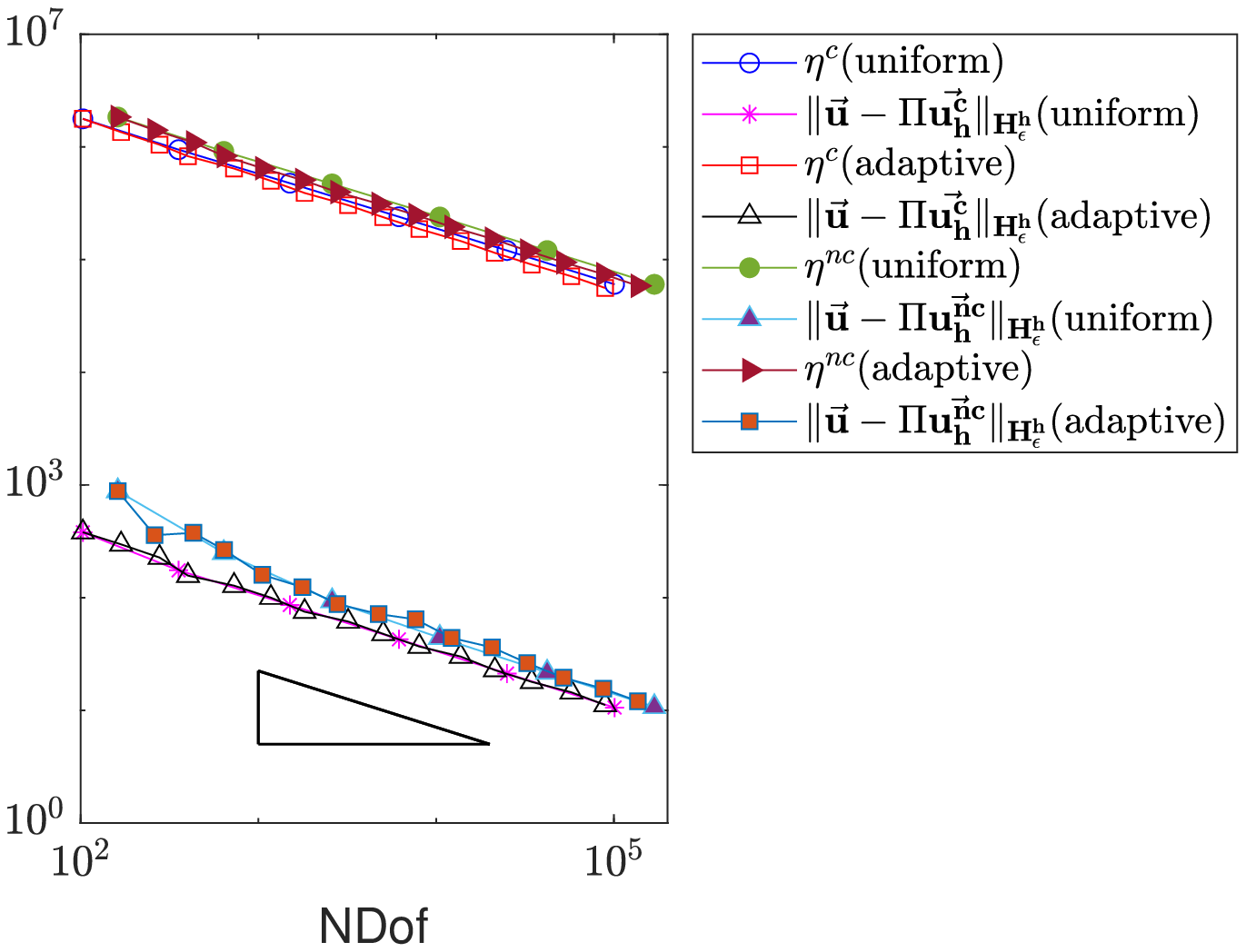}
		\end{subfigure}%
		\begin{subfigure}{.5\textwidth}
			\centering
			\includegraphics[width=0.95\linewidth]{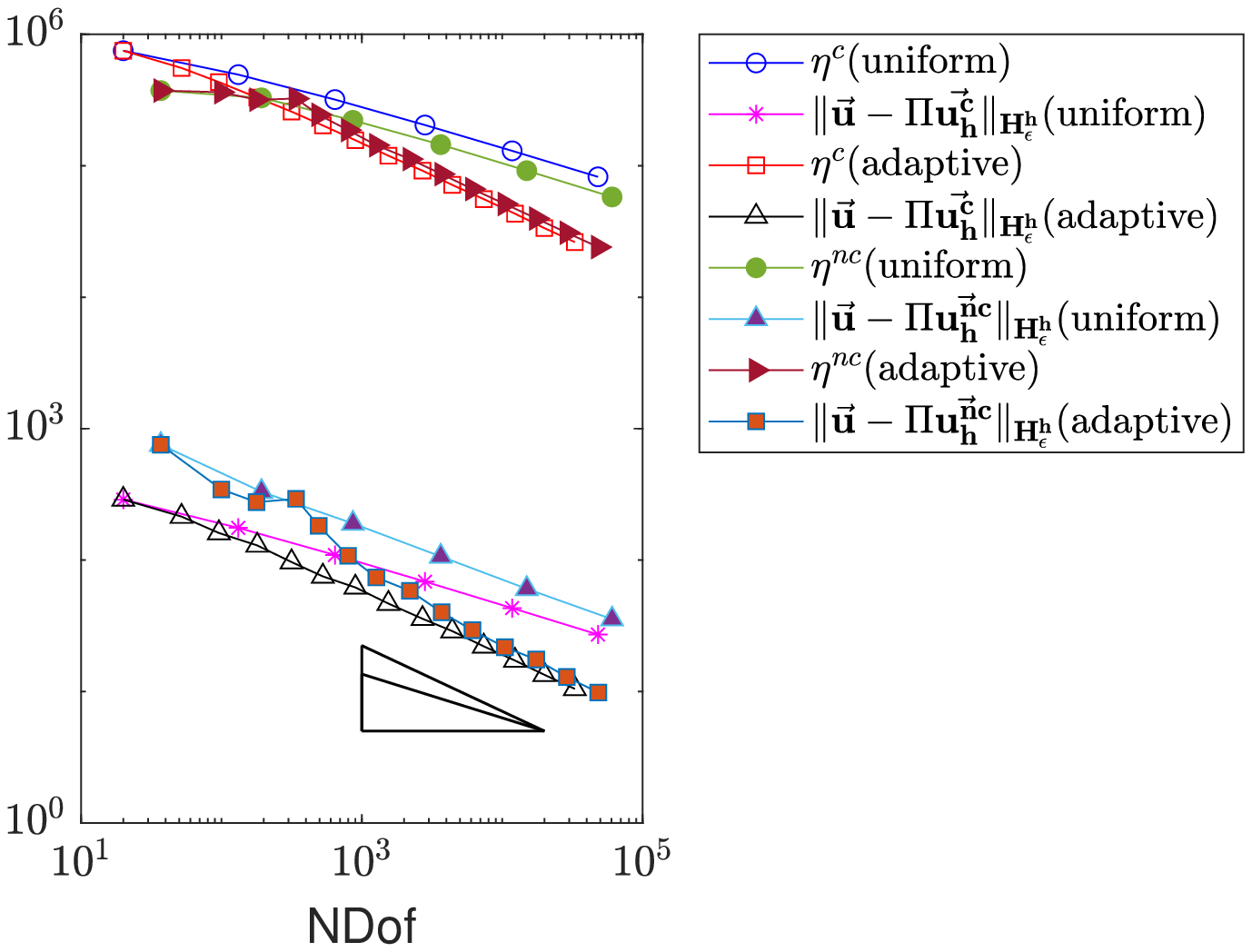}
		\end{subfigure}
		\caption{Left (resp. right) panel displays NDof vs error in energy norm and estimator in both uniform and adaptive refinements and both conforming and for nonconforming VEM in Example~1 (resp. Example~2). The superscript $c$ (resp. $nc$) denotes the respect term in the conforming (resp. nonconforming) case. The triangle in left (resp. right) panel represents slope $0.5$ (resp. $0.33$ and $0.5$).}
		\label{fig:scale}
	\end{figure}
\bibliographystyle{amsplain}
\bibliography{kmr-bib}
\end{document}